\documentclass[11pt]{article}

\usepackage{microtype}
\usepackage{graphicx}
\usepackage{subfigure}
\usepackage{booktabs} 

\usepackage{hyperref}
\usepackage[table]{xcolor}         

\newcommand{\declarecolor}[2]{\definecolor{#1}{RGB}{#2}\expandafter\newcommand\csname #1\endcsname[1]{\textcolor{#1}{##1}}}
\declarecolor{White}{255, 255, 255}
\declarecolor{Black}{0, 0, 0}
\declarecolor{Maroon}{128, 0, 0}
\declarecolor{Coral}{255, 127, 80}
\declarecolor{Red}{182, 21, 21}
\declarecolor{LimeGreen}{50, 205, 50}
\declarecolor{DarkGreen}{0, 80, 0}
\declarecolor{Purple}{146, 42, 158}
\declarecolor{Navy}{0, 0, 128}
\declarecolor{LightBlue}{84, 101, 202}
\definecolor{mydarkblue}{rgb}{0,0.08,0.45}
\hypersetup{ %
    pdftitle={},
    pdfkeywords={},
    pdfborder=0 0 0,
    pdfpagemode=UseNone,
    colorlinks=true,
    linkcolor=Navy,
    citecolor=DarkGreen,
    filecolor=Purple,
    urlcolor=Purple,
}

\usepackage[longnamesfirst]{natbib}
\usepackage[italicdiff]{physics}
\usepackage{bm}
\usepackage{bbm}
\usepackage{xspace}
\usepackage{complexity}
\usepackage{mleftright}
\usepackage{xcolor}
\mleftright

\newclass{\SigmaP}{\Sigma_2^P}
\newclass{\CLS}{CLS}

\newcommand{\mat}[1]{\mathbf{#1}}

\def\va{{\bm{a}}}

\def\vc{{\bm{c}}}

\def\vu{{\bm{u}}}
\def\vv{{\bm{v}}}

\def\vx{{\bm{x}}}
\def\vy{{\bm{y}}}
\def\vz{{\bm{z}}}

\renewcommand\vec\bm

\newcommand{\tilcX}{\tilde{\cX}}
\newcommand{\tilF}{\tilde{F}}

\newcommand{\VI}{\textsc{VI}}

\newcommand{\diffeps}{\tilde{\epsilon}}

\newcommand{\svigap}{\mathsf{SVIGap}}

\newcommand{\vatil}{\tilde{\va}}

\newcommand{\optellipsoid}{\texttt{ExtraGradientEllipsoid}}

\newcommand{\sviorstrictevi}{\texttt{SVIorStrictEVI}}

\newcommand{\opt}{\mathsf{OPT}}

\def\cA{{\mathcal{A}}}
\def\cB{{\mathcal{B}}}

\def\cE{{\mathcal{E}}}

\def\cK{{\mathcal{K}}}

\def\cP{{\mathcal{P}}}

\def\cX{{\mathcal{X}}}

\def\cZ{{\mathcal{Z}}}

\newcommand{\delimit}[3]{\newcommand{#1}[1]{\left#2##1\right#3}}
\delimit \ceil \lceil \rceil
\delimit \floor \lfloor \rfloor

\DeclareMathOperator*{\argmin}{argmin}
\DeclareMathOperator*{\argmax}{argmax}
\let\E\relax
\DeclareMathOperator*{\E}{\mathbb E}

\DeclareMathOperator{\size}{\mathsf{size}}

\newcommand{\reg}{\overline{\mathsf{Reg}}}
\newcommand{\sw}{\mathsf{SW}}

\newcommand{\sep}{\mathtt{SEP}}

\newcommand{\vol}{\op{vol}}

\newcommand{\proj}{\Pi}

\renewcommand{\R}{\mathbb R}
\newcommand{\N}{\mathbb N}

\newcommand{\Q}{\mathbb Q}

\let\op\operatorname
\let\eps\epsilon
\let\grad\nabla
\let\ip\ev

\newcommand{\defeq}{:=}

\newcommand{\ie}{{\em i.e.}\xspace}
\newcommand{\eg}{{\em e.g.}\xspace}

\newcommand{\sgn}{\op{sgn}}

\usepackage{multirow}

\usepackage[ruled,linesnumbered,noend]{algorithm2e}

\SetKwInput{KwInput}{Input}
\SetKwInput{KwOutput}{Output}

\usepackage{amsmath}
\usepackage{amssymb}
\usepackage{mathtools}
\usepackage{amsthm}

\usepackage{tikz}
\usepackage{pgfplots}
\usetikzlibrary{arrows.meta, calc}

\usepackage{thm-restate}

\usepackage{fullpage}

\usepackage[capitalize,noabbrev,nameinlink]{cleveref}

\theoremstyle{plain}
\newtheorem{theorem}{Theorem}[section]
\newtheorem{proposition}[theorem]{Proposition}
\newtheorem{lemma}[theorem]{Lemma}
\newtheorem{claim}[theorem]{Claim}
\newtheorem{corollary}[theorem]{Corollary}
\theoremstyle{definition}
\newtheorem{definition}[theorem]{Definition}
\newtheorem{assumption}[theorem]{Assumption}
\newtheorem{observation}[theorem]{Observation}

\theoremstyle{remark}
\newtheorem{remark}[theorem]{Remark}

\usepackage{nicefrac}
\usepackage{enumitem}
\setlist[enumerate]{itemsep=1mm,parsep=0mm,topsep=.5mm}
\setlist[itemize]{itemsep=1mm,parsep=0mm,topsep=.5mm}

\usepackage{tikz}

\usepackage{MnSymbol}
\renewcommand\vec\bm
\renewcommand\grad\nabla

\usepackage{authblk}

\title{A Polynomial-Time Algorithm for Variational Inequalities\\ under the Minty Condition}

\author[1]{Ioannis Anagnostides}
\author[2]{Gabriele Farina}
\author[1,3]{Tuomas Sandholm}
\author[2]{Brian Hu Zhang}

\affil[1]{Carnegie Mellon University}
\affil[2]{Massachusetts Institute of Technology}
\affil[3]{Additional affiliations: Strategy Robot, Inc., Strategic Machine, Inc., Optimized Markets, Inc.}
\affil[ ]{}
\affil[ ]{\texttt{\{ianagnos,sandholm\}}\texttt{@cs.cmu.edu}, \texttt{gfarina}\texttt{@mit.edu}, \texttt{zhangbh@csail.mit.edu}}

\begin{document}

\maketitle

\begin{abstract}
Solving \emph{(Stampacchia) variational inequalities (SVIs)} is a foundational problem at the heart of optimization, with a host of critical applications ranging from engineering to economics. However, this expressivity comes at the cost of 
computational hardness. As a result, most research has focused on carving out specific subclasses that elude those intractability barriers. A classical property that goes back to the 1960s is the \emph{Minty condition}, which postulates that the \emph{Minty} VI (MVI) problem---the weak dual of the SVI problem---admits a solution.

In this paper, we establish the first polynomial-time algorithm---that is, with complexity growing polynomially in the dimension $d$ and $\log(1/\epsilon)$---for solving $\epsilon$-SVIs for Lipschitz continuous mappings under the Minty condition. Prior approaches either incurred an exponentially worse dependence on $1/\epsilon$ (and other natural parameters of the problem) or made more restrictive assumptions, such as monotonicity. To do so, we introduce a new variant of the ellipsoid algorithm whereby separating hyperplanes are obtained after taking a descent step from the center of the ellipsoid. It succeeds even though the set of SVIs can be nonconvex and not fully dimensional. Moreover, when our algorithm is applied to an instance with no MVI solution and fails to identify an SVI solution, it produces a succinct \emph{certificate of MVI infeasibility}. We also show that deciding whether the Minty condition holds is $\mathsf{coNP}$-complete, thereby establishing that the disjunction of those two problems---computing $\epsilon$-SVIs and ascertaining MVI infeasibility---is polynomial-time solvable even though each problem is individually intractable.

We provide several extensions and new applications of our main results. Most notably, we obtain the first polynomial-time algorithms for i) globally minimizing a (potentially nonsmooth) \emph{quasar-convex} function, and ii) computing Nash equilibria in multi-player \emph{harmonic games}. Finally, in two-player general-sum concave games, we give the first polynomial-time algorithm that outputs \emph{either} a Nash equilibrium \emph{or} a \emph{strict} coarse correlated equilibrium.
\end{abstract}

\pagenumbering{gobble}

\clearpage

\tableofcontents

\clearpage

\pagenumbering{arabic}

\section{Introduction}

\emph{Variational inequalities (VIs)} are a mainstay framework at the heart of optimization that unifies a host of foundational problems in diverse areas ranging from engineering to economics~\citep{Kinderlehrer00:Introduction,Facchinei03:Finite,Bensoussan11:Applications}. The basic problem underpinning VIs---when restricted to Euclidean spaces---can be formulated as follows.

\begin{definition}[SVIs]
    \label{def:VI}
    Let $\cX$ be a convex and compact subset of $\R^d$ and let $F : \cX \to \R^d$ be a (single-valued) mapping. The $\epsilon$-approximate \emph{Stampacchia variational inequality (SVI)} problem\footnote{It is common to refer to SVIs as simply VIs, but we adopt the former nomenclature here to disambiguate with the \emph{Minty VI} problem, introduced in~\Cref{def:MVI}.} asks for a point $\vx \in \cX$ such that
    \begin{equation}
    \label{eq:SVI}
    \langle F(\vx), \vx' - \vx \rangle \geq - \epsilon \quad \forall \vx' \in \cX.
\end{equation}
\end{definition}

Assuming $F$ is continuous, an exact SVI solution (that is, with $\epsilon = 0$) always exists~\citep{Facchinei03:Finite}---by Brouwer's fixed-point theorem applied on the function $\vx \mapsto \proj_{\cX}(\vx - \eta F(\vx))$, where $\proj_\cX$ is the (Euclidean) projection mapping. For computational purposes, we have introduced a slackness $\epsilon > 0$ in the right-hand side of~\eqref{eq:SVI}; without this relaxation, the only SVI solution can be irrational. It is assumed that $\epsilon$ is given (in binary) as part of the input, and the goal is to design algorithms with complexity growing polynomially in $\log(1/\epsilon)$ and the dimension $d$.

A canonical example of~\Cref{def:VI} is the problem of computing fixed points of gradient descent in constrained optimization. In particular, for a differentiable function $f$, the $\epsilon$-SVI problem corresponding to $F \defeq \nabla f$ can be expressed as $\langle \nabla f(\vx), \vx' - \vx \rangle \geq - \epsilon$ for all $\vx' \in \cX$, which captures precisely the first-order optimality conditions~\citep{Boyd04:Convex}. Another standard example of~\Cref{def:VI} is the problem of computing (approximate) \emph{Nash equilibria} in multi-player games~\citep{Nash50:Equilibrium}.\looseness-1

Unfortunately, by virtue of encompassing (approximate) Nash equilibria, solving general SVIs is computationally intractable---namely, $\PPAD$-hard~\citep{Papadimitriou94:On}---even when $F$ is linear and $\epsilon$ is an absolute constant~\citep{Rubinstein15:Inapproximability}; recent work by~\citet{Bernasconi24:Role,Kapron24:Computational} characterizes the exact complexity of SVIs, as well as generalizations thereof. In fact, even in the special case discussed above wherein $F \defeq \nabla f$, computing $\epsilon$-SVI solutions is also hard---under certain well-believed complexity assumptions---when $\epsilon$ is exponentially small in the dimension~\citep{Fearnley23:Complexity}.

In light of the intractability of solving general SVIs, most research has restricted its attention to more structured classes of problems. Following a long and burgeoning line of work that goes back to the 1960s, we operate under the \emph{Minty condition} (introduced below as~\Cref{ass:Minty}); it is based on a variant of SVIs (\Cref{def:VI}) tracing back to~\citet{Minty67:Generalization}, known as the \emph{Minty} VI problem.

\begin{definition}[MVIs]
    \label{def:MVI}
    Let $\cX$ be a convex and compact subset of $\R^d$ and let $F : \cX \to \R^d$ be a mapping. The \emph{Minty variational inequality (MVI)} problem asks for a point $\vx \in \cX$ such that
\begin{equation}
    \label{eq:MVI}
    \langle F(\vx'), \vx' - \vx \rangle \geq 0 \quad \forall \vx' \in \cX.
\end{equation}
\end{definition}

Unlike SVIs, an MVI solution may not exist even when $F$ is continuous---indeed, the Minty condition is precisely the requirement that an MVI solution exists:

\begin{assumption}[Minty condition]
    \label{ass:Minty}
    A variational inequality problem $\VI(\cX, F)$ satisfies the {\em Minty condition} if there exists $\vx \in \cX$ that satisfies~\eqref{eq:MVI}.
\end{assumption}
Assuming continuity, \Cref{def:MVI} refines~\Cref{def:VI} in the following sense. 

\begin{lemma}[Minty's lemma]
    \label{lemma:Minty}
    If $F$ is continuous and $\cX$ is convex and compact, then any MVI solution is also an SVI solution.
\end{lemma}

On the other hand, an SVI solution need not be an MVI solution---otherwise~\Cref{ass:Minty} would always hold under a continuous mapping. One well-known special case in which the set of MVI solutions does coincide with the set of SVI solutions is when $F$ is \emph{monotone}; that is, when $\langle F(\vx) - F(\vx'), \vx - \vx' \rangle \geq 0$ for all $\vx, \vx' \in \cX$. This holds more generally when $F$ is \emph{pseudomonotone}, which means that $\langle F(\vx'), \vx - \vx' \rangle \geq 0 \implies \langle F(\vx), \vx - \vx' \rangle \geq 0$ for all $\vx, \vx' \in \cX$. Importantly, the Minty condition is more permissive than monotonicity, encompassing a broader class of problems; indeed, it captures a host of nonconvex optimization problems (\emph{cf.}~\Cref{sec:Mintyexamples}).

By now, it is well-known that under the Minty condition, there are algorithms with complexity scaling polynomially in $1/\epsilon$ (and all other natural parameters of the problem) for computing an $\epsilon$-SVI solution. This goes back to the early, pioneering work of~\citet{Sibony70:Methodes}, \citet{Martinet70:Breve}, and~\citet{Korpelevich76:Extragradient}; in particular, the latter showed that the extra-gradient method converges in all monotone VIs---this stands in stark contrast to the usual gradient descent algorithm, which can cycle even in two-player zero-sum games~\citep{Mertikopoulos18:Cycles}, which are monotone. Motivated by many applications in machine learning and reinforcement learning, there has been renewed interest in the Minty condition recently (\emph{e.g.},~\citealp{Burachik2020:Projection,Lei21:Extragradient,Lin24:Perseus,Song20:Optimistic,Ye22:Infeasible,Mertikopoulos19:Learning,Mertikopoulos18:Optimistic,Daskalakis20:Independent,Patris24:Learning,Cai24:Accelerated}). While some of these bounds do not have an explicit dependence on the dimension $d$, they become exponential when the precision $\epsilon$ is exponentially small in terms of $d$. On the other end of the spectrum, all existing algorithms with $\poly(d, \log(1/\epsilon))$ complexity make restrictive assumptions that are significantly stronger than the Minty condition, akin to monotonicity or some type of an ``error bound'' (for example, we refer to~\citealp{Song20:Optimistic,Ye22:Infeasible,Luthi85:Solution,Magnanti95:Unifying}, discussed further in~\Cref{sec:related}).\looseness-1

\subsection{Our results}

\begin{table}[ht]
\centering
\caption{Summary of main results. All algorithms (upper bounds) have runtimes that depend polynomially on both $d$ and $\log(1/\eps)$, and, to our knowledge, are the first algorithms for their respective settings with this dependence. $^\dagger$: Unlike our other positive results, our result for quasar-convex optimization holds even when $F = \grad f$ is not Lipschitz continuous.}
\vspace{.2cm}
\label{tab:bounds}
\begin{tabular}{p{3cm}p{9.8cm}l}
\toprule
& \textbf{Result} & \textbf{Reference} \\
\midrule
\multirow{5}{*}{Upper bounds} & \textbullet\,~$\epsilon$-SVIs under the Minty condition & \Cref{theorem:mainMinty-informal} \\
 &  \textbullet\,~$\epsilon$-SVIs \emph{or} $\Omega_\epsilon(\epsilon^2)$-strict EVIs & \Cref{theorem:maineither-informal} \\
& \textbullet\,~$\epsilon$-global optimization under quasar-convexity$^\dagger$ & \Cref{cor:improved-quasarconvex} \\
&  \textbullet\,~$\epsilon$-Nash equilibria in harmonic games & \Cref{cor:harmonic} \\
 & \textbullet\,~$\eps$-Nash \emph{or} $\Omega_\eps(\eps^6)$-strict CCEs in two-player concave games & \cref{theorem:informal-twoplayerstrict}\\
\midrule
\multirow{2}{*}{Lower bounds} & \textbullet\,~Deciding MVI feasibility (\emph{i.e.},~Minty condition) & \Cref{theorem:hardmintyconp-informal} \\
 & \textbullet\,~Solving MVIs under the Minty condition & \Cref{theorem:Mintyexp-informal} \\
\bottomrule
\end{tabular}
\end{table}

We establish the first polynomial-time algorithm for solving $\epsilon$-SVIs under the Minty condition. In what follows, we assume that the constraint set $\cX$ is accessed through a (weak) separation oracle (\Cref{def:weaksep}) and $\cB_1(\vec{0}) \subseteq \cX \subseteq \cB_R(\vec{0})$ for some $R \leq \poly(d)$; the latter can be met by bringing $\cX$ into isotropic position, a transformation that does not affect our main result (as formalized in~\Cref{sec:isotropic}). With regard to the mapping $F$, we assume that it can be evaluated in polynomial time, it is $L$-Lipschitz continuous, and $\|F(\vx)\|$ is bounded by $B > 0$ (\Cref{ass:polyeval}). We are now ready to state our main result.

\begin{theorem}[Precise version in~\Cref{theorem:mainMinty-precise}]
    \label{theorem:mainMinty-informal}
    Under the Minty condition (\Cref{ass:Minty}), there is an algorithm that runs in $\poly(d, \log(B/\epsilon), \log L)$ time and returns an $\epsilon$-SVI solution.
\end{theorem}

Compared to all previous results cited earlier under the Minty condition, this result improves exponentially the dependence on $1/\epsilon$ and $L$. We clarify that, although the underlying algorithm posits the Minty condition, its execution does not hinge on knowing an MVI solution. Indeed, a peculiar feature of~\Cref{theorem:mainMinty-informal} is that while its main precondition concerns MVIs, the output itself is an approximate SVI solution; the reason for this discrepancy will become clear shortly. In the special case where $F$ is monotone, \Cref{theorem:mainMinty-informal} implies a polynomial-time algorithm for finding $\epsilon$-MVI solutions---since under monotonicity the set of MVIs coincides with the set of SVIs.

Before we present some more results that build on~\Cref{theorem:mainMinty-informal} (gathered in~\Cref{tab:bounds}), we dive into our technical approach. 



\subsubsection{Technical approach: proof of~\texorpdfstring{\Cref{theorem:mainMinty-informal}}{Theorem~\ref{theorem:mainMinty-informal}}}

The proof of~\Cref{theorem:mainMinty-informal} relies on the usual (central-cut) ellipsoid algorithm, but with certain unusual twists. 

\paragraph{Challenge 1: lack of convexity} The first immediate conundrum lies in the fact that the set of SVI solutions is generally not convex even when the Minty condition holds. By contrast, the set of MVI solutions is convex, thereby being a better candidate on which to execute ellipsoid. But this approach also runs into an apparent difficulty: while a point $\vx \in \cX$ can be confirmed to be an $\epsilon$-SVI solution by invoking a (linear) optimization oracle, this is not so for MVIs. Indeed, as we show in this paper (\Cref{sec:introlowerbounds}), ascertaining MVI membership is \coNP-complete.

In summary, the set of SVI solutions admits an efficient membership oracle, while the set of MVI solutions is convex. A natural approach now presents itself: execute the ellipsoid with respect to the set of MVIs, but only verify SVI membership. To do so, the first key idea is this: if a point $\vx \in \cX$ is \emph{not} an $\epsilon$-SVI solution, then $F(\vx)$ yields a hyperplane separating $\vx$ from the set of MVIs. As a result, this allows us to run the ellipsoid algorithm with respect to the convex set of MVI solutions, with the peculiarity that we only test for SVI membership during the execution of the algorithm. So long as the algorithm fails to identify an approximate SVI solution, the volume of the ellipsoid shrinks geometrically.

\paragraph{Challenge 2: lack of full dimensionality} This now brings us to the next key challenge: the set of MVI solutions is generally not fully dimensional. It is therefore unclear whether the volume of the ellipsoid can be used as a yardstick to measure the progress of the algorithm. There is a standard approach for addressing this issue, at least for rational polyhedral sets (for example,~\citealp[Chapter 6]{Grotschel12:Geometric}): restrict the execution of the ellipsoid to suitable subspaces whenever one of the ellipsoid's axes gets too small. However, this standard approach falls short in our problem; we provide a concrete numerical example in~\Cref{sec:noopt} illustrating that the usual ellipsoid algorithm fails to identify $\epsilon$-SVI solutions.

To address this problem, we introduce a new algorithmic idea. Namely, we show that we can produce, in polynomial time, what we refer to as a \emph{strict} separating hyperplane; this key ingredient underpinning~\Cref{theorem:mainMinty-informal} is summarized below (the precise version is~\Cref{lemma:strict1}).
(Technically, we need to allow the given point $\va$ in~\Cref{lemma:sviormvi} to be approximately in $\cX$ since we are dealing with general convex sets---in accordance with~\Cref{lemma:extendedsemi}---but we do not dwell on this issue in the introduction.)

\begin{lemma}[SVI membership or MVI \emph{strict} separation]
    \label{lemma:sviormvi}
    Given a point $\va \in \cX \cap \Q^d$ and $\epsilon \in \Q_{>0}$, there is a polynomial-time algorithm that either
    \begin{enumerate}
        \item  ascertains that $\vec{a}$ is an $\epsilon$-SVI solution or
        \item returns $\vec{c} \in \Q^d$, with $\| \vec{c} \|_\infty = 1$, such that $\langle \vec{c}, \vx \rangle \leq \langle \vec{c}, \va \rangle - \gamma$ for any point $\vx \in \cX$ that satisfies the Minty VI~\eqref{eq:MVI}, where $\gamma = \epsilon^2 \cdot \poly(R^{-1}, L^{-1}, B^{-1})$.\label{item:strictsep}
    \end{enumerate}
\end{lemma}

\begin{figure}[thp]
    \centering
    \tikzset{every picture/.style={line width=0.75pt}} 

\begin{tikzpicture}[x=0.75pt,y=0.75pt,yscale=-1,xscale=1]

\begin{scope}[scale=.9]

\draw [color={rgb, 255:red, 208; green, 2; blue, 27 }  ,draw opacity=1 ]   (97,172) .. controls (106.21,177.48) and (108.31,188.14) .. (108.98,204.51) ;
\draw [shift={(109.05,206.32)}, rotate = 268.03] [color={rgb, 255:red, 208; green, 2; blue, 27 }  ,draw opacity=1 ][line width=0.75]    (10.93,-3.29) .. controls (6.95,-1.4) and (3.31,-0.3) .. (0,0) .. controls (3.31,0.3) and (6.95,1.4) .. (10.93,3.29)   ;
\draw [color={rgb, 255:red, 74; green, 144; blue, 226 }  ,draw opacity=1 ] [dash pattern={on 4.5pt off 4.5pt}]  (38.5,293.1) -- (258,73.6) ;
\draw [color={rgb, 255:red, 74; green, 144; blue, 226 }  ,draw opacity=1 ][fill={rgb, 255:red, 74; green, 144; blue, 226 }  ,fill opacity=1 ]   (157.58,173.78) -- (182.13,198.33) ;
\draw [shift={(184.25,200.45)}, rotate = 225] [fill={rgb, 255:red, 74; green, 144; blue, 226 }  ,fill opacity=1 ][line width=0.08]  [draw opacity=0] (10.72,-5.15) -- (0,0) -- (10.72,5.15) -- (7.12,0) -- cycle    ;
\draw  [dash pattern={on 4.5pt off 4.5pt}]  (46.5,228) -- (287.5,191) ;
\draw  [fill={rgb, 255:red, 0; green, 0; blue, 0 }  ,fill opacity=0.05 ][line width=1.5]  (57.8,104.06) .. controls (93.97,67.89) and (170.65,85.91) .. (229.06,144.32) .. controls (287.46,202.72) and (305.48,279.4) .. (269.31,315.57) .. controls (233.13,351.75) and (156.46,333.73) .. (98.05,275.32) .. controls (39.64,216.91) and (21.62,140.24) .. (57.8,104.06) -- cycle ;
\node[text={rgb, 255:red, 208; green, 2; blue, 27}] at (109.05,213.32) {{\rotatebox{45}{\small$\blacksquare$}}};
\draw [color={rgb, 255:red, 0; green, 0; blue, 0 }  ,draw opacity=1 ][fill={rgb, 255:red, 74; green, 144; blue, 226 }  ,fill opacity=1 ]   (163.55,209.82) -- (167.54,235.44) ;
\draw [shift={(168,238.4)}, rotate = 261.16] [fill={rgb, 255:red, 0; green, 0; blue, 0 }  ,fill opacity=1 ][line width=0.08]  [draw opacity=0] (10.72,-5.15) -- (0,0) -- (10.72,5.15) -- (7.12,0) -- cycle    ;
\draw  [draw opacity=0][fill={rgb, 255:red, 0; green, 0; blue, 0 }  ,fill opacity=1 ] (160.55,209.82) .. controls (160.55,208.16) and (161.9,206.82) .. (163.55,206.82) .. controls (165.21,206.82) and (166.55,208.16) .. (166.55,209.82) .. controls (166.55,211.47) and (165.21,212.82) .. (163.55,212.82) .. controls (161.9,212.82) and (160.55,211.47) .. (160.55,209.82) -- cycle ;
\draw  [draw opacity=0][fill={rgb, 255:red, 74; green, 144; blue, 226 }  ,fill opacity=1 ] (154.58,173.78) .. controls (154.58,172.12) and (155.92,170.78) .. (157.58,170.78) .. controls (159.24,170.78) and (160.58,172.12) .. (160.58,173.78) .. controls (160.58,175.43) and (159.24,176.78) .. (157.58,176.78) .. controls (155.92,176.78) and (154.58,175.43) .. (154.58,173.78) -- cycle ;

\draw (57.5,134.5) node [anchor=north west][inner sep=0.75pt]  [color={rgb, 255:red, 208; green, 2; blue, 27 }  ,opacity=1 ] [align=left] {\begin{minipage}[lt]{37.88pt}\setlength\topsep{0pt}
\begin{center}
MVI\\solution
\end{center}

\end{minipage}};
\draw (143.05,188.82) node [anchor=north west][inner sep=0.75pt]    {$\boldsymbol{a}^{( t)}$};
\draw (162.35,235.1) node [anchor=north west][inner sep=0.75pt]  [color={rgb, 255:red, 0; green, 0; blue, 0 }  ,opacity=1 ]  {$F\left(\boldsymbol{a}^{( t)}\right)$};
\draw (175.35,166.6) node [anchor=north west][inner sep=0.75pt]  [color={rgb, 255:red, 74; green, 144; blue, 226 }  ,opacity=1 ]  {$F\left(\boldsymbol{\tilde{a}}^{( t)}\right)$};
\draw (94,297) node [anchor=north west][inner sep=0.75pt]    {$\mathcal{E}^{( t)}$};
\draw (142.05,147.82) node [anchor=north west][inner sep=0.75pt]  [color={rgb, 255:red, 74; green, 144; blue, 226 }  ,opacity=1 ]  {$\tilde{\boldsymbol{a}}^{( t)}$};

\end{scope}


\begin{scope}[xshift=8.0cm,yshift=4.0cm]
    \path (0,0) edge[black,->,bend right] (50,0);
\end{scope}

\begin{scope}[scale=.9,xshift=2cm]

\draw [color={rgb, 255:red, 74; green, 144; blue, 226 }  ,draw opacity=1 ][fill={rgb, 255:red, 74; green, 144; blue, 226 }  ,fill opacity=1 ]   (511.55,213.82) -- (546.84,247.92) ;
\draw [shift={(549,250)}, rotate = 224.02] [fill={rgb, 255:red, 74; green, 144; blue, 226 }  ,fill opacity=1 ][line width=0.08]  [draw opacity=0] (10.72,-5.15) -- (0,0) -- (10.72,5.15) -- (7.12,0) -- cycle    ;
\draw [color={rgb, 255:red, 74; green, 144; blue, 226 }  ,draw opacity=1 ]   (401.8,323.57) -- (621.3,104.07) ;
\draw [color={rgb, 255:red, 74; green, 144; blue, 226 }  ,draw opacity=1 ] [dash pattern={on 4.5pt off 4.5pt}]  (386.5,297.1) -- (606,77.6) ;
\draw  [dash pattern={on 0.84pt off 2.51pt}] (405.8,108.06) .. controls (441.97,71.89) and (518.65,89.91) .. (577.06,148.32) .. controls (635.46,206.72) and (653.48,283.4) .. (617.31,319.57) .. controls (581.13,355.75) and (504.46,337.73) .. (446.05,279.32) .. controls (387.64,220.91) and (369.62,144.24) .. (405.8,108.06) -- cycle ;
\node[text={rgb, 255:red, 208; green, 2; blue, 27}] at (457.05,217.32) {\rotatebox{45}{\small$\blacksquare$}};

\draw  [draw opacity=0][fill={rgb, 255:red, 0; green, 0; blue, 0 }  ,fill opacity=1 ] (508.55,213.82) .. controls (508.55,212.16) and (509.9,210.82) .. (511.55,210.82) .. controls (513.21,210.82) and (514.55,212.16) .. (514.55,213.82) .. controls (514.55,215.47) and (513.21,216.82) .. (511.55,216.82) .. controls (509.9,216.82) and (508.55,215.47) .. (508.55,213.82) -- cycle ;
\draw  [fill={rgb, 255:red, 0; green, 0; blue, 0 }  ,fill opacity=0.05 ][line width=1.5]  (406.28,107.58) .. controls (451.43,62.43) and (517.48,55.28) .. (553.82,91.62) .. controls (590.15,127.95) and (583,194) .. (537.85,239.15) .. controls (492.7,284.3) and (426.65,291.45) .. (390.32,255.12) .. controls (353.98,218.78) and (361.13,152.73) .. (406.28,107.58) -- cycle ;
\draw [color={rgb, 255:red, 74; green, 144; blue, 226 }  ,draw opacity=1 ]   (399.12,288.12) -- (410,299) -- (415.88,304.88) ;
\draw [shift={(418,307)}, rotate = 225] [fill={rgb, 255:red, 74; green, 144; blue, 226 }  ,fill opacity=1 ][line width=0.08]  [draw opacity=0] (8.93,-4.29) -- (0,0) -- (8.93,4.29) -- cycle    ;
\draw [shift={(397,286)}, rotate = 45] [fill={rgb, 255:red, 74; green, 144; blue, 226 }  ,fill opacity=1 ][line width=0.08]  [draw opacity=0] (8.93,-4.29) -- (0,0) -- (8.93,4.29) -- cycle    ;

\draw (350,255) node [anchor=north west][inner sep=0.75pt]    {$\mathcal{E}^{( t+1)}$};
\draw (495.55,189.82) node [anchor=north west][inner sep=0.75pt]    {$\boldsymbol{a}^{( t)}$};
\draw (393,299) node [anchor=north west][inner sep=0.75pt]  [color={rgb, 255:red, 74; green, 144; blue, 226 }  ,opacity=1 ]  {$\gamma $};
\draw (551,232.6) node [anchor=north west][inner sep=0.75pt]  [color={rgb, 255:red, 74; green, 144; blue, 226 }  ,opacity=1 ]  {$F\left(\boldsymbol{\tilde{a}}^{( t)}\right)$};

\end{scope}
\end{tikzpicture}
    \caption{One step of our ellipsoid algorithm---when the current ellipsoid center $\va^{(t)}$ is not already an $\epsilon$-SVI solution. While $F(\va^{(t)})$ separates $\va^{(t)}$ from the set of MVI solutions (in this case a single point), $F(\vatil^{(t)})$ yields a  $\gamma$-\emph{strict} separating hyperplane, which turns out to be crucial; see~\Cref{sec:noopt}.}
    \label{fig:ellipsoids}
\end{figure}

To obtain a hyperplane that \emph{strictly} separates $\va \in \cX$ from the set of MVI solutions (in the sense of~\Cref{item:strictsep}), assuming that $\va$ is not an $\epsilon$-SVI solution, we perform a descent step starting from $\va$. Since $\va$ is not an $\epsilon$-SVI solution, it can be shown that the resulting point, say $\vatil \in \cX$, is such that $F(\vatil)$ strictly separates $\va$ from the set of MVIs (\Cref{lemma:strict1}); this is illustrated in~\Cref{fig:ellipsoids}. 

Interestingly, this simple algorithmic maneuver is, at least conceptually, similar to the extra-gradient method~\citep{Korpelevich76:Extragradient} (and variants thereof), which is known to converge---albeit at an inferior rate that grows polynomially in $1/\epsilon$---under the Minty condition. Accordingly, we call the overall algorithm $\optellipsoid$ (\Cref{algo:optimisticellipsoid}); it is an incarnation of the central-cut ellipsoid endowed with additional descent steps. We stress again that this step cannot be avoided: the usual ellipsoid algorithm without the additional step fails (\Cref{sec:noopt}) due to its inability to generate strict separating hyperplanes. This phenomenon mirrors the behavior of first-order methods, where regular gradient descent fails to converge to an SVI solution---even in monotone problems---whereas the extra-gradient method succeeds~\citep{Korpelevich76:Extragradient}. As such, we find that there is an intriguing analogy between the behavior we uncover for ellipsoid-based algorithms and what has been known for decades pertaining to gradient-based algorithms.

To finish the proof of~\Cref{theorem:mainMinty-informal}, we observe that~\Cref{lemma:sviormvi}---and in particular the sequence of strict separating hyperplanes produced during the execution of the ellipsoid---implies that the volume of the ellipsoid cannot shrink too much (\Cref{lemma:largevol}). Indeed, since the ellipsoid always contains the set of MVI solutions, and the separating hyperplanes are {\em strict}, every MVI solution is in some sense {\em far} from the boundary of the ellipsoid, which in turn implies that the ellipsoid must have nontrivial volume. In other words, the algorithm will necessarily terminate with an $\epsilon$-SVI solution, as promised---so long as an MVI solution exists. 

If no MVI solution exists, the above algorithm might fail, that is, the volume may end up shrinking too much. But in this case, as we shall see, a small adaptation to $\optellipsoid$ can produce a polynomial \emph{certificate of MVI infeasibility} (\emph{cf.}~\Cref{theorem:maineither-informal}).

\subsubsection{A certificate of MVI infeasibility}
\label{subsec:EVIs}

We now treat the general setting in which the Minty condition can be altogether violated. Our next result shows how to employ $\optellipsoid$ so as to produce a certificate of MVI infeasibility. But what does such a ``certificate'' look like?

To answer this, we rely on the recently introduced concept of \emph{expected VIs (EVIs)} (\Cref{def:EVIs}). This is a relaxation of~\Cref{def:VI} that only imposes the SVI constraint \emph{in expectation} for points drawn from a distribution. For readers familiar with equilibrium concepts in game theory, it is instructive to have in mind that EVIs are to SVIs what \emph{(average) coarse correlated equilibria (ACCEs)} (\Cref{def:acce}) are to Nash equilibria. The key point is that there is a strong duality between MVIs and EVIs (\Cref{prop:duality}); namely, the Minty condition holds if and only if no EVI solution with \emph{negative gap} exists---we refer to the latter object as a \emph{strict} EVI solution. In other words, the mere existence of a strict EVI exposes MVI infeasibility. This brings us to the following key refinement of~\Cref{theorem:mainMinty-informal}.

\begin{theorem}[SVI or strict EVI; precise version in~\Cref{theorem:maineither-precise}]
    \label{theorem:maineither-informal}
    There is an algorithm that runs in $\poly(d, \log(B/\epsilon), \log L)$ time and returns either 
    \begin{enumerate}
        \item an $\epsilon$-SVI solution or\label{item:eitherSVI}
        \item an $\Omega_\epsilon(\epsilon^2)$-strict EVI solution.\label{item:strictEVI}
    \end{enumerate}
\end{theorem}

This clearly strengthens~\Cref{theorem:mainMinty-informal}: under the Minty condition no strict EVIs exist, so \Cref{item:strictEVI} in~\Cref{theorem:maineither-informal} will never arise under~\Cref{ass:Minty}. A key reference point here is a result by~\citet{Anagnostides22:Last}, who provided an algorithm with a similar output guarantee, but with complexity scaling polynomially---rather than logarithmically---in $1/\epsilon$; as such, \Cref{theorem:maineither-informal} yields again an exponential improvement over existing results.

The proof of~\Cref{theorem:maineither-informal} is based on an application of duality between MVIs and EVIs. In particular, the minimax theorem implies that, once the volume of the ellipsoid becomes sufficiently small, there is a distribution supported on $\vatil^{(1)}, \dots, \vatil^{(T)}$ that is a $\gamma/2$-strict EVI, where $\gamma$ is the strictness parameter per~\Cref{lemma:sviormvi} and $\vatil^{(t)}$ is obtained after a descent step starting from the center of the ellispoid at the $t$th iteration (\Cref{fig:ellipsoids}); coupled with~\Cref{lemma:sviormvi}, this explains the $\Omega_\epsilon(\gamma) = \Omega_\epsilon(\epsilon^2)$ strictness in~\Cref{item:strictEVI}. We are thus left with the simple problem of optimizing the mixing weights of a distribution with a polynomial support. This algorithmic maneuver closely resembles the celebrated ``ellipsoid against hope'' algorithm of~\citet{Papadimitriou08:Computing} (\emph{cf.}~\citealp{Jiang11:Polynomial,Farina24:Polynomial,Daskalakis24:Efficient}), which is also based on running ellipsoid on an infeasible program.

A strict EVI, besides certifying MVI infeasibility, is an interesting object in its own right. It is, by definition, a solution concept with negative gap, thereby being particularly stable---any possible deviation is not just suboptimal, but significantly so. Indeed, its incarnation in the context of games has been already used to address the equilibrium selection problem, as we explain more in~\Cref{sec:related}. Furthermore, in certain applications, the EVI gap translates to a performance guarantee in terms of some underlying objective function; the \emph{smoothness} framework of~\citet{Roughgarden15:Intrinsic}---and its extension for general VI problems given in~\Cref{def:smoothVI}---is a prime example of this in the context of multi-player games. It turns out that an EVI with a negative gap yields a \emph{strict} improvement over the bound predicted by the smoothness framework.

But there is something more that is especially notable about~\Cref{theorem:maineither-informal}: each of the computational problems in~\Cref{item:eitherSVI,item:strictEVI} is computationally hard on its own, yet~\Cref{theorem:maineither-informal} shows that their disjunction is easy! In particular, computing $\epsilon$-SVI solutions is a well-known \PPAD-complete problem. With regard to strict EVIs, we provide a characterization of its complexity in this paper, establishing $\NP$-completeness (\Cref{theorem:hardmintyconp-informal}).\footnote{Strict EVIs are dual to MVI solutions;  \Cref{theorem:hardmintyconp-informal} shows that deciding MVI feasibility is \coNP-complete, so deciding strict EVI existence is \NP-complete.} To put this into context, we highlight that there has been interest in characterizing the complexity of the union of two problems, especially in the realm of $\TFNP$. A notable contribution here is the work of~\citet{Daskalakis11:Continuous} that examined the complexity of problems in $\PPAD \cap \PLS$, where $\PLS$ stands for ``polynomial local search''~\citep{Johnson88:How}. They observed that if a problem $\textsc{A}$ is $\PPAD$-complete and $\textsc{B}$ is $\PLS$-complete, then the problem that must return \emph{either} a solution to $\textsc{A}$ or to $\textsc{B}$ is $\PPAD \cap \PLS$-complete---that is, $\CLS$-complete~\citep{Fearnley23:Complexity}; unlike~\Cref{theorem:maineither-informal}, this observation by~\citet{Daskalakis11:Continuous} assumes that $\textsc{A}$ and $\textsc{B}$ are defined with respect to different instances.
In this context, \Cref{theorem:maineither-informal} provides an example in which the disjunction of two hard problems---one \PPAD-complete and the other \NP-complete---defined with respect to the \emph{same} instance is easy.

\paragraph{Comparison with earlier approaches for monotone VIs} The basic ingredient that underpins prior approaches for solving \emph{monotone} VIs is what we may call ``equilibrium collapse,'' which essentially means that an $\epsilon$-EVI solution induces an $O_\eps(\epsilon)$-SVI solution by taking the marginals. Earlier algorithms operate by running a cutting-plane algorithm, such as the ellipsoid, so as to identify an $\epsilon$-EVI solution. Now, if it so happens that an $\epsilon$-SVI solution is queried, the algorithm terminates; otherwise, an $\epsilon$-EVI solution induces an $O_\epsilon(\epsilon)$-SVI solution on account of the equilibrium collapse.\looseness-1

\citet{Rodomanov23:Subgradient} claim that such an equilibrium collapse holds ``for all known instances'' of problems adhering to the Minty condition, and \citet{Nemirovski10:Accuracy} claim that no methods are known for solving monotone VIs that do not amount to producing an $\eps$-EVI solution and using equilibrium collapse. However, we give an explicit counterexample in \Cref{sec:collapse} that shows that this equilibrium collapse does not hold assuming only the Minty condition. Thus, other techniques are required for Minty VIs. Indeed, the main result of our paper provides a method for solving VIs under the Minty condition (and hence also monotone VIs) without relying on equilibrium collapse.

\subsubsection{Implications and extensions}

Moving forward, we discuss several important consequences and extensions of our main results for optimization and game theory.

\paragraph{Quasar-convex optimization}

The first implication concerns optimizing \emph{quasar-convex} functions; this is a relaxation of convexity that has attracted significant interest recently (for example,~\citealp{Fu23:Accelerated,Hinder20:Near,Caramanis24:Optimizing,Hardt18:Gradient,Gower21:SGD,Danilova22:Recent,Wang23:Continuized}).\looseness-1

\begin{restatable}[Quasar-convexity]{definition}{quasarconvex}
    \label{def:quasarconvex}
    Let $\lambda \in (0, 1]$ and $\vx$ be a minimizer of a differentiable function $f : \cX \to \R$. We say that $f$ is \emph{$\lambda$-quasar-convex} with respect to $\vx$ if
    \begin{equation}
        \label{eq:quasarconvex}
        f(\vx) \geq f(\vx') + \frac{1}{\lambda} \langle \nabla f(\vx'), \vx - \vx' \rangle \quad \forall \vx' \in \cX.
    \end{equation}
\end{restatable}

(We elaborate more on how this relates to other properties in~\Cref{sec:optimization}.) Not only does $\VI(\cX, \nabla f)$ (under~\Cref{def:quasarconvex}) satisfy the Minty condition, but every approximate SVI solution is also an approximate global minimum of $f$ (\Cref{prop:quasarconvex}); combined with~\Cref{theorem:mainMinty-informal}, we obtain the first polynomial-time algorithm for globally minimizing smooth quasar-convex functions.

\begin{corollary}
    \label{cor:quasarconvex}
    There is a $\poly(d, \log(B/\epsilon), \log(1/\lambda), \log L)$-time algorithm that outputs a point $\vx \in \cX$ such that $f(\vx) \leq \min_{\vx' \in \cX} f(\vx') + \epsilon$ for any $\lambda$-quasar-convex function $f$.
\end{corollary}

This result can be generalized (\emph{cf.}~\Cref{prop:smoothVIs}) by considering a broader class of VI problems (\Cref{def:smoothVI}), beyond quasar-convex functions. Interestingly, this class of problems encompasses (a special case of)~\emph{smooth games}, famously introduced by~\citet{Roughgarden15:Intrinsic}; we explain this in more detail in~\Cref{sec:optimization}. 

Furthermore, as we shall now see, \Cref{cor:quasarconvex} can be significantly strengthened by relaxing the assumption that $\nabla f$ is continuous (\Cref{cor:improved-quasarconvex}). This follows a popular line of research in \emph{nonsmooth} optimization (\emph{e.g.},~\citealp{Zhang20:Complexity,Davis22:Gradient,Tian22:Finite,Jordan23:Deterministic}). In this context, we show that under quasar-convexity---and, more generally, its extension based on~\Cref{def:smoothVI}---it is possible to entirely eliminate the (logarithmic) dependence on $L$.\looseness-1

\begin{theorem}[Precise version in~\Cref{theorem:quasar-precise}]
    \label{cor:improved-quasarconvex}
    There is a $\poly(d, \log(B/\epsilon), \log(1/\lambda))$-time algorithm that outputs a point $\vx \in \cX$ such that $f(\vx) \leq \min_{\vx' \in \cX} f(\vx') + \epsilon$ for any $\lambda$-quasar-convex function $f$.
\end{theorem}

(Since we are considering additive approximations, a dependence on $B$ is necessary since one can always rescale $f$.) The key idea here is that quasar-convexity yields a strict separating hyperplane \emph{without} requiring an extra-gradient step, which is where the Lipschitz continuity of $F$ came into play in~\Cref{lemma:sviormvi}. \Cref{cor:improved-quasarconvex} should be compared with the result of~\citet{Lee16:Optimizing} pertaining to optimizing \emph{star-convex} functions---the special case of~\Cref{def:quasarconvex} where $\lambda = 1$.

\paragraph{Weak Minty condition} We also strengthen~\Cref{theorem:mainMinty-informal} along another axis. Perhaps the most immediate question is how far can one relax the assumption that $\VI(\cX, F)$ satisfies the Minty condition---while accepting the fact that, in light of the intractability of general VIs, imposing some assumptions is inevitable. Naturally, this question has received ample attention in contemporary research (\emph{cf.}~\Cref{sec:weakMinty}). One permissive condition that has emerged from that line of work is the \emph{weak} Minty property put forward by~\citet{Diakonikolas21:Efficient} in the unconstrained setting. In~\Cref{def:weakMinty}, we introduce a natural version of that notion for the constrained setting. And we show, in~\Cref{theorem:weakMinty}, that \Cref{theorem:mainMinty-informal} can be applied in a certain regime of the weak Minty condition.

\paragraph{Harmonic games}

We next discuss two implications and extensions of our main results for game theory. The first concerns \emph{harmonic games} (\Cref{def:harmonic}), a class of multi-player games at the heart of the seminal decomposition of~\citet{Candogan11:Flows}, covered in more detail in~\Cref{sec:harmonic}. Leveraging~\Cref{theorem:mainMinty-informal}, we obtain the first polynomial-time algorithm for computing $\epsilon$-Nash equilibria (per~\Cref{def:NE}, captured by $\epsilon$-SVIs) in multi-player harmonic games under the polynomial expectation property~\citep{Papadimitriou08:Computing}; this latter condition postulates that one can efficiently compute utility gradients (equivalently, the underlying mapping $F$ can be evaluated efficiently), which holds in most succinct classes of games.

\begin{corollary}
    \label{cor:harmonic}
    There is a polynomial-time algorithm for computing $\epsilon$-Nash equilibria in (succinct) multi-player harmonic games under the polynomial expectation property.
\end{corollary}

This algorithm is based on the observation that harmonic games satisfy a weighted version of the Minty condition. In particular, after applying a suitable transformation, we show that the induced VI problem satisfies the usual Minty condition under a Lipschitz continuous mapping (\Cref{prop:harmonic}), at which point~\Cref{cor:harmonic} follows from~\Cref{theorem:mainMinty-informal}. Crucial to this argument is the fact that the weights in harmonic games cannot be too close to $0$ (\Cref{lemma:bitcompl}), for otherwise the Lipschitz continuity parameter would blow up; this issue 
is the crux in our refinement for two-player games, which is the subject of the next paragraph.

\paragraph{Nash or strict coarse correlated equilibria in two-player games}

\label{sec:intro-twoplayer}
Our next application concerns equilibrium computation in two-player concave games.
As we alluded to, EVIs are closely related to the notion of a coarse correlated equilibrium (CCE) from game theory. More precisely, when the underlying VI problem corresponds to a multi-player game, the EVI gap equates to the \emph{sum} of the players' deviation benefits under a correlated distribution; this does not quite capture the usual notion of a CCE in which one bounds the \emph{maximum} of the deviation benefits (\Cref{sec:duality}). In light of this, we also provide the following refinement of~\Cref{theorem:maineither-informal} in two-player concave games.

\begin{theorem}[Precise version in~\Cref{th:2p-nash-or-scce}]
    \label{theorem:informal-twoplayerstrict}
    In a two-player concave game, there is an algorithm that runs in time $\poly(d, \log(B/\epsilon), \log L)$ and returns either
    \begin{enumerate}
        \item an $\epsilon$-Nash equilibrium or
        \item an $\Omega_\epsilon(\epsilon^6)$-strict CCE.
    \end{enumerate}
\end{theorem}

Unlike~\Cref{theorem:maineither-informal}, which can be applied to multi-player games to find Nash or strict {\em ACCE}s, \Cref{th:2p-nash-or-scce} cannot be extended to games with more than two players---one can always include a third player who always obtains zero utility. \Cref{theorem:informal-twoplayerstrict} provides an exponential improvement over a known result by~\citet{Anagnostides22:Optimistic}, who gave a $\poly(d, B, L, 1/\eps)$-time algorithm for the same problem via \emph{optimistic mirror descent}.\looseness-1

As in harmonic games, the proof of~\Cref{theorem:informal-twoplayerstrict} is based on analyzing a weighted version of the Minty condition---this allows transitioning from the sum to the maximum deviation benefit. But, unlike harmonic games, here we need to handle the case where one of the weights is arbitrarily close to $0$, in which case the Lipschitz continuity parameter blows up, which in turn neutralizes the strict separation oracle of~\Cref{lemma:sviormvi}. We address this by providing a tailored separation oracle for two-player games when one of the weights gets too close to $0$ (\Cref{lemma:semiseparation-twoplayer}).

\subsubsection{Lower bounds}
\label{sec:introlowerbounds}

As promised, we complement~\Cref{theorem:maineither-informal} by proving that determining whether the Minty condition holds is $\coNP$-complete.

\begin{theorem}[Precise version in~\Cref{prop:conp minty inclusion,theorem:convexgame-hard,theorem:hardmintyconp}]
    \label{theorem:hardmintyconp-informal}
    Determining whether a VI problem satisfies the Minty condition (\Cref{ass:Minty}) is $\coNP$-complete. Hardness holds even for two-player concave games or multi-player (succinct) normal-form games and even when a constant approximation error $\eps$ is allowed.
\end{theorem}

Inclusion in $\coNP$ follows because a strict EVI solution is itself an efficiently verifiable witness of MVI infeasibility. 
On the other hand, for explicitly represented (normal-form) games, the duality between EVIs and MVIs (\Cref{prop:duality}) enables us to show the following positive result.

\begin{proposition}
    There is a polynomial-time algorithm that determines whether an explicitly represented (normal-form) game satisfies the Minty condition.
\end{proposition}

In particular, for such problems, it is well-known that one can efficiently optimize over the polytope of EVI solutions, so one can in particular minimize the equilibrium gap.

One final natural question that arises from~\Cref{theorem:mainMinty-informal} concerns the complexity of computing \emph{Minty} VI solutions (per~\cref{def:MVI}), {\em when promised that such a solution exists}; by~\Cref{lemma:Minty}, this would be stronger than computing SVI solutions. With a straightforward construction, we observe that this is information-theoretically impossible.

\begin{proposition}[Precise version in~\Cref{theorem:Mintyexp,theorem:Mintyexp-d}]
    \label{theorem:Mintyexp-informal}
    Computing an $\epsilon$-MVI solution requires $\Omega(1/\eps)$ oracle evaluations to $F$ even when an MVI solution is guaranteed to exist and $d=1$. When $d$ is large, it requires $\Omega(2^d)$ oracle evaluations even when $\eps$ is an absolute constant.
\end{proposition}

In particular, this lower bound implies that the dependence on $\lambda$ in~\Cref{cor:improved-quasarconvex,cor:quasarconvex} cannot be removed (\Cref{theorem:Mintyexp}).
\subsection{Further related work}
\label{sec:related}

We have already provided references establishing $\poly(d, 1/\epsilon)$-time algorithms (ignoring the dependency on other parameters of the problem) for solving $\epsilon$-SVIs under the Minty condition. Here, it is worth elaborating more on the prior results with complexity scaling polynomially in $\log(1/\epsilon)$; the upshot is that they all rest on restrictive assumptions that are stronger than the Minty condition. \citet{Luthi85:Solution} gave an algorithm for $\epsilon$-SVIs using the ellipsoid, but the analysis assumes that $F$ is \emph{strongly} monotone. As we have seen, the Minty condition is weaker than the assumption that $F$ is monotone, while strong monotonicity is a significantly stronger assumption still. \citet{Aghassi06:Solving} derive a compact convex formulation for a class of SVIs that includes monotone, affine VIs over polyhedral sets (\emph{cf.}~\citealp{Harker91:Polynomial}). \citet{Magnanti95:Unifying} established a polynomial-time algorithm for SVIs under what is today referred to as \emph{$\alpha$-cocoercivity}: $\langle F(\vx) - F(\vx'), \vx - \vx' \rangle \geq \alpha \|F(\vx) - F(\vx') \|^2$ for all $\vx, \vx' \in \cX$, where $\alpha > 0$. Relatedly, \citet{Goffin97:Analytic} provide a nonasymptotic analysis for a cutting-plane method under the assumption that $F$ is \emph{pseudo-cocoercive}---that is, $\langle F(\vx'), \vx - \vx' \rangle \geq 0 \implies \langle F(\vx), \vx - \vx' \rangle \geq \alpha \| F(\vx) - F(\vx') \|^2$ for all $\vx, \vx' \in \cX$. \citet{Song20:Optimistic} provided an algorithm with complexity scaling as $O_{\epsilon}(\log(1/\epsilon))$ under a strengthening of the Minty condition.\footnote{Some authors (\emph{e.g.},~\citealp{Song20:Optimistic,Zhou20:Convergence}) refer to the Minty condition as ``variational coherence.''} Linear convergence was also established by~\citet{Ye22:Infeasible} under a certain ``error bound,'' which is again significantly stronger than the Minty condition. Polynomial-time algorithms were known under monotonicity (\emph{e.g.}, we refer to~\citealp{Rodomanov23:Subgradient,Nemirovski10:Accuracy}), but, as we explained earlier, rely fundamentally on an approach that falls short under the Minty condition (\Cref{sec:collapse}). Another class of problems that are amenable to the ellipsoid algorithm concerns computing fixed points of non-expansive mappings~\citep{Sikorski93:Ellipsoid,Huang99:Approximating}, which in turn allows solving VIs under cocoercivity (\emph{e.g.},~\citealp{Diakonikolas20:Halpern}).

The Minty variational inequality problem also relates to the seminal concept of an \emph{evolutionarily stable strategy (ESS)} from evolutionary game theory~\citep{Smith73:Logic}. We refer to~\citet{Migot21:Minty} for precise connections between MVIs and ESSs; the upshot is that a point $\vx \in \cX$ that satisfies the MVI \emph{strictly} for any $\vx' \neq \vx$ equates to a certain variant of ESS. Interestingly, \citet{Conitzer19:Exact} showed that even ascertaining the existence of an ESS in two-player (normal-form) games is $\SigmaP$-complete; by contrast, \Cref{prop:Mintyeasy} shows that determining whether the Minty property holds in games with a constant number of players can be solved in polynomial time. From a computational standpoint, this makes the Minty condition a more compelling criterion for ascertaining evolutionary plausibility.

Recent research has focused extensively on \emph{higher-order} methods for computing SVIs~\citep{Bullins22:Higher,Huang22:Cubic,Adil22:Optimal,Huang22:Approximation,Jiang22:Generalized,Lin24:Perseus}. For example, \citet{Lin24:Perseus} developed an algorithm with iteration complexity scaling with $O_\epsilon(\epsilon^{-2/(p+1)})$; the main caveat with those results is that implementing each iteration requires time that grows exponentially in $p$.

Moreover, computing the strictest (A)CCE has been used to address the equilibrium selection problem~\citep{Marris25:Deviation}; a game can have multiple CCEs of varying properties, so it is often not clear which one is to be preferred from either a prescriptive or descriptive point of view. The (A)CCE that minimizes the equilibrium gap---which can be negative (\emph{cf.}~\Cref{prop:duality})---offers one possible answer to that dilemma.

For further references pertaining to earlier developments on variational inequalities, we refer to the survey of~\citet{Harker90:Finite}.
\section{Preliminaries}
\label{sec:prels}

Before moving on, we lay out some basic notation and background on optimization, and then proceed to formally state our blanket assumptions.

\paragraph{Notation} We use lowercase boldface letters, such as $\vx$, to denote points in $\R^d$ and capital letters, such as $\mat{A} \in \R^{d \times d'}$, for matrices. The $j$th coordinate of $\vx$ is accessed by $\vx[j]$. For $\vx \in \R^{d}$ and $\vx' \in \R^{d'}$, $(\vx, \vx') \in \R^{d + d'}$ represents their concatenation. We denote by $\size(r)$ the \emph{encoding length} of a rational number $r \in \Q$ in binary. If $\langle \cdot, \cdot \rangle$ represents the standard inner product, $\|\vx\| \defeq \sqrt{\langle \vx, \vx\rangle} $ is the (Euclidean) $\ell_2$ norm of $\vx$, while $\|\vx\|_\infty \defeq \max_{1 \leq j \leq d} \lvert\vx[j]\rvert$ is its $\ell_\infty$ norm. For a matrix $\mat{A} \in \R^{d \times d'}$, $\|\mat{A}\|$ denotes its spectral norm. $\mat{I}_{d} \in \R^{d \times d}$ is the $d \times d$ identity matrix. $\cB_r(\vx)$ is the (closed) Euclidean ball centered at $\vx \in \R^d$ with radius $r > 0$. We use $\Delta(d)$ for the $d$-simplex: $\Delta(d) \defeq \{ \vx \in \R^d_{\geq 0} : \sum_{j=1}^d \vx[j] = 1 \}$. We use the notation $O(\cdot), \Theta(\cdot), \Omega(\cdot)$ to suppress absolute constants. We sometimes use $O_n(\cdot)$ to highlight the dependence only on the parameter $n$.

We distinguish between a VI problem, denoted by $\VI(\cX, F)$---which is given by a mapping $F: \cX \to \R^d$ that can be evaluated in polynomial time (\Cref{ass:polyeval}) and a constraint set that is implicitly accessed through a separation oracle (\Cref{def:sep})---and a solution thereof, be it an SVI (\Cref{def:VI}) or an MVI (\Cref{def:MVI}).

Let $\cX$ be a convex and compact set in $\R^d$. We define
\begin{equation}
    \label{eq:deepset}
    \cX^{-\epsilon} \defeq \{ \vx \in \cX: \cB_\epsilon(\vx) \subseteq \cX \},
\end{equation}
where we recall that $\cB_\epsilon(\vx)$ is the (closed) Euclidean ball centered at $\vx \in \R^d$ with radius $\epsilon > 0$. \eqref{eq:deepset} describes all points that are ``$\epsilon$-deep'' inside $\cX$. Similarly, we define
\begin{equation}
    \label{eq:approxset}
    \cX^{+\epsilon} \defeq \{ \vx' \in \R^d : \|\vx' - \vx \| \leq \epsilon \text{ for some } \vx \in \cX \} = \bigcup_{\vx \in \cX} \cB_\epsilon(\vx),
\end{equation}
which is the set of all points that are ``$\epsilon$-close'' to $\cX$.

We say that $\cX$ is in \emph{isotropic position} if for a uniformly sampled $\vx \sim \cX$, we have $\E[\vx] = 0$ and $\E[ \vx \vx^\top] = \mat{I}_{d \times d}$. There is a polynomial-time algorithm that brings any constraint set $\cX$ into isotropic position (for example,~\citealp{Lovasz06:Simulated}). This transformation does not affect our main result (\Cref{theorem:mainMinty-precise}) or implications thereof (\Cref{sec:isotropic}), so we assume it without loss of generality. If $\cX$ is in isotropic position, we can assume that $\cB_1(\vec{0}) \subseteq \cX \subseteq \cB_R(\vec{0})$ for some $R \leq \poly(d)$; in fact, all our positive results apply even when $R \leq \exp(\poly(d))$.

Continuing with some basic geometric definitions, we say that a set $\cE \subseteq \R^d$ is an \emph{ellipsoid} if there exists a vector $\vec{a} \in \R^d$ and a positive definite matrix $\mat{A} \in \R^{d\times d}$ such that
\begin{equation*}
    \cE = \cE(\mat{A}, \vec{a}) \defeq \{ \vx \in \R^d : \langle \vx - \vec{a}, \mat{A}^{-1} (\vx - \vec{a}) \rangle \leq 1 \};
\end{equation*}
above, we use the inverse of $\mat{A}$ so as to be consistent with~\citet{Grotschel12:Geometric}. $\vol(\cE)$ is the \emph{volume} of the ellipsoid $\cE$.

We are now ready to introduce some basic oracles, which are known to be (polynomial-time) equivalent when $\cX$ is well-bounded.\footnote{For simplicity, in the main body of the paper we posit the strong versions of the oracles. \Cref{sec:weak} explains how to generalize our analysis under weak oracles.} The first one simply ascertains membership.

\begin{definition}[Membership;~\citealp{Grotschel12:Geometric}]
    \label{def:mem}
    Given a point $\vx \in \Q^d$, decide whether $\vx \in \cX$.
\end{definition}

A separation oracle takes a step further: if the point is not in the set, it proceeds by producing a separating hyperplane, as defined next.

\begin{definition}[Separation;~\citealp{Grotschel12:Geometric}]
    \label{def:sep}
    Given a point $\vec{x} \in \Q^d$, either
    \begin{itemize}
        \item  assert that $\vx \in \cX$ or
        \item find a vector $\vec{c} \in \Q^d$ with $\| \vec{c} \|_\infty = 1$ such that $\langle \vec{c}, \vx' \rangle < \langle \vec{c}, \vx \rangle$ for every $\vx' \in \cX$.
    \end{itemize}
\end{definition}

The final useful oracle optimizes linear functions with respect to the constraint set; it enables verifying whether a point is an approximate SVI solution.

\begin{definition}[Optimization;~\citealp{Grotschel12:Geometric}]
    \label{def:opt}
    Given a vector $\vec{c} \in \Q^d$, find a vector $\vx \in \cX$ such that $\langle \vec{c}, \vec{x} \rangle \leq \langle \vec{c}, \vx' \rangle$ for all $\vx' \in \cX$. 
\end{definition}

Our main assumption concerning $\cX$ is that it is given implicitly through access to any of those computationally equivalent oracles. 

With regard to the mapping $F$ of the underlying variational inequality problem, we gather our assumptions below; some of our results, such as~\Cref{theorem:quasar-precise}, weaken these assumptions.

\begin{assumption}
    \label{ass:polyeval}
    The mapping $F : \cX \to \R^d$ satisfies the following:
    \begin{enumerate}
        \item for any rational $\vx \in \cX \cap \Q^d$, $F(\vx)$ is a rational vector that can be evaluated exactly in $\poly(d)$ time, with $\size(F(\vx)) \leq \poly(\size(\vx))$;\label{item:polyeval}
        \item for any $\vx, \vx' \in \cX$, $\|F(\vx) - F(\vx') \| \leq L \|\vx - \vx' \|$ for some $L \in \Q_{>0}$; and\label{item:Lparam}
        \item for any $\vx \in \cX$, $\|F(\vx)\| \leq B$ for some $B \in \Q_{>0}$.\label{item:Bparam}
    \end{enumerate}
\end{assumption}
When specialized to multi-player games, \Cref{item:polyeval} is precisely the \emph{polynomial expectation property} of~\citet{Papadimitriou08:Computing}.

For simplicity, our algorithm takes as input $L$ and $B$, under the promise that $F$ satisfies~\Cref{ass:polyeval}; this is not necessary: one can run the algorithm starting from $L_0, B_0$; if the output does not satisfy the desired property, it suffices to repeat, setting $L_{1} = 2 L_0$ and $B_{1} = 2 B_0$, and so on.
\section{Consequences and manifestations of the Minty condition}
\label{sec:Mintyexamples}

The Minty condition (\Cref{ass:Minty}) is intimately related to several important and seemingly disparate concepts in optimization and game theory. This section gathers a number of such connections, most of which are new. In~\Cref{sec:optimization}, we begin by exploring connections in optimization, mostly revolving around quasar-convexity (\Cref{def:quasarconvex}) and a more general incarnation thereof in general VI problems (\Cref{def:smoothVI}). \Cref{sec:duality} establishes a duality between MVIs and a relaxation of SVIs called \emph{expected VIs} (\Cref{prop:duality}), which are closely related to the notion of coarse correlated equilibria from game theory. We then leverage this duality to show that in explicitly represented multi-player games, ascertaining whether the Minty condition holds can be phrased as a linear program of polynomial size; this is to be contrasted with our \coNP-hardness for succinct games (\Cref{theorem:hardmintyconp}). \Cref{sec:harmonic} examines classes of games that satisfy the Minty condition, which notably includes \emph{harmonic games}---after applying a suitable transformation (\Cref{prop:harmonic}).

\subsection{Optimization}
\label{sec:optimization}

We first turn our attention to optimizing a single function. Let $f : \cX \to \R$ be a differentiable function to be minimized.\footnote{We always assume that $f$ is differentiable on an open superset $\hat{\cX} \supset \cX$.} As we discussed earlier, an $\epsilon$-SVI solution $\vx \in \cX$ to the problem arising when $F \defeq \nabla f$ is an approximate stationary point of gradient descent applied on $f$; namely, $\langle \nabla f(\vx), \vx' -\vx \rangle \geq - \epsilon$ for all $\vx' \in \cX$. In particular, if there exists $r > 0$ such that $\cB_r(\vx) \subseteq \cX$---that is, $\vx$ is in the interior of $\cX$---it follows that $\| \nabla f(\vx) \| \leq \epsilon/r$. Of course, when $f$ is nonconvex, $\epsilon$-SVIs are not generally approximate global minima---they can even be saddle points. This is in stark contrast to MVI solutions, although their existence is not guaranteed in general; indeed, the following was observed, for example, by~\citet[Theorem 2.10]{Huang23:Monotone}.

\begin{proposition}[\citealp{Huang23:Monotone}]
    Consider the optimization problem $\min_{\vx \in \cX} f(\vx)$, where $f$ is differentiable and $\cX$ is convex and compact. If $\vx \in \cX$ is an MVI solution with respect to $F \defeq \nabla f$, then $\vx$ is a global minimum of $f$.
\end{proposition}

On the other hand, a global minimum of $f$ is not necessarily an MVI solution---otherwise our main result (\Cref{theorem:mainMinty-precise}) would imply $\CLS \subseteq \P$ (\emph{cf.}~\citealp{Fearnley23:Complexity}); see~\citet[Remark 2.11]{Huang23:Monotone} for a concrete example.

One special case in which SVI solutions do correspond to global minima is when $f$ is \emph{quasar-convex}~\citep{Fu23:Accelerated,Hinder20:Near,Caramanis24:Optimizing,Hardt18:Gradient}, in the following sense (restated from the introduction).

\quasarconvex*

Quasar-convexity is a generalization of the usual notion of convexity: if~\eqref{eq:quasarconvex} holds for any $\vx, \vx' \in \cX$ and $\lambda = 1$, one recovers precisely convexity for differentiable functions. Further, when~\eqref{eq:quasarconvex} holds only with respect to $\vx$ (as in~\Cref{def:quasarconvex}) but with $\lambda = 1$, we get the notion of \emph{star-convexity}~\citep{Lee16:Optimizing,Nesterov06:Cubic}. It follows immediately from the definition that quasar-convexity is in fact a strengthening of the Minty condition, which further guarantees that all approximate SVI solutions are approximately global minima in terms of value---this latter property does not necessarily hold under solely the Minty condition.\footnote{\citet{Hinder20:Near} established a lower bound of $\Omega(\lambda^{-1} \epsilon^{-1/2} )$ in terms of the number of gradient evaluations required to minimize a $\lambda$-quasar convex function; this does not contradict~\Cref{cor:improved-quasarconvex} because their lower bound only applies if the dimension is large enough as a function of $\epsilon$; in particular, their lower bound targets ``dimension-free'' algorithms.}

\begin{proposition}
    \label{prop:quasarconvex}
    Let $F \defeq \nabla f$ for a differentiable, $\lambda$-quasar-convex function $f : \cX \to \R$ with respect to a global minimum $\vx \in \cX$ of $f$. Then, $\vx$ is a solution to the Minty VI problem. Furthermore, any $\epsilon$-SVI solution $\vx'$ satisfies $f(\vx') \leq f(\vx) + \epsilon/\lambda$.
\end{proposition}

\begin{proof}
    The fact that $\vx$ satisfies the Minty VI follows since
    \begin{equation*}
        \langle \nabla f(\vx'), \vx' - \vx \rangle \geq \lambda (f(\vx') - f(\vx)) \geq 0
    \end{equation*}
    for any $\vx' \in \cX$, where we used the fact that $f(\vx) \leq f(\vx')$ ($\vx$ is a global minimum of $f$). Now, let $\vx'$ be an $\epsilon$-SVI solution, which implies that $\langle \nabla f(\vx'), \vx - \vx' \rangle \geq - \epsilon$. Combining with~\eqref{eq:quasarconvex}, we have
    \begin{equation*}
        f(\vx') \leq f(\vx) + \frac{\epsilon}{\lambda},
    \end{equation*}
    as promised.
\end{proof}

\paragraph{Smooth VIs} The notion of quasar-convexity given in~\Cref{def:quasarconvex} can be significantly generalized to general VI problems, as observed recently by~\citet{Zhang25:Expected}; as we shall see, this captures a special case of the seminal notion of \emph{smoothness}, introduced by~\citet{Roughgarden15:Intrinsic}.\footnote{We caution that smoothness per~\Cref{def:smoothVI} is different than the usual notion of smoothness in optimization; we chose to overload notation so as to be consistent with the terminology of~\citet{Roughgarden15:Intrinsic}.} (For convenience, we take the perspective of maximization in the following definition.)

\begin{definition}[Smoothness for VIs; \citealp{Zhang25:Expected}]
    \label{def:smoothVI}
    Let $\lambda > 0$ and $\nu > -1$. Consider further a function $Q : \cX \to \R$ and a global maximum $\vx \in \cX$ of $Q$. A VI problem with respect to the mapping $F : \cX \to \R^d$ is called \emph{$(\lambda, \nu)$-smooth} with respect to $Q$ and $\vx$ if
    \begin{equation}
        \label{eq:smoothVIs}
        \langle F(\vx'), \vx' - \vx \rangle \geq \lambda Q(\vx) - (\nu+1) Q(\vx') \quad \forall \vx' \in \cX.
    \end{equation}
\end{definition}
In particular, $(\lambda, \lambda - 1)$-smoothness equates to $\lambda$-quasar-convexity when we define $Q \defeq - f$ and $F \defeq \nabla f$. Furthermore, in the special case of multi-player games, \Cref{def:smoothVI} is equivalent to the seminal concept of smoothness introduced by~\citet{Roughgarden15:Intrinsic}---in particular, \Cref{def:smoothVI} is a direct extension of the more general concept of ``local smoothness'' per~\citet{Roughgarden15:Local}.

\begin{definition}[Smoothness for games;~\citealp{Roughgarden15:Intrinsic}]
    \label{def:smoothagmes}
    Let $\lambda > 0$ and $\nu > -1$, and $\vx \in \argmax_{\vx' \in \cX} \sw(\vx')$ with respect to an $n$-player game $\Gamma$. $\Gamma$ is called \emph{$(\lambda, \nu)$-smooth} with respect to $\vx$ if
    \begin{equation}
        \label{eq:smoothgames}
        \sum_{i=1}^n u_i(\vx_i, \vx_{-i}') \geq \lambda \sw(\vx) - \nu \sw(\vx') \quad \forall \vx' \in \cX.
    \end{equation}
\end{definition}
Above, $\sw(\vx) \defeq \sum_{i=1}^n u_i(\vx)$ is the social welfare function. By setting $Q \defeq \sw$, we see that~\eqref{eq:smoothgames} is equivalent to~\eqref{eq:smoothVIs} due to multilinearity. The key motivation behind~\Cref{def:smoothagmes} is that it enables bounding the social welfare of any (A)CCE in terms of the optimal welfare~\citep{Roughgarden15:Intrinsic}. Smoothness manifests itself prominently in a host of important applications; for example, we refer to the survey of~\citet{Roughgarden17:Price}. For our purposes, the key point is that~\Cref{def:smoothVI} enables generalizing~\Cref{prop:quasarconvex} to a broader family of problems beyond (single-function) optimization:

\begin{proposition}
    \label{prop:smoothVIs}
    Let $\lambda > 0$, a function $Q : \cX \to \R$ with a global maximum at $\vx \in \cX$, and a mapping $F : \cX \to \R^d$. If the corresponding VI problem is $(\lambda, \lambda - 1)$-smooth with respect to $Q$ and $\vx$, then $\vx$ is a solution to the Minty VI problem. Furthermore, any $\epsilon$-SVI solution $\vx'$ satisfies $Q(\vx') \geq Q(\vx) - \epsilon/\lambda$.
\end{proposition}

The proof is analogous to that of~\Cref{prop:quasarconvex}.

\subsection{Expected VIs and duality with MVIs}
\label{sec:duality}

The next connection is that MVIs are, in a precise sense, duals of a certain relaxation of SVIs recently introduced by~\citet{Zhang25:Expected}, called \emph{expected VIs}; we refer to~\citet{Cai24:Tractable} and~\citet{Ahunbay25:First} for some precursors of that definition in the context of nonconcave games.

We begin by stating the definition of expected VIs; following~\citet{Zhang25:Expected}, we use the term ``expected VIs,'' as opposed to expected SVIs, although they relax the SVI problem of~\Cref{def:VI}.

\begin{definition}[\citealp{Zhang25:Expected,Nemirovski10:Accuracy}]
    \label{def:EVIs}
    In the context of~\Cref{def:VI}, the $\epsilon$-\emph{expected} VI (EVI) problem asks for a \emph{distribution}\footnote{In the context of a monotone VI problem, \citet{Nemirovski10:Accuracy} call an $\eps$-EVI an ``accuracy certificate.'' In this paper, we use the $\eps$-EVI terminology because it is more general and does not presuppose that an $\eps$-EVI is informative of a VI solution; {\em cf.} the discussion in \Cref{sec:collapse}.} $\mu \in \Delta(\cX)$ over $\cX$ such that
    \begin{equation}
        \label{eq:EVI}
        \E_{\vx \sim \mu} \langle F(\vx), \vx' - \vx \rangle \geq - \epsilon \quad \forall \vx' \in \cX.
    \end{equation}
\end{definition}

That is, in an expected VI, it suffices if the (S)VI constraint holds \emph{in expectation} when $\vx$ is drawn from a distribution $\mu$. Unlike SVIs, expected VIs can be solved in $\poly(d, \log(1/\epsilon))$ time for any approximation $\epsilon > 0$~\citep{Zhang25:Expected}. \Cref{def:EVIs} places no restriction on $\epsilon$ being nonnegative; indeed, expected VIs with a negative gap may exist (\emph{cf.}~\Cref{prop:duality})---this is obviously not possible for SVIs. For $\epsilon < 0$, an $\epsilon$-EVI solution will also be referred to as a \emph{$(-\epsilon)$-strict EVI} solution.

In this context, starting from~\eqref{eq:EVI}, we observe that
\begin{equation*}
    \max_{\mu \in \Delta(\cX)} \min_{\vx' \in \cX} \E_{\vx \sim \mu} \langle F(\vx), \vx' - \vx \rangle = \min_{\vx' \in \cX} \max_{\mu \in \Delta(\cX)} \E_{\vx \sim \mu} \langle F(\vx), \vx' - \vx \rangle,
\end{equation*}
where the equality follows by the minimax theorem~\citep{Sion58:On}; indeed, the function $\E_{\vx \sim \mu} \langle F(\vx), \vx - \vx' \rangle$ is bilinear in terms of $\mu$ and $\vx'$. Equivalently,
\begin{equation}
    \label{eq:dual}
    - \max_{\mu \in \Delta(\cX)} \min_{\vx' \in \cX} \E_{\vx \sim \mu} \langle F(\vx), \vx' - \vx \rangle = \max_{\vx' \in \cX} \min_{\vx \in \cX} \langle F(\vx), \vx - \vx' \rangle,
\end{equation}
where we used the fact that, for a given $\vx' \in \cX$, $\min_{\mu \in \Delta(\cX)} \E_{\vx \sim \mu} \langle F(\vx), \vx - \vx' \rangle = \min_{\vx \in \cX} \langle F(\vx), \vx - \vx' \rangle$. By~\eqref{eq:dual}, we arrive at the following characterization of the Minty condition.

\begin{proposition}
    \label{prop:duality}
    The Minty condition (\Cref{ass:Minty}) holds if and only if there is no $\epsilon$-expected VI solution (\Cref{def:EVIs}) with $\epsilon < 0$.
\end{proposition}

\begin{proof}
    If the Minty condition holds, there exists $\vx' \in \cX$ such that $\min_{\vx \in \cX} \langle F(\vx), \vx - \vx' \rangle \geq 0$. By~\eqref{eq:dual}, this implies that $\max_{\mu \in \Delta(\cX)} \min_{\vx' \in \cX} \E_{\vx \sim \mu} \langle F(\vx), \vx' - \vx \rangle \leq 0$, which means that every $\epsilon$-expected SVI solution $\mu$ must satisfy $\epsilon \geq 0$. The converse is also immediate.
\end{proof}

\paragraph{Coarse correlated equilibria in games} \Cref{prop:duality} has a particularly notable consequence in the context of $n$-player (normal-form) games. Here, each player $i \in [n]$ selects as (mixed) strategy a probability distribution $\vx_i \in \Delta(\cA_i) \eqqcolon \cX_i$ over a finite set of available actions $\cA_i$. Under a joint strategy $\vx = (\vx_1, \dots, \vx_n) \in \cX_1 \times \dots \times \cX_n \eqqcolon \cX$, we denote by $u_i(\vx)$ the expected utility of a player $i$. As noted by~\citet{Zhang25:Expected}, expected VIs (per~\Cref{def:EVIs}) correspond to the following equilibrium concept; this can be readily extended to general concave games (\Cref{sec:twoplayer} does so for two-player games).

\begin{definition}[Average CCE]
    \label{def:acce}
    For an $n$-player game, a distribution $\mu \in \Delta(\cA_1 \times \dots \times \cA_n)$ is an \emph{$\epsilon$-average coarse correlated equilibrium ($\epsilon$-ACCE)} if
    \begin{equation}
        \label{eq:acce}
        \sum_{i=1}^n  \E_{\vx \sim \mu} [ u_i(\vx_i', \vx_{-i}) - u_i(\vx) ] \leq \epsilon \quad \forall \vx_1', \dots, \vx_n' \in \cX_1 \times \dots \times \cX_n.
    \end{equation}
\end{definition}

Several remarks are in order. An ACCE is a relaxation of a CCE~\citep{Moulin78:Strategically}, which in turn relaxes \emph{correlated equilibria (CE)} \emph{\`a la}~\citet{Aumann74:Subjectivity}. The key difference between ACCEs and CCEs is that the former only insists on bounding the \emph{cumulative} deviation benefit over all players, whereas in a CCE one bounds the \emph{maximum} deviation benefit; the term ``average'' CCE is due to~\citet{Nadav10:Limits}, who also separated ACCEs from CCEs. Now, as we shall see, expected VIs are to ACCEs what SVIs are to \emph{Nash equilibria}; we now recall the latter definition.

\begin{definition}
    \label{def:NE}
    For an $n$-player game, a strategy $\vx \in \cX_1 \times \dots \times \cX_n$ is an \emph{$\epsilon$-Nash equilibrium} if\footnote{It is more common to bound the maximum deviation benefit (as opposed to the cumulative one), but---unlike CCEs---the two are equivalent up to a factor of $n$ in the approximation.}
    \begin{equation*}
        \sum_{i=1}^n ( u_i(\vx'_i, \vx_{-i}) - u_i(\vx) ) \leq \epsilon \quad \forall \vx_1', \dots, \vx_n' \in \cX_1 \times \dots \times \cX_n.
    \end{equation*}
\end{definition}

As we alluded to, \Cref{def:acce} can be naturally cast as an expected VI problem per~\Cref{def:EVIs}. Indeed, we define 
\begin{equation}
    \label{eq:gameoperator}
    F: \vx \mapsto ((- u_i(a_i, \vx_{-i}) )_{a_i \in \cA_i} )_{i=1}^n \text{ and } \cX \defeq \Delta(\cA_1) \times \dots \times \Delta(\cA_n).
\end{equation}
That~\eqref{eq:acce} is equivalent to the resulting expected VI problem follows immediately from the definitions, noting that one can without any loss of generality restrict $\mu$ to be a distribution over pure strategies; this is sometimes referred to as the ``revelation principle,'' which is interestingly not satisfied for more general versions of the expected VI problem~\citep{Zhang25:Expected}.

In this context, a direct consequence of~\Cref{prop:duality} is that there is a linear program with a number of variables and constraints polynomial in $\prod_{i=1}^n |\cA_i|$, whose output determines whether the Minty property holds. In particular, one can compute the ACCE that minimizes the equilibrium gap per~\eqref{eq:acce}. Let $\epsilon$ be the output of that linear program. By~\Cref{prop:duality}, the Minty condition---with respect to the corresponding mapping $F$---holds if and only if $\epsilon = 0$ (of course, there is always a $0$-ACCE simply because an exact Nash equilibrium exists).

\begin{proposition}
    \label{prop:Mintyeasy}
    For any $n$-player normal-form game, there is an algorithm polynomial in $\prod_{i=1}^n |\cA_i|$ (and the number of bits needed to encode the payoff tensor) that determines whether the Minty condition with respect to the corresponding VI problem per~\eqref{eq:gameoperator} holds.
\end{proposition}

As a result, there is a polynomial-time algorithm for determining whether the Minty condition holds in explicitly represented (normal-form) games, meaning that the input fully specifies each entry of the utility tensors; as we show in~\Cref{theorem:hardmintyconp}, this is no longer the case in succinct games (with the polynomial expectation property). Moreover, we also state the following immediate consequence.

\begin{proposition}
    \label{prop:MVIeasy}
    For any $n$-player game that satisfies the Minty condition, there is an algorithm polynomial in $\prod_{i=1}^n |\cA_i|$ that determines an MVI solution (per~\Cref{def:MVI}).
\end{proposition}

\paragraph{Games with nonnegative sum of regrets} Moreover, the condition that every $\epsilon$-expected VI solution has nonnegative gap (\Cref{prop:duality}) is precisely the condition put forward by~\citet{Anagnostides22:Last}, which was used to define the class of games with ``nonnegative sum of regrets.'' To be precise, the \emph{regret} of a player $i \in [n]$ who has observed the sequence of utilities $(\vu_i^{(t)}(\vx^{(t)}_{-i}) )_{1 \leq t \leq T}$ and has selected the sequence of strategies $(\vx_i^{(t)})_{1 \leq t \leq T}$ is defined as
\begin{equation}
    \label{eq:genreg}
    \reg_i^{(T)} \defeq \max_{\vx_i \in \cX_i} \sum_{t=1}^T \alpha^{(t)} \langle \vx_i - \vx_i^{(t)}, \vu_i^{(t)} \rangle,
\end{equation}
where $\alpha^{(1)}, \dots, \alpha^{(T)} \geq 0$ are weights such that $\sum_{t=1}^T \alpha^{(t)} = 1$; it is common to define regret when $\alpha^{(1)} = \dots = \alpha^{(T)} = \nicefrac{1}{T}$, but we will operate under the more flexible definition given in~\eqref{eq:genreg}.

\begin{observation}
    We say that a game $\Gamma$ has nonnegative sum of regrets if for any $T \in \N$, weights $(\alpha^{(t)})_{1 \leq t \leq T} \in \Delta(T)$, and sequence of joint strategies $(\vx^{(t)})_{1 \leq t \leq T} \in \cX^T$, it holds that $\sum_{i=1}^n \reg_i^{(T)} \geq 0$. This is equivalent to any $\epsilon$-ACCE in $\Gamma$ satisfying $\epsilon \geq 0$. By~\Cref{prop:duality}, it is also equivalent to the Minty property with respect to $\Gamma$.
\end{observation}
Among other implications, in such games it is possible to show that a broad class of no-regret dynamics---namely, ones satisfying the \emph{RVU bound} of~\citet{Syrgkanis15:Fast}, such as \emph{optimistic} mirror descent---guarantees optimal per-player (average) regret vanishing at a rate of $T^{-1}$; it remains an open question whether this is possible in general multi-player games (\emph{cf.}~\citealp{Daskalakis21:Near}).

\subsection{Harmonic games}
\label{sec:harmonic}

Moving forward, we observe that (a weighted version of) the Minty condition manifests itself in \emph{harmonic games}. This is a class of games introduced by~\citet{Candogan11:Flows}, who famously provided a decomposition of any game---based on Helmholtz
decomposition---into a direct sum of a \emph{potential} game and a harmonic game; this decomposition is unique up to an affine transformation that preserves the equilibria of the game. The potential component captures games with aligned interests, whereas the harmonic component captures games with conflicting interests.

Following recent follow-up work~\citep{Abdou22:Decomposition,Legacci24:No}, we give below a more general definition of harmonic games than the one introduced by~\citet{Candogan11:Flows}.\footnote{Under the original definition of~\citet{Candogan11:Flows}, the strategy profile in which each player mixes uniformly at random is a Nash equilibrium, thereby trivializing equilibrium computation.}

\begin{definition}[Harmonic games; \citealp{Abdou22:Decomposition,Legacci24:No,Candogan11:Flows}]
    \label{def:harmonic}
    A finite game $\Gamma$ is called \emph{harmonic} if for each player $i \in [n]$ there exists $\vec{\sigma}_i \in \R_{>0}^{\cA_i}$ such that
    \begin{equation}
        \label{eq:harmonic}
        \sum_{i=1}^n \sum_{a_i \in \cA_i} \vec{\sigma}_i[a_i] (u_i(\vec{a}') - u_i(a_i, \vec{a}'_{-i}) ) = 0 \quad \forall \vec{a}' \in \cA_1 \times \dots \times \cA_n.
    \end{equation}
\end{definition}
To cast this as a special case of the Minty property, we observe that~\eqref{eq:harmonic} can be equivalently reformulated as asking for a collection of (strictly) positive weights $w_1, \dots, w_n$ and fully mixed strategies $\vx_1 \in \Delta(\cA_1) \cap \R^{\cA_1}_{>0}, \dots, \vx_n \in \Delta(\cA_n) \cap \R_{>0}^{\cA_n}$ such that\footnote{Combining~\Cref{lemma:Minty} and~\Cref{prop:harmonic}, it follows that $\vx$ is an exact Nash equilibrium. In particular, this means that any harmonic game admits a fully mixed Nash equilibrium, but we are yet not aware of any prior polynomial-time algorithm for computing Nash equilibria in harmonic games. To elaborate on this point further, in certain classes of games, such as two-player (general-sum) games, knowing the support of the equilibrium reduces the problem to a linear system, which can be in turn solved in polynomial time (\emph{e.g.},~\citealp{Sandholm05:Mixed}). This is not so in multi-player games: \citet[Corollary 13]{Etessami07:Complexity} proved certain hardness results based on a three-player game promised to have a unique, fully mixed Nash equilibrium.}
\begin{equation}
    \label{eq:harm-equiv}
    \sum_{i=1}^n w_i ( u_i(\va') - u_i(\vx_i, \va'_{-i}) ) = 0 \quad \forall \va' \in \cA_1 \times \dots \times \cA_n.
\end{equation}

We can now recognize this as a special case of the Minty property after we suitably rescale the utilities in the definition of $F$. But there are two lingering issues with this observation: first, we do not know the weights $w_1, \dots, w_n$; and second, even if we could rescale the utilities, the complexity of the algorithm would be polynomial in $\log(\rho)$, where $\rho \defeq \max_{1 \leq i \leq n} w_i / \min_{1 \leq i \leq n} w_i$.

To address these issues, we first make a simple observation regarding the bit complexity of $\vec{\sigma}$ that satisfies~\eqref{eq:harmonic}.

\begin{lemma}
    \label{lemma:bitcompl}
    Let $\size(u_i(\va)) \leq \poly(n, \max_{1 \leq i \leq n} |\cA_i|)$ for all $i \in [n]$ and $\va \in \cA_1 \times \dots \times \cA_n$. Suppose further that~\eqref{eq:harmonic} is feasible. Then, it can be satisfied with respect to $\vec{\sigma} = (\vec{\sigma}_1, \dots, \vec{\sigma}_n) \in \R_{>0}^{\cA_1} \times \dots \times \R_{>0}^{\cA_n}$ such that $\size(\vec{\sigma}) \leq \poly(n, \max_{1 \leq i \leq n} |\cA_i|)$.
\end{lemma}

\begin{proof}
    \eqref{eq:harmonic} induces a linear program with a polynomial number of variables (and exponential number of constraints). The number of active constraints required to define a vertex is thus polynomial. Since the coefficients of each constraint have polynomial bit complexity (by assumption), the claim follows.\footnote{For explicitly represented (normal-form) games, it is evident from~\eqref{eq:harmonic} that there is a polynomial-time algorithm for computing $\vec{\sigma}$, and hence a Nash equilibrium of that game. Our focus here is on succinct games (with the polynomial expectation property), in which case the LP induced by~\eqref{eq:harmonic} has exponentially many constraints.}
\end{proof}

Returning to~\eqref{eq:harm-equiv}, we can assume that $\sum_{i=1}^n w_i = 1$ (by rescaling); \Cref{lemma:bitcompl} implies that $w_i \geq 1/2^{ \poly(n, \max_{1 \leq i \leq n} |\cA_i|)}$. Having established this lower bound, we consider the mapping
\begin{equation}
    \label{eq:G}
    G : \mathcal{P}_{\geq \alpha} \ni (w_1, \dots, w_n, w_1 \vx_1, \dots, w_n \vx_n) \mapsto (u_1(\vx), \dots, u_n(\vx), - (u_i(a_i, \vx_{-i})_{a_i \in \cA_i})_{i=1}^n),
\end{equation}
where
\begin{equation}
    \label{eq:Pa}
    \mathcal{P}_{\geq \alpha} \defeq \left\{ (w_1, \dots, w_n, w_1 \vx_1, \dots, w_n \vx_n) : \vec{w} \in \Delta(n) \cap \R^n_{\geq \alpha}, \vx_1 \in \Delta(\cA_1), \dots, \vx_n \in \Delta(\cA_n) \right\}
\end{equation}
for a sufficiently small $\alpha = 1/2^{ \poly(n, \max_{1 \leq i \leq n} |\cA_i|)}$ (per~\Cref{lemma:bitcompl}). The following now follows directly from the definitions.

\begin{proposition}
    \label{prop:harmonic}
    Consider any harmonic game $\Gamma$ per~\Cref{def:harmonic}. If we define $\VI(\mathcal{P}_{\geq \alpha}, G)$ per~\eqref{eq:G} and~\eqref{eq:Pa}, the following properties hold:
    \begin{itemize}
        \item $\VI(\mathcal{P}_{\geq \alpha}, G)$ satisfies the Minty condition;
        \item $G$ is $2^{\poly(n, \max_{1 \leq i \leq n} |\cA_i|)}$-Lipschitz continuous; and
        \item an $\epsilon$-SVI of $\VI(\mathcal{P}_{\geq \alpha}, G)$ is an $\epsilon \cdot 2^{\poly(n, \max_{1 \leq i \leq n} |\cA_i|)}$-Nash equilibrium of $\Gamma$.
    \end{itemize}
\end{proposition}

As a result, our main result (\Cref{theorem:mainMinty-precise}) implies a polynomial-time algorithm for computing $\epsilon$-Nash equilibria in harmonic games.

Some simple examples of games that adhere to~\Cref{def:harmonic} include ``cyclic games,'' in the sense of~\citet{Hofbauer00:Sophisticated}, the buyer-seller game of~\citet{Friedman91:Evolutionary}, and the crime deterrence game analyzed by~\citet{Cressman98:Evolutionary}.

\paragraph{Polymatrix zero-sum games and beyond} We continue with a related but distinct class of games known as \emph{polymatrix games}~\citep{Cai16:Zero}; in particular, it is easy to see that any two-player zero-sum game with a fully mixed Nash equilibrium is harmonic per~\Cref{def:harmonic}. We first make an observation with regard to general MVI problems, providing a sufficient condition under which~\Cref{ass:Minty} holds.

\begin{proposition}
    \label{prop:polymatrix}
    Consider a problem $\VI(\cX, F)$ such that $F$ is linear and $\langle F(\vx), \vx \rangle = 0$ for all $\vx \in \cX$. Then, the Minty condition (\Cref{ass:Minty}) holds.
\end{proposition}

\begin{proof}
    The Minty condition is equivalent to $\max_{\vx \in \cX} \min_{\vx' \in \cX} \langle F(\vx'), \vx' - \vx \rangle \geq 0$. But under our assumptions, the function $(\vx, \vx') \mapsto \langle F(\vx'), \vx' - \vx \rangle$ is bilinear, which in turn implies that $$\max_{\vx \in \cX} \min_{\vx' \in \cX} \langle F(\vx'), \vx' - \vx \rangle = \min_{\vx' \in \cX} \max_{\vx \in \cX} \langle F(\vx'), \vx' - \vx \rangle \geq 0,$$ 
    by the minimax theorem~\citep{Sion58:On}.
\end{proof}

When specialized to multi-player games, the two preconditions of~\Cref{prop:polymatrix} are satisfied when i) the game is (globally) zero-sum, meaning that $\sw(\vx) \defeq \sum_{i=1}^n u_i(\vx) = - \langle F(\vx), \vx \rangle = 0$ (by multilinearity) for all $\vx$, and ii) the utility gradient of each player is linear in the joint strategy. Those two assumptions are satisfied in zero-sum polymatrix games~\citep{Cai16:Zero}. A polynomial-time algorithm for computing Nash equilibria in zero-sum polymatrix games was obtained by~\citet{Cai16:Zero}, who observed that taking the marginals of any CCE yields a Nash equilibrium; this approach falls short more generally if one merely assumes that the Minty condition holds (\Cref{prop:equilcollapse}). On the other hand, without the zero-sum restriction, computing $\epsilon$-Nash equilibria in polymatrix games is \PPAD-hard~\citep{Rubinstein15:Inapproximability,Deligkas23:Tight}.
\section{Solving SVIs under the Minty condition}

In this section, we establish our main result. To begin with, we gather some basic facts about the central-cut ellipsoid in~\Cref{sec:centralcut}, without assuming that the underlying convex set is fully dimensional (\Cref{theorem:basicellipsoid}). We then show how to adapt this basic paradigm (in \Cref{algo:optimisticellipsoid}) by introducing some new key ideas, arriving at our main result in~\Cref{theorem:mainMinty-precise}: a polynomial-time algorithm for computing $\epsilon$-SVI solutions under the Minty condition. \Cref{sec:weakMinty,sec:noncont} concern two basic extensions of our main result: the former relaxes the Minty condition following the weaker property put forward by~\citet{Diakonikolas21:Efficient}, while the latter relaxes the assumption that the mapping $F$ is continuous, imposing instead an assumption generalizing quasar-convexity---itself a strengthening of the Minty condition (\Cref{prop:smoothVIs}). Finally, \Cref{sec:strict-EVIs} deals with the most general setting wherein the Minty condition (and relaxations thereof per~\Cref{sec:weakMinty}) can be altogether violated. It shows how the execution of our main algorithm (\Cref{algo:optimisticellipsoid}) can produce a polynomial certificate---in the form of a \emph{strict} EVI solution---that the Minty condition is violated.

\subsection{Central-cut ellipsoid}
\label{sec:centralcut}

We will use the following standard result concerning one incarnation of the central-cut ellipsoid~\citep[Theorem 3.2.1]{Grotschel12:Geometric}. It is suited to our purposes as it does not rest on the usual assumption that the underlying constraint set is fully dimensional.

\begin{theorem}[\citealp{Grotschel12:Geometric}]
    \label{theorem:basicellipsoid}
    Let $\epsilon \in \Q_{> 0}$ and $\cK \subseteq \cB_R(\vec{0})$ be a circumscribed closed and convex set, with $R \geq 1$, given by a polynomial-time oracle $\sep_\cK$ such that for any $\vx \in \Q^d$ and $\delta \in \Q_{>0}$, it either asserts that $\vx \in \cK^{+\delta}$ or finds a vector $\vec{c} \in \Q^d$ with $\|\vec{c} \|_\infty = 1$ with $\langle \vec{c}, \vx' \rangle \leq \langle \vec{c}, \vx \rangle + \delta$ for every $\vx' \in \cK$. There is a polynomial-time algorithm that returns one of the following:
    \begin{itemize}
        \item a point in $\cK^{+\epsilon}$ or
        \item an ellipsoid $\cE \subseteq \R^d$, described by a positive definite matrix $\mat{A} \in \Q^{d \times d}$ and a point $\vec{a} \in \Q^d$, such that $\cK \subseteq \cE$ and $\vol(\cE) \leq \epsilon$.
    \end{itemize}
\end{theorem}

\begin{algorithm}[!ht]
\SetAlgoLined
\KwInput{A separation oracle $\sep_\cK$ for $\cK$ per~\Cref{theorem:basicellipsoid}, rational $\epsilon > 0$.}
\KwOutput{A point in $\cK^{+\epsilon}$ or an ellipsoid $\cE \subseteq \R^d$ such that $\cK \subseteq \cE$ and $\vol(\cE) \leq \epsilon$.}

\BlankLine
Set the maximum number of iterations as $T \defeq \lceil 5 d \log(1/\epsilon) + 5d^2 \log(2R) \rceil$\;\label{line:T}
Set the precision parameter as $p \defeq 8 T$\;
Set the error tolerance for $\sep_\cK$ as $\delta \defeq 2^{-p}$\;\label{line:deltaset}
Initialize the ellipsoid $\cE^{(0)} \defeq \cE^{(0)}(\mat{A}^{(0)}, \vec{a}^{(0)})$ as $\vec{a}^{(0)} \defeq \vec{0}$ and $\mat{A}^{(0)} \defeq R^2 \mat{I}_{d}$\;
\For{$t = 0, \dots, T-1$}{
    Invoke $\sep_\cK$ with input $\vec{a}^{(t)}$ and error $\delta$\;
    \If{$\vec{a}^{(t)} \in \cK^{+\delta}$}{
        \Return{$\vec{a}^{(t)}$}\;\label{line:stop}
    }
    \Else{
        $\sep_\cK$ has returned $\vec{c} \in \Q^d$, with $\|\vec{c}\|_\infty = 1$, such that $\langle \vec{c}, \vx \rangle \leq \langle \vec{c}, \vec{a}^{(t)} \rangle + \delta$ for all $\vx \in \cK$\;
        Update the ellipsoid:
        $$\vec{a}^{(t+1)} \approx_p \vec{a}^{(t)} - \frac{1}{d+1} \frac{\mat{A}^{(t)} \vec{c} }{ \sqrt{ \langle \vec{c}, \mat{A}^{(t)} \vec{c} \rangle } }  \text{ and } \mat{A}^{(t+1)} \approx_p \frac{2d^2 + 3}{2 d^2} \left( \mat{A}^{(t)} - \frac{2}{d+1} \frac{ \mat{A}^{(t)} \vec{c} \vec{c}^\top \mat{A}^{(t)}}{ \langle \vec{c}, \mat{A}^{(t)} \vec{c} \rangle } \right).$$\label{line:precision}
    }
}
\BlankLine
\Return{$\cE^{(T)}$}\;
\caption{Central-cut ellipsoid~\citep{Grotschel12:Geometric}}
\label{algo:centalcut}
\end{algorithm}

\Cref{theorem:basicellipsoid} is based on the \emph{central-cut} ellipsoid method (\Cref{algo:centalcut}). It produces a sequence of ellipsoids, $\cE^{(0)}, \cE^{(1)}, \dots, \cE^{(T)}$, each of which contains the underlying set $\cK$, such that either at least one of their centers belongs to $\cK^{+\epsilon}$, or the last ellipsoid $\cE^{(T)}$ has volume at most $\epsilon$. We clarify that, in~\Cref{line:precision}, we use the notation $\approx_p$ to mean that the left-hand side is obtained by truncating the binary expansions of the numbers on the right-hand side after $p$ digits behind the binary point. The correctness of~\Cref{algo:centalcut} boils down to the following lemma.

\begin{lemma}[\citealp{Grotschel12:Geometric}]
    At every iteration $t$ of~\Cref{algo:centalcut}, the following properties hold:
    \begin{itemize}
        \item the matrix $\mat{A}^{(t)}$ is positive definite with $\|\vec{a}^{(t)} \| \leq R 2^t$, $\|\mat{A}^{(t)} \| \leq R^2 2^t$, and $\|(\mat{A}^{(t)})^{-1} \| \leq R^{-2} 4^t$;
        \item $\cK \subseteq \cE^{(t)}$; and
        \item $\vol(\cE^{(t+1)}) / \vol(\cE^{(t)}) \leq e^{-1/(5d)}$.
    \end{itemize}
\end{lemma}

Armed with this lemma, \Cref{theorem:basicellipsoid} follows by noting that $\vol(\cE^{(T)}) \leq e^{-T/(5d)} \vol(\cE^{(0)})$ and $\vol(\cE^{(0)}) \leq (2R)^d$; by the choice of $T$ in~\Cref{line:T}, we conclude that, if the algorithm failed to terminate (in~\Cref{line:stop}) with a point in $\cK^{+\epsilon}$ (the value of $\delta$ in~\Cref{line:deltaset} implies that~$\cK^{+\delta} \subseteq \cK^{+\epsilon}$), we have $\vol(\cE^{(T)}) \leq \epsilon$, as promised.

\subsection{Our algorithm and its analysis}
\label{sec:ouralgo}

We will now show how to leverage~\Cref{algo:centalcut} to compute $\epsilon$-SVI solutions under the Minty property. To do so, we are first faced with an immediate concern: the set of SVI solutions is not necessarily convex even when the Minty property holds (see the function behind~\Cref{theorem:Mintyexp}). On the other hand, while the set of MVI solutions is convex (\Cref{claim:convexity}), it is hard to verify whether a point satisfies the Minty VI, as we show in~\Cref{theorem:hardmintyconp,theorem:convexgame-hard}.

We address this by executing the following hybrid version of the ellipsoid. We let $\cK$ be the set of MVI solutions---points that satisfy~\eqref{eq:MVI}; for now, we assume that $\cK \neq \emptyset$, although we will relax that assumption later (\Cref{sec:weakMinty,sec:strict-EVIs}). At each iteration, we evaluate whether the center of the ellipsoid---when it belongs to $\cX$---is an $\epsilon$-SVI solution, which boils down to a call to the optimization oracle (\Cref{def:opt}); if not, the key observation is that we can \emph{strictly} separate that point from the set of MVIs---in the sense of~\Cref{def:strictsep}.  

\subsubsection{Strict separation oracle}

The basic building block of our algorithm is what we refer to as a \emph{strict} separation oracle---a strengthening of the second item of~\cref{def:sep}:

\begin{definition}[Strict separation]
    \label{def:strictsep}
    Given a point $\vx \in \R^d$ and a rational number $\gamma \in \Q_{>0}$, we say that a vector $\vec{c} \in \Q^d$, with $\|\vec{c}\|_\infty = 1$, \emph{$\gamma$-strictly separates} $\vx$ from a convex set $\cK$ if $\langle \vec{c}, \vec{x}' \rangle \leq \langle \vec{c}, \vx \rangle - \gamma$ for all $\vx' \in \cK$.
\end{definition}

The upshot is that a strict separation oracle for the set of MVIs can be indeed implemented in polynomial time assuming that the point to be separated is in $\cX$ but is \emph{not} an approximate SVI solution.

\begin{lemma}[Strict semi-separation oracle]
    \label{lemma:semiseparation}
    Given a point $\va \in \cX \cap \Q^d$ and $\epsilon \in \Q_{>0}$, there is a polynomial-time algorithm that either
    \begin{itemize}
        \item  ascertains that $\vec{a}$ is an $\epsilon$-SVI solution; or
        \item returns $\vec{c} \in \Q^d$, with $\| \vec{c} \|_\infty = 1$, such that $\langle \vec{c}, \vx \rangle \leq \langle \vec{c}, \va \rangle - \gamma$ for any point $\vx$ that satisfies the Minty VI~\eqref{eq:MVI}, where $\gamma \defeq \epsilon^2 L B^{-1} / (B + 4 R L)^2$.
    \end{itemize}
\end{lemma}

\begin{figure}[htbp]
  \centering 
  \begin{tikzpicture}[>=Stealth]
    \def\a{3}      
    \def\b{1}      
    \def\margin{0.5} 
    \def\angle{30} 

    \draw[blue, thick, rotate=\angle] (0,0) ellipse (\a cm and \b cm);

    \draw[gray, thick, rotate=\angle] (-\a-0.5, \b) -- (\a+0.5, \b) node[right] {};

    \draw[gray, thick, rotate=\angle] (-\a-0.5, \b-\margin) -- (\a+0.5, \b-\margin) node[right] {};

    \draw[->, thick, rotate=\angle] (-\a-1, \b) -- (-\a-1, \b-\margin) node[midway, left] {$\gamma$};

    \coordinate (x') at ($(-.72,.75)$);
    \filldraw[rotate=\angle, black] (x') circle (1.5pt) node[above] {$\vec{x}'$};

    \coordinate (x) at ($(x') + (+0.25, -0.43)$);
    \filldraw[rotate=\angle, black] (x) circle (1.5pt) node[below right] {$\vec{x} \in \mathcal{K}$};

    \draw[rotate=\angle, dashed, thick] (x') -- (x);
  \end{tikzpicture}
  \caption{A sequence of $\gamma$-strict separating hyperplanes implies that any $\vx \in \cK$ is far from the boundary of the ellipsoid (depicted in blue), assuming that the closest point outside the ellipsoid, labeled $\vx'$, belongs to $\cX$. \Cref{lemma:closeinters} shows how to make this argument when $\vx' \notin \cX$ by considering instead a point $\vz \in \cX$ that is close to $\vx'$.}
  \label{fig:ellipsoid-margin} 
\end{figure}

We proceed with the proof of this lemma. We first determine whether $\va \in \cX$ is an $\epsilon$-SVI solution; this can be done in polynomial time by invoking an optimization oracle for $\cX$ (\Cref{def:opt}). If so, the algorithm can terminate since $\va$ is an $\epsilon$-SVI solution. Otherwise, we define $\vatil \in \Q^d$ per the descent step $\proj_{\cX}(\va - \eta F(\va))$; such $\vatil$ guarantees the following, establishing~\cref{lemma:semiseparation} (since $\|F(\vatil) \| \leq B$).

\begin{lemma}
    \label{lemma:strict1}\label{lem:progress}
    Suppose that $\va \in \cX$ is not an $\epsilon$-SVI solution. If $\vatil \defeq \proj_{\cX}(\va - \eta F(\va))$ with $\eta = 1/(2L)$, then $\langle F(\vatil), \va - \vx \rangle \geq \gamma$ for any $\vx \in \cX$ that satisfies the Minty VI~\eqref{eq:MVI}, where $\gamma \defeq \epsilon^2 L / (B + 4 R L)^2$. Furthermore, $\langle F(\tilde \va), \va - \tilde \va \rangle \ge \gamma$.
\end{lemma}

\begin{proof}
    By the first-order optimality conditions, we have
    \begin{equation*}
    \left\langle F(\va) + \frac{1}{\eta} ( \vatil - \va ), \va - \vatil \right\rangle \geq 0, 
    \end{equation*}
    which in turn implies that
    \begin{equation}
        \label{eq:first-opt}
        \langle F(\va), \va - \vatil \rangle \geq \frac{1}{\eta} \| \va - \vatil \|^2.
    \end{equation}
    Moreover, for any MVI solution $\vx \in \cX$,
    \begin{align}
        \langle F(\vatil), \va - \vx \rangle &= \langle F(\vatil), \vatil - \vx \rangle + \langle F(\vatil), \va - \vatil \rangle \notag \\
        &\geq \langle F(\vatil), \va - \vatil \rangle = \langle F(\va), \va - \vatil \rangle + \langle F(\vatil) - F(\va), \va - \vatil \rangle  \label{align:Mintyatil} \\
        &\geq \frac{1}{\eta} \|\va - \vatil \|^2 - \| F(\vatil) - F(\va) \| \|\va - \vatil \| \label{align:first} \\
        &\geq \frac{1}{\eta} \|\va - \vatil \|^2 - L \|\va - \vatil \|^2 \label{align:Lips} \\
        &\geq \frac{1}{2\eta} \| \va - \vatil \|^2,\label{align:final}
    \end{align}
where~\eqref{align:Mintyatil} uses the fact that $\vx$ satisfies~\eqref{eq:MVI}; \eqref{align:first} follows from~\eqref{eq:first-opt} and Cauchy-Schwarz; and~\eqref{align:Lips} uses the fact that $F$ is $L$-Lipschitz continuous. Finally, using again the first-order optimality conditions, we have that for any $\vx \in \cX$,
\begin{equation}
    \label{eq:fo-start}
    \left\langle F(\va) + \frac{1}{\eta} (\vatil - \va), \vx - \vatil \right\rangle \geq 0 \implies \langle F(\va), \vx - \va \rangle + \langle F(\va), \va - \vatil \rangle + \frac{1}{\eta} \langle \vatil - \va, \vx - \vatil \rangle \geq 0.
\end{equation}
But $\va \in \cX$ is not an $\epsilon$-SVI solution, which implies that there exists $\vx \in \cX$ such that $\langle F(\va), \vx - \va \rangle < - \epsilon$. So, continuing from~\eqref{eq:fo-start},
\begin{equation*}
    \epsilon < \langle F(\va), \va - \vatil \rangle + \frac{1}{\eta} \langle \vatil - \va, \vx - \vatil \rangle \leq B \|\va - \vatil \| + \frac{2R}{\eta} \|\va - \vatil \|,
\end{equation*}
which in turn yields
\begin{equation*}
    \|\va - \vatil \| \geq \left( B + 4 L R \right)^{-1} \epsilon.
\end{equation*}
Combining with~\eqref{align:final}, the proof follows.
\end{proof}


\paragraph{The algorithm} We are now ready to describe our main construction, given as~\Cref{algo:optimisticellipsoid}. It is based on the central-cut ellipsoid we saw in~\Cref{algo:centalcut}. In every iteration, it checks whether the center of the current ellipsoid is an $\epsilon$-SVI solution. If so, the algorithm can terminate (\Cref{line:terminate}). Otherwise, if the center of the current ellipsoid lies in $\cX$, it proceeds by producing a $\gamma$-strict separating hyperplane with respect to the set of MVIs (\Cref{line:strict1,line:strict2})---by taking an extra-gradient step per~\Cref{lemma:strict1}. If the center does not belong to $\cX$, it suffices to invoke the separation oracle of $\cX$ (\Cref{line:sepX}). The algorithm continues by updating the ellipsoid (\Cref{line:update}). 

Having analyzed our semi-separation oracle (\Cref{lemma:semiseparation}), we conclude the analysis by showing that the number of iterations prescribed in~\Cref{line:newT}, which is polynomial in all relevant parameters, suffices to identify an $\epsilon$-SVI solution.

\begin{algorithm}[!ht]
\SetAlgoLined
\KwInput{ 
\begin{itemize}[noitemsep]
    \item Oracle access to a convex, compact set $\cX \subseteq \cB_R(\vec{0})$ in isotropic position;
    \item oracle access to $F : \cX \to \R^d$ satisfying~\Cref{ass:polyeval};
    \item rational $\epsilon > 0$.
\end{itemize}
}
\KwOutput{An $\epsilon$-SVI solution per~\Cref{def:VI}.}
\BlankLine
Set the strictness parameter $\gamma \defeq \epsilon^2 L / (B + 4 R L)^2$\;
Set the termination volume $v \defeq r^d/d^{d}$, where $r \defeq \gamma/(16 R B)$\;\label{line:v-r}
Set the maximum number of iterations as $T \defeq \lceil 5 d \log(1/v) + 5d^2 \log(2R) \rceil$\;\label{line:newT}
Set the precision parameter as $p \defeq 8 T$\;
Initialize the ellipsoid $\cE^{(0)} \defeq \cE^{(0)}(\mat{A}^{(0)}, \vec{a}^{(0)})$ as $\vec{a}^{(0)} \defeq \vec{0}$ and $\mat{A}^{(0)} \defeq R^2 \mat{I}_{d}$\;
\For{$t = 0, \dots, T-1$}{
    \If{$\vec{a}^{(t)}$ is an $\epsilon$-SVI solution}{
        \Return{$\vec{a}^{(t)}$}\;\label{line:terminate}
    }
    \Else{
        \If{ $\va^{(t)} \notin \cX$}{
            Let $\vec{c} \in \Q^d$ be a hyperplane separating $\va^{(t)}$ from $\cX$\;\label{line:sepX}
        }
        \Else{
        Compute $\vatil^{(t)} \defeq \proj_{\cX}( \va^{(t)} - \eta F(\va^{(t)}) )$, where $\eta \defeq 1/(2L)$\;\label{line:strict1}
        Set $\vec{c} \defeq F(\vatil^{(t)})/\|F(\vatil^{(t)})\|$\;\label{line:strict2}
        }
        Update \label{line:update}
        $$\vec{a}^{(t+1)} \approx_p \vec{a}^{(t)} - \frac{1}{d+1} \frac{\mat{A}^{(t)} \vec{c} }{ \sqrt{ \langle \vec{c}, \mat{A}^{(t)} \vec{c} \rangle } }  \text{ and } \mat{A}^{(t+1)} \approx_p \frac{2d^2 + 3}{2 d^2} \left( \mat{A}^{(t)} - \frac{2}{d+1} \frac{ \mat{A}^{(t)} \vec{c} \vec{c}^\top \mat{A}^{(t)}}{ \langle \vec{c}, \mat{A}^{(t)} \vec{c} \rangle } \right).$$
    }
}
\BlankLine
\Return{\textsf{``there are no MVI solutions''} }\;
\caption{$\optellipsoid$}
\label{algo:optimisticellipsoid}
\end{algorithm}

\subsubsection{Putting everything together}

The set of MVIs, denoted by $\cK \neq \emptyset$,
is generally not fully dimensional; nevertheless, by virtue of having a \emph{strict} separating hyperplane throughout the execution of~\Cref{algo:optimisticellipsoid} (whenever the center of the ellipsoid belongs to $\cX$), we will show that the volume of the ellipsoid can indeed be used as a yardstick to track the progress of the algorithm. The basic idea is that every axis of the ellipsoid needs to have a non-trivial length (\Cref{fig:ellipsoid-margin})---as dictated by the strictness parameter $\gamma$, thereby implying that the volume of the ellipsoid cannot shrink too much. Formally, our proof of correctness will show the following.

\begin{lemma}
    \label{lemma:largevol}
    Suppose that $\cK \neq \emptyset$---that is, the Minty condition (\Cref{ass:Minty}) holds. For any $t$ during the execution of~\Cref{algo:optimisticellipsoid}, the ellipsoid $\cE^{(t)}$ contains a (Euclidean) ball of radius $r \defeq \gamma/(16 R B)$, where $\gamma > 0$ is the strictness parameter per~\Cref{lemma:strict1}.
\end{lemma}

This will be established as part of the proof of~\Cref{theorem:mainMinty-precise} below; we state~\Cref{lemma:largevol} to highlight the key invariance of the analysis. In this context, we are ready to complete the proof of correctness of~\Cref{algo:optimisticellipsoid}, summarized in the theorem below.

\begin{theorem}
    \label{theorem:mainMinty-precise}
    Let $\cX$ be a convex and compact set in isotropic position to which we have oracle access, $F : \cX \to \R^d$ a mapping satisfying~\Cref{ass:polyeval}, and $\epsilon \in \Q_{>0}$ a rational number. Under the Minty condition (\Cref{ass:Minty}), \Cref{algo:optimisticellipsoid} can be implemented in time $\poly(d, \log(B/\epsilon), \log L)$ and returns an $\epsilon$-SVI solution to $\VI(\cX, F)$.
\end{theorem}

\begin{proof}
    That~\Cref{algo:optimisticellipsoid} can be implemented in time $\poly(d, \log(B/\epsilon), \log L)$ is immediate. We thus focus on proving correctness. For the sake of contradiction, suppose that the algorithm never identified an $\epsilon$-SVI solution. The volume of a $d$-dimensional Euclidean ball with radius $r > 0$ is given by
\begin{equation*}
    \frac{\pi^{d/2}}{\Gamma( \frac{d}{2} + 1 )} r^d > \frac{1}{d^d} r^d,
\end{equation*}
where $\Gamma(\cdot)$ is the gamma function. By our choice of parameters in~\Cref{line:v-r,line:newT} and~\Cref{theorem:basicellipsoid}, it follows that the short axis of the $T$th ellipsoid will have length strictly smaller than $r = \gamma/(16 R B)$. Let $\vx \in \cX$ be any point inside the final ellipsoid and $\vc$ the unit vector in the direction of the short axis of the ellipsoid. We will show that $\vx$ strictly violates the MVI constraint. Thus, assuming $\cK \neq \emptyset$, this will imply that the algorithm must have terminated at some earlier iteration with an $\epsilon$-SVI solution, establishing~\Cref{lemma:largevol}.

Let $\cZ$ be the union of the two $(d-1)$-dimensional disks of points $\vz$ lying in the planes $\ip{\vc, \vz - \vx} = \pm 2r$, and within $r' := \gamma/(4B)$ of $\vx$. That is, 
$$
    \cZ = \qty{ \vz \in\R^d : |\ip{\vc, \vz - \vx}| = 2r, \norm{\vx - \vz} \le r'}.
$$
\begin{claim}
    \label{lemma:closeinters}
$\cZ$ intersects $\cX$. 
\end{claim}
\begin{proof}
Since $\cX$ contains a ball of radius $1$, there must be a point $\vy \in \cX$ such that $\abs{\ip{\vc, \vy - \vx}} = 1$. Assume $\ip{\vc, \vy - \vx} = 1$ (The case $\ip{\vc, \vy - \vx} = -1$ is symmetric). Let $\vz$ be the point on line segment $[\vx, \vy]$ such that $\ip{\vc, \vz - \vx} = 2r$, that is, let $\vz = \vx + 2r(\vy - \vx)$. Since $\cX$ is convex, we have $\vz \in \cX$. Since $\cX \subseteq \cB_R(0)$, we have $\norm{\vy - \vx} \le 2R$. Thus, $\norm{\vz - \vx} \le 2r \cdot 2R = \gamma/(4B) = r'$, so $\vz \in \cZ$.
\end{proof}

\begin{lemma}
    \label{lemma:certificate}
    If the algorithm fails to return an $\epsilon$-SVI solution after $T$ rounds, where $T$ is as defined in~\Cref{line:newT}, any point $\vx \in \cX$ strictly violates the MVI constraint. In particular, there exists a timestep $t$ such that $\tilde \va^{(t)} \in \cX$ and $\langle F(\tilde \va^{(t)}), \tilde \va^{(t)} - \vx \rangle \le - \gamma/2$.
\end{lemma}

\begin{proof}
    It suffices to consider $\vx \in \cE^{(T)}$. Let $\vz \in \cZ \cap \cX$, which must exist by~\Cref{lemma:closeinters}. By definition, we have $\norm{\vz - \vx} \le r' < \gamma/(2B)$. 
    Since the short axis of the final ellipsoid has radius less than $r$ and $| \ip{\vec{c}, \vz - \vx} | = 2r$, it follows that $\vz$ is not in the ellipsoid. Thus, at some point, there must have been a separating hyperplane that has $\vz$ on the opposite side. This separating hyperplane could not have come from the separation oracle of $\cX$, because $\vz \in \cX$. Thus, it must have come from an extra-gradient step, \ie, the separating hyperplane must have the form 
    $$
        \ip{F(\tilde \va^{(t)}), \vz - \va^{(t)}} \ge 0.
    $$
    for some timestep $t$.
    From \Cref{lemma:strict1} and the construction of $\tilde\va^{(t)}$, we also have $\ip{F(\tilde \va^{(t)}), \va^{(t)} - \tilde \va^{(t)}} \ge \gamma.$ Thus, we have
    \begin{align*}
        \ip{F(\tilde\va^{(t)}), \vx - \tilde\va^{(t)}} &= \ip{F(\tilde\va^{(t)}), \vz - \va^{(t)}} + \ip{F(\tilde\va^{(t)}), \va^{(t)} - \tilde\va^{(t)}} + \ip{F(\tilde\va^{(t)}), \vx - \vz} 
        \\&\ge \gamma - B \cdot \frac{\gamma}{2B} \ge \frac{\gamma}{2}. \tag*\qedhere
    \end{align*} 
\end{proof}
This concludes the proof of~\Cref{lemma:certificate}, and~\Cref{theorem:mainMinty-precise} follows.
\end{proof}

\subsection{Weak Minty condition}
\label{sec:weakMinty}

Moving forward, we first observe that the previous analysis can be extended beyond the Minty condition (\Cref{ass:Minty}). In particular, we lean on the more permissive assumption put forward by~\citet{Diakonikolas21:Efficient}; we will shortly discuss how it relates to other conditions. Below, we use the notation $\svigap : \cX \ni \vx \mapsto \max_{\vx' \in \cX} \langle F(\vx), \vx - \vx' \rangle$, so that~\Cref{def:VI} can be equivalently expressed as $\svigap(\vx) \leq \epsilon$.

\begin{definition}[Weak Minty]
    \label{def:weakMinty}
    A problem $\VI(\cX, F)$ satisfies the \emph{$\rho$-weak Minty condition}, with $\rho > 0$, if there exists $\vx \in \cX$ such that
    \begin{equation}
        \label{eq:weakMinty}
        \langle F(\vx'), \vx' - \vx \rangle \geq - \rho (\svigap(\vx'))^2 \quad  \forall \vx' \in \cX.
    \end{equation}
\end{definition}
\citet{Diakonikolas21:Efficient} focused on the unconstrained setting, positing that the right-hand side of~\eqref{eq:weakMinty} instead reads $- \rho \|F(\vx') \|^2$ (we have removed a factor of $2$ compared to the definition given by~\citealp{Diakonikolas21:Efficient}, which amounts to simply rescaling $\rho$); \Cref{def:weakMinty} can be seen as a natural counterpart of that condition to the constrained setting. For the unconstrained setting, this is weaker than another well-studied condition; namely, $F$ is called \emph{$(-\rho)$-comonotone}~\citep{Bauschke21:Generalized} (see also cohypomonotone operators per~\citealp{Combettes04:Proximal}) if for all $\vx, \vx' \in \R^d$, $\langle F(\vx) - F(\vx'), \vx - \vx' \rangle \geq - \rho \|F(\vx) - F(\vx') \|^2$; since $F(\vx) = 0$ for any SVI solution (in the unconstrained setting), the condition of~\citet{Diakonikolas21:Efficient} is weaker.

In this context, the purpose of this subsection is to show that our previous analysis can be readily extended under~\Cref{def:weakMinty}---in place of~\Cref{ass:Minty}. For completeness, we begin with a simple claim showing that the set of points satisfying the $\rho$-weak Minty condition is convex.

\begin{claim}
    \label{claim:convexity}
    Let $\cK$ be the set of solutions to~\eqref{eq:weakMinty}. $\cK$ is a convex set.
\end{claim}

\begin{proof}
    Suppose $\cK \neq \emptyset$. Let $\vx_1, \vx_2 \in \cK$ and $\lambda \in [0, 1]$. We need to show that for all $\vx' \in \cX$,
    \begin{equation*}
    \lambda \langle F(\vx'), \vx' - \vx_1 \rangle + (1 - \lambda) \langle F(\vx'), \vx' - \vx_2 \rangle \geq - \rho ( \svigap(\vx') )^2.
    \end{equation*}
    This follows since $\langle F(\vx'), \vx' - \vx_1 \rangle \geq - \rho (\svigap(\vx'))^2 $ and $\langle F(\vx'), \vx' - \vx_2 \rangle \geq - \rho (\svigap(\vx'))^2$.
\end{proof}

Now, to show that an $\epsilon$-SVI solution can be computed in polynomial time even under the weak Minty condition, for small enough $\rho$, it suffices to adjust~\Cref{lemma:strict1} according to the statement below.

\begin{lemma}
    \label{lemma:strictweakMinty}
    Suppose that $\va \in \cX$ is not an $\epsilon$-SVI solution. Let $\vatil \defeq \proj_\cX(\va - \eta F(\va))$, with $\eta = 1/(2L)$, and suppose that $\rho \leq 1/(72 R^2 L)$. Then, $\langle F(\vatil), \va - \vx \rangle \geq \epsilon^2 L / (2 (B + 4 R L)^2)$ for any $\vx \in \cX$ that satisfies the $\rho$-weak Minty condition of~\Cref{def:weakMinty}.
\end{lemma}

The proof is similar to that of~\Cref{lemma:strict1}.

\begin{proof}[Proof of~\Cref{lemma:strictweakMinty}]
By the first-order optimality conditions, we have
\begin{equation*}
\left\langle F(\va) + \frac{1}{\eta} ( \vatil - \va ), \va - \vatil \right\rangle \geq 0,
\end{equation*}
which in turn implies that
\begin{equation}
\label{eq:first-opt-weak}
\langle F(\va), \va - \vatil \rangle \geq \frac{1}{\eta} \| \va - \vatil \|^2.
\end{equation}
The first-order optimality conditions also imply, for every $\vx \in \cX$,
\begin{align*}
\langle F(\vatil), \vatil - \vx \rangle
&\leq \left\langle F(\vatil) - F(\va) + \frac{1}{\eta} (\va - \vatil), \vatil - \vx \right\rangle \\
&\leq 2R \left(L + \frac{1}{\eta} \right) \|\va - \vatil\| \\
&= 6RL \|\va - \vatil\|.
\end{align*}
As a result,
\begin{equation*}
\svigap(\vatil) \leq 6RL \|\va - \vatil\|.
\end{equation*}
Moreover, for any $\vx \in \cX$ that satisfies the $\rho$-weak Minty condition of~\Cref{def:weakMinty},
\begin{align}
\langle F(\vatil), \va - \vx \rangle &= \langle F(\vatil), \vatil - \vx \rangle + \langle F(\vatil), \va - \vatil \rangle \notag \\
&\geq - \rho (\svigap(\vatil))^2 + \langle F(\vatil), \va - \vatil \rangle \notag \\
&= - \rho (\svigap(\vatil))^2 + \langle F(\va), \va - \vatil \rangle + \langle F(\vatil) - F(\va), \va - \vatil \rangle \label{align:weakMintyatil} \\
&\geq - \rho (\svigap(\vatil))^2 + \frac{1}{\eta} \|\va - \vatil \|^2 - \| F(\vatil) - F(\va) \| \|\va - \vatil \| \label{align:first-weak} \\
&\geq - \rho (\svigap(\vatil))^2 + \frac{1}{\eta} \|\va - \vatil \|^2 - L \|\va - \vatil \|^2 \label{align:Lips-weak} \\
&\geq (L - 36 \rho R^2 L^2) \| \va - \vatil \|^2 \notag \\
&\geq (L/2) \| \va - \vatil \|^2,\label{align:final-weak}
\end{align}
where~\eqref{align:weakMintyatil} uses the fact that $\vx$ satisfies~\Cref{def:weakMinty}; \eqref{align:first-weak} follows from~\eqref{eq:first-opt-weak} and Cauchy-Schwarz; and~\eqref{align:Lips-weak} uses the fact that $F$ is $L$-Lipschitz continuous. Finally, using again the first-order optimality conditions, we have that for any $\vx \in \cX$,
\begin{equation}
\label{eq:fo-start-weak}
\left\langle F(\va) + \frac{1}{\eta} (\vatil - \va), \vx - \vatil \right\rangle \geq 0 \implies \langle F(\va), \vx - \va \rangle + \langle F(\va), \va - \vatil \rangle + \frac{1}{\eta} \langle \vatil - \va, \vx - \vatil \rangle \geq 0.
\end{equation}
But $\va \in \cX$ is not an $\epsilon$-SVI solution, which implies that there exists $\vx \in \cX$ such that $\langle F(\va), \vx - \va \rangle < - \epsilon$. So, continuing from~\eqref{eq:fo-start-weak},
\begin{equation*}
\epsilon < \langle F(\va), \va - \vatil \rangle + \frac{1}{\eta} \langle \vatil - \va, \vx - \vatil \rangle \leq B \|\va - \vatil \| + \frac{2R}{\eta} \|\va - \vatil \|,
\end{equation*}
which in turn yields
\begin{equation*}
\|\va - \vatil \| \geq \left( B + 4 L R \right)^{-1} \epsilon.
\end{equation*}
Combining with~\eqref{align:final-weak}, the proof follows.
\end{proof}

\Cref{lemma:strictweakMinty} yields a strict separating hyperplane under the condition of~\Cref{theorem:weakMinty}. The rest of the argument is analogous to~\Cref{theorem:mainMinty-precise}.

\begin{theorem}
    \label{theorem:weakMinty}
    Let $\cX$ be a convex and compact set in isotropic position to which we have oracle access, $F : \cX \to \R^d$ a mapping satisfying~\Cref{ass:polyeval}, and $\epsilon \in \Q_{>0}$ a rational number. Under the $\rho$-weak Minty condition (\Cref{def:weakMinty}) with $\rho \leq 1/(72 R^2 L)$, there is a $\poly(d, \log(B/\epsilon), \log L)$-time algorithm that returns an $\epsilon$-SVI solution to $\VI(\cX, F)$.
\end{theorem}

\subsection{Relaxing continuity}
\label{sec:noncont}

Another natural question is whether one can relax the assumption that $F$ is Lipschitz continuous. Under an additional assumption---namely, a generalization of quasar-convexity (\Cref{def:quasarconvex}) based on~\Cref{def:smoothVI}---we show that this is indeed possible by obviating the need for the extra-gradient step in~\Cref{algo:optimisticellipsoid} (\Cref{lemma:strict1}), which is where Lipschitz continuity was used. In particular, the following lemma shows that, under a suitable strengthening of the Minty condition, $F(\va)$ already yields a strict separating hyperplane. We again call attention to the fact that smoothness per~\Cref{def:smoothVI}, which accords with the terminology of~\citet{Roughgarden15:Intrinsic}, is different from the usual notion of a smooth function in optimization.

\begin{lemma}
    \label{lemma:strict-quasar}
    Let $\VI(\cX, F)$ be a $(\lambda, \lambda-1)$-smooth VI problem (\Cref{def:smoothVI}) with respect to $Q$ and a global maximizer $\vx \in \cX$ thereof. If $\va \in \cX$ is such that $Q(\va) \leq Q(\vx) - \epsilon$, then
    \begin{equation*}
        \langle F(\va), \va - \vx \rangle \geq \lambda \epsilon.
    \end{equation*}
\end{lemma}

\begin{proof}
    By~\Cref{def:smoothVI} (with $\nu = \lambda - 1$), we have $\langle F(\va), \va - \vx \rangle \geq \lambda ( Q(\vx) - Q(\va)) \geq \lambda \epsilon$.
\end{proof}

In this case, we define $\cK$ to contain any point $\vx \in \cX$ such that $\langle F(\vx'), \vx' - \vx \rangle \geq \lambda ( \opt - Q(\vx'))$ for all $\vx' \in \cX$, where $\opt$ is the maximum value attained by $Q$. This is still a convex set. Further, $\cK \neq \emptyset$---in particular, this means that the Minty condition is satisfied.

We give the overall construction in~\Cref{algo:ellipsoidnoncont}. Compared to~\Cref{algo:optimisticellipsoid}, we call attention to the following differences. First, we do not check in each iteration $t$ whether the center of the current ellipsoid $\va^{(t)}$ satisfies $Q(\va^{(t)}) \geq \opt - \epsilon$; we do not know the value of $\opt$, so instead the algorithm eventually returns the point attaining the highest value throughout the execution (\Cref{line:returnmax}). The second difference is that $F(\va^{(t)})$ (\Cref{line:quasar}) already yields a strict separating hyperplane (\Cref{lemma:strict-quasar}), so there is no need for an extra-gradient step.

By our previous analysis in~\Cref{sec:ouralgo} and~\Cref{lemma:strict-quasar}, it follows that there must be some iteration such that $Q(\va^{(t)}) \geq \opt - \epsilon$, for otherwise the underlying promise---namely, $(\lambda, \lambda-1)$-smoothness per~\Cref{def:smoothVI}---would be violated; \Cref{line:returnmax} returns such a point. We summarize the guarantee of~\Cref{algo:ellipsoidnoncont} below.

\begin{theorem}
    \label{theorem:quasar-precise}
    Let $\cX$ be a convex and compact set in isotropic position to which we have oracle access, $F : \cX \to \R^d$ a mapping satisfying~\Cref{item:Bparam,item:polyeval} of~\Cref{ass:polyeval}, and $\epsilon \in \Q_{>0}$ a rational number. If $\VI(\cX, F)$ is $(\lambda, \lambda - 1)$-smooth (\Cref{def:smoothVI}) with respect to $Q$, \Cref{algo:ellipsoidnoncont} can be implemented in time $\poly(d, \log(B/\epsilon), \log (1/\lambda))$ and returns a point $\va \in \cX$ such that $Q(\va) \geq \max_{\vx \in \cX} Q(\vx) - \epsilon$.
\end{theorem}

\begin{algorithm}[!ht]
\SetAlgoLined
\KwInput{ 
\begin{itemize}[noitemsep]
    \item Oracle access to a convex, compact set $\cX \subseteq \cB_R(\vec{0})$ in isotropic position;
    \item oracle access to $F : \cX \to \R^d$ satisfying~\Cref{item:polyeval,item:Bparam} of~\Cref{ass:polyeval};
    \item oracle access to $Q : \cX \to \R$ such that $\langle F(\vx'), \vx' - \vx \rangle \geq \lambda ( Q(\vx) - Q(\vx'))$ for all $\vx' \in \cX$, where $\vx \in \cX$ is some global maximum of $Q$ and $\lambda \in \Q \cap (0, 1]$;
    \item rational $\epsilon > 0$.
\end{itemize}
}
\KwOutput{A point $\va \in \cX$ such that $Q(\va) \geq \max_{\vx \in \cX} Q(\vx) - \epsilon$.}
\BlankLine
Set the strictness parameter $\gamma \defeq 
\lambda \epsilon$\;
Set the termination volume $v \defeq r^d/d^{d}$, where $r \defeq \gamma/(16 R B)$\;
Set the maximum number of iterations as $T \defeq \lceil 5 d \log(1/v) + 5d^2 \log(2R) \rceil$\;
Set the precision parameter as $p \defeq 8 T$\;
Initialize the ellipsoid $\cE^{(0)} \defeq \cE^{(0)}(\mat{A}^{(0)}, \vec{a}^{(0)})$ as $\vec{a}^{(0)} \defeq \vec{0}$ and $\mat{A}^{(0)} \defeq R^2 \mat{I}_{d}$\;
\For{$t = 0, \dots, T-1$}{
    
    \If{ $\va^{(t)} \notin \cX$}{
        Let $\vec{c} \in \Q^d$ be a hyperplane separating $\va^{(t)}$ from $\cX$\;
    }
    \Else{
    Set $\vec{c} \defeq F(\va^{(t)})/\|F(\va^{(t)})\|$\;\label{line:quasar}
    }
    Update
    $$\vec{a}^{(t+1)} \approx_p \vec{a}^{(t)} - \frac{1}{d+1} \frac{\mat{A}^{(t)} \vec{c} }{ \sqrt{ \langle \vec{c}, \mat{A}^{(t)} \vec{c} \rangle } }  \text{ and } \mat{A}^{(t+1)} \approx_p \frac{2d^2 + 3}{2 d^2} \left( \mat{A}^{(t)} - \frac{2}{d+1} \frac{ \mat{A}^{(t)} \vec{c} \vec{c}^\top \mat{A}^{(t)}}{ \langle \vec{c}, \mat{A}^{(t)} \vec{c} \rangle } \right).$$
    
}
\BlankLine
Choose $t^* \in \argmax \{ Q(\va^{(t)}) : 0 \leq t < T \text{ and } \va^{(t)} \in \cX \}$\;
\Return{$\va^{(t^*)}$}\;\label{line:returnmax}
\caption{Ellipsoid for $(\lambda, \lambda-1)$-smooth VIs.}
\label{algo:ellipsoidnoncont}
\end{algorithm}

\subsection{A certificate of MVI infeasibility: strict EVIs}
\label{sec:strict-EVIs}

Computing $\epsilon$-SVI solutions is generally \PPAD-hard, so certain VI problems violate the Minty conditions (and relaxations thereof per~\Cref{sec:weakMinty}). In such cases, the transcript of~\Cref{algo:optimisticellipsoid} itself provides a certificate of MVI infeasibility. In fact, as we observe in this subsection, the certificate of infeasibility can be expressed as an expected VI solution (in the sense of~\Cref{def:EVIs}) \emph{with negative gap}; the mere existence of such an object implies that the Minty condition is violated due to the duality between EVIs and MVIs  (\Cref{prop:duality}).

To make this argument formal, we first observe that if~\Cref{algo:optimisticellipsoid} fails to identify an $\epsilon$-SVI solution, \Cref{lemma:certificate} implies that
\begin{equation*}
    \min_{\vx \in \cX} \max_{t \in [T]}\ip*{F(\tilde \va^{(t)}), \vx - \tilde \va^{(t)}} \ge \frac{\gamma}{2},
\end{equation*}
which, by strong duality, is equivalent to
\begin{equation}
    \label{eq:strictEVI}
    \max_{\mu \in \Delta(T)} \min_{\vx \in \cX} \E_{t \sim \mu} \ip*{F(\tilde \va^{(t)}), \vx - \tilde \va^{(t)}} \ge \frac{\gamma}{2}.
\end{equation}
In particular, a distribution $\mu$ over $[T]$ that satisfies~\eqref{eq:strictEVI} is, by definition, a $\gamma/2$-strict EVI solution, a certificate of MVI infeasibility. Further, such a distribution can be computed in polynomial time (\emph{e.g.},~\citealp{Jiang20:Improved}): \eqref{eq:strictEVI} is a bilinear saddle-point problem over convex sets that admit efficient oracle access. The crucial point here is that the maximization in~\eqref{eq:strictEVI} simply optimizes over the mixing weights with respect to a fixed support of size $T$, which is polynomial. This approach resembles the celebrated ``ellipsoid against hope'' algorithm of~\citet{Papadimitriou08:Computing}, which applies the ellipsoid algorithm on a certain program guaranteed to be infeasible; similarly to our case, the certificate of infeasibility produces a correlated equilibrium.

The resulting construction is~\Cref{algo:sviorstrictevi}. It is almost the same as~\Cref{algo:optimisticellipsoid}; the key difference is that, when the algorithm fails to identify an $\epsilon$-SVI solution, the last step (\Cref{item:reoptimize}) performs an additional optimization to compute a $\gamma/2$-strict EVI solution. The guarantee of~\Cref{algo:sviorstrictevi} is given below; it is the precise version of~\Cref{theorem:maineither-informal}.

\begin{theorem}
    \label{theorem:maineither-precise}
    Let $\cX$ be a convex and compact set in isotropic position to which we have oracle access, a mapping $F : \cX \to \R^d$ satisfying~\Cref{ass:polyeval}, and a rational number $\epsilon \in \Q_{>0}$. \Cref{algo:sviorstrictevi} can be implemented in time $\poly(d, \log(B/\epsilon), \log L)$ and returns either
    \begin{itemize}
        \item   an $\epsilon$-SVI solution to $\VI(\cX, F)$ or
        \item an $\Omega_\epsilon(\epsilon^2)$-\emph{strict} EVI solution to $\VI(\cX, F)$.
    \end{itemize}
\end{theorem}

\begin{algorithm}[!ht]
\SetAlgoLined
\KwInput{ 
\begin{itemize}[noitemsep]
    \item Oracle access to a convex, compact set $\cX \subseteq \cB_R(\vec{0})$ in isotropic position;
    \item oracle access to $F : \cX \to \R^d$ satisfying~\Cref{ass:polyeval};
    \item rational $\epsilon > 0$.
\end{itemize}
}
\KwOutput{An $\epsilon$-SVI solution per~\Cref{def:VI} \emph{or} a strict EVI solution per~\Cref{def:EVIs}.}
\BlankLine
Set the strictness parameter $\gamma \defeq \epsilon^2 L / (B + 4 R L)^2$\;
Set the termination volume $v \defeq r^d/d^{d}$, where $r \defeq \gamma/(16 R B)$\;
Set the maximum number of iterations as $T \defeq \lceil 5 d \log(1/v) + 5d^2 \log(2R) \rceil$\;
Set the precision parameter as $p \defeq 8 T$\;
Initialize the ellipsoid $\cE^{(0)} \defeq \cE^{(0)}(\mat{A}^{(0)}, \vec{a}^{(0)})$ as $\vec{a}^{(0)} \defeq \vec{0}$ and $\mat{A}^{(0)} \defeq R^2 \mat{I}_{d}$\;
\For{$t = 0, \dots, T-1$}{
    \If{$\vec{a}^{(t)}$ is an $\epsilon$-SVI solution}{
        \Return{$\vec{a}^{(t)}$}\;
    }
    \Else{
        \If{ $\va^{(t)} \notin \cX$}{
            Let $\vec{c} \in \Q^d$ be a hyperplane separating $\va^{(t)}$ from $\cX$\;
        }
        \Else{
        Compute $\Q^d \ni \vatil^{(t)} \defeq \proj_{\cX}( \va^{(t)} - \eta F(\va^{(t)}) )$, where $\eta \defeq 1/(2L)$\;
        Set $\vec{c} \defeq F(\vatil^{(t)})/\|F(\vatil^{(t)})\|$\;
        }
        Update
        $$\vec{a}^{(t+1)} \approx_p \vec{a}^{(t)} - \frac{1}{d+1} \frac{\mat{A}^{(t)} \vec{c} }{ \sqrt{ \langle \vec{c}, \mat{A}^{(t)} \vec{c} \rangle } }  \text{ and } \mat{A}^{(t+1)} \approx_p \frac{2d^2 + 3}{2 d^2} \left( \mat{A}^{(t)} - \frac{2}{d+1} \frac{ \mat{A}^{(t)} \vec{c} \vec{c}^\top \mat{A}^{(t)}}{ \langle \vec{c}, \mat{A}^{(t)} \vec{c} \rangle } \right).$$
    }
}
\BlankLine
\Return{$\argmax_{\mu \in \Delta([T])} \min_{\vx \in \cX} \E_{t \sim \mu} \ip*{F(\tilde \va^{(t)}), \vx - \tilde \va^{(t)}}$}\;\label{item:reoptimize}
\caption{$\sviorstrictevi$}
\label{algo:sviorstrictevi}
\end{algorithm}
\section{Two-player smooth concave games}
\label{sec:twoplayer}

This section provides a refinement of~\Cref{theorem:maineither-precise} in the context of two-player games. We begin by formally introducing this class of problems.

\begin{definition}[Two-player smooth concave games]
    A {\em two-player concave game} is given by two convex and compact strategy sets $\cX_1 \subseteq \R^{d_1}$, $\cX_2 \subseteq \R^{d_2}$, and two utility functions $u_1, u_2 : \cX_1 \times \cX_2 \to \R$, such that the utility of each player is concave (in the player's own strategy), differentiable, and $L$-smooth. Formally,  $u_1(\cdot, \vx_2) : \cX_1 \to \R$ is concave for every fixed $\vx_2$, the gradient operator $\grad_{\vx_1} u_1 : \cX_1 \times \cX_2 \to \R^{d_1}$ is $L$-Lipschitz continuous, and the symmetric guarantees hold for the second player as well.
\end{definition}
A {\em (pure) strategy profile} is a pair $(\vx_1, \vx_2) \in \cX_1 \times \cX_2$. A strategy profile is an {\em $\eps$-Nash equilibrium} if each player is $\eps$-best responding to the other player; that is,
\begin{equation*}
    \max_{\vx_1' \in \cX_1} u_1(\vx_1', \vx_2) -  u_1(\vx_1, \vx_2) \le\eps \qq{and}  \max_{\vx_2' \in \cX_2} u_2(\vx_1, \vx_2') - u_2(\vx_1, \vx_2) \le \eps.
\end{equation*}
A {\em correlated strategy profile} is a distribution $\mu \in \Delta(\cX_1\times\cX_2)$. A correlated strategy profile is a {\em $\eps$-coarse correlated equilibrium} (CCE) if no player can profit by more than $\eps$ using a {\em unilateral} deviation, that is,
\begin{equation*}
    \max_{\vx_1' \in \cX_1} \E_{(\vx_1, \vx_2)\sim\mu}[ u_1(\vx_1', \vx_2) -  u_1(\vx_1, \vx_2)] \le\eps \qq{and}  \max_{\vx_2' \in \cX_2} \E_{(\vx_1, \vx_2)\sim\mu}[u_2(\vx_1, \vx_2') - u_2(\vx_1, \vx_2)] \le \eps.
\end{equation*}
We will call a $(-\eps)$-CCE an $\eps$-{\em strict CCE}. In this section, we will show the following result for smooth concave games. 
\begin{theorem}\label{th:2p-nash-or-scce}
    Assume that $\cX_1$ and $\cX_2$ are in isotropic position and given by separation oracles. Assume also the gradient operators $\grad_{\vx_1} u_1 : \cX_1 \times \cX_2 \to \R^{d_1}$ and $\grad_{\vx_2} u_2 : \cX_1 \times \cX_2 \to \R^{d_2}$ satisfy \Cref{ass:polyeval}. Let $d = d_1 + d_2$. Then there exists a $\poly(d) \cdot \polylog(B, L, R, 1/\eps)$-time algorithm that outputs {\em either} an $\eps$-Nash equilibrium {\em or} an $\eps'$-strict CCE, where $\eps' \ge \eps^6 / \poly(B, L, R)$.
\end{theorem}
We first note that this result is {\em not} an immediate corollary of \Cref{theorem:maineither-precise}. $\eps$-Nash equilibria indeed correspond to $\eps$-SVI  solutions. However, as per the discussion in \Cref{sec:Mintyexamples}, \Cref{theorem:maineither-precise} would only give a strict {\em average} CCE (\Cref{def:acce}), not a strict CCE. Circumventing this problem therefore requires a few new insights.

First, we need to modify the SVI problem so that strict EVI solutions correspond to strict CCEs. Consider the set
\[
    \cP := \qty{ \mqty(\lambda_1 \\ \lambda_2 \\ \lambda_1\vx_1 \\ \lambda_2\vx_2) : \mqty(\lambda_1 \\ \lambda_2) \in \Delta(2), \vx_1 \in \cX_1, \vx_2 \in \cX_2} \subseteq \R^{2+d}.
\]
It is easy to see that $\cP$ is convex and compact, since $\cX_1$ and $\cX_2$ are. Moreover, consider the operator $G : \cP \to \R^{2+d}$ given by 
\[
    G\mqty(\lambda_1 \\ \lambda_2 \\ \lambda_1\vx_1 \\ \lambda_2\vx_2) := \mqty( \ip{\grad_{\vx_1} u_1(\vx_1, \vx_2), \vx_1} \\ \ip{\grad_{\vx_2} u_2(\vx_1, \vx_2), \vx_2}\\ -\grad_{\vx_1} u_1(\vx_1, \vx_2) \\ -\grad_{\vx_2} u_2(\vx_1, \vx_2)).
\]
We now immediately encounter our first problem: $G$ is not Lipschitz, or even well-defined, when $\lambda_1 = 0$ or $\lambda_1 = 1$. The solution to this problem is to restrict ourselves to the slightly smaller set
\[
    \cP_\alpha := \qty{ \mqty(\lambda_1 \\ \lambda_2 \\ \lambda_1\vx_1 \\ \lambda_2\vx_2) : \mqty(\lambda_1 \\ \lambda_2) \in \Delta(2) \cap \R_{\ge \alpha}^2 , \vx_1 \in \cX_1, \vx_2 \in \cX_2}
\]
where $\alpha \in (0, 1)$ is a constant to be selected later. Over this domain, $G : \cP_\alpha \to \R^{2+d}$ is indeed well-defined and bounded. Moreover, $G$ can be expressed as $G = G_1 \circ G_2$, where
\[
    G_1\mqty(\vx_1 \\ \vx_2) := \mqty( \ip{\grad_{\vx_1} u_1(\vx_1, \vx_2), \vx_1} \\ \ip{\grad_{\vx_2} u_2(\vx_1, \vx_2), \vx_2}\\ -\grad_{\vx_1} u_1(\vx_1, \vx_2) \\ -\grad_{\vx_2} u_2(\vx_1, \vx_2)), \qq{} G_2\mqty(\lambda_1 \\ \lambda_2 \\ \vv_1 \\ \vv_2) = \mqty(\vv_1 / \lambda_1 \\ \vv_2 / \lambda_2).
\]
Note that $G_1$ has Lipschitz constant $\poly(B, L, R)$, and $G_2$ has Lipschitz constant $O(R/\alpha^2)$, so their composition has Lipschitz constant $\poly(B, L, R)/\alpha^2$.

We start by noting that, for any $\vz = (\lambda_1, \lambda_2, \lambda_1\vx_1, \lambda_2\vx_2), \vz' = (\lambda_1', \lambda_2', \lambda_1'\vx_1', \lambda_2'\vx_2')\in \cP_\alpha$, we have
    \begin{align*}
        -\ip{G(\vz), \vz'} &= -\mqty(  \ip{\grad_{\vx_1} u_1(\vx_1, \vx_2), \vx_1} \\ \ip{\grad_{\vx_2} u_2(\vx_1, \vx_2), \vx_2} \\ -\grad_{\vx_1} u_1(\vx_1, \vx_2) \\ -\grad_{\vx_2} u_2(\vx_1, \vx_2))^\top \mqty(\lambda_1' \\ \lambda_2' \\ \lambda_1'\vx_1' \\ \lambda_2'\vx_2')
        \\&= \lambda_1' \ip{\grad_{\vx_1} u_1(\vx_1, \vx_2), \vx_1' - \vx_1} +\lambda_2' \ip{\grad_{\vx_2} u_2(\vx_1, \vx_2), \vx_2' - \vx_2}.
    \end{align*}
    That is, $-\ip{G(\vz), \vz'}$ is the weighted sum of deviation benefits for the two players if they deviate to $\vz'$ from $\vz$. In particular, $\ip{G(\vz), \vz} = 0$ for all $\vz$. 
\begin{lemma}
    For any $\alpha < 1$, any $3\alpha B R$-strict EVI solution to $\VI(\cP_\alpha, G)$ is an $\alpha B R$-strict CCE.
\end{lemma}
\begin{proof}
    Let $\mu \in \Delta(\cP_\alpha)$ be a $3\alpha BR$-strict EVI solution. Then, for any $\vz' = (\lambda_1', \lambda_2', \lambda_1'\vx_1', \lambda_2'\vx_2') \in \cP_\alpha$, we have
    \begin{align*}
        -3\alpha BR &\ge \E_{\vz \sim \mu} \ip{G(\vz), \vz - \vz'}
        = \lambda_1' \E_{\vz \sim \mu} \ip{\grad_{\vx_1} u_1(\vx_1, \vx_2), \vx_1' - \vx_1} + \lambda_2' \E_{\vz \sim \mu}\ip{\grad_{\vx_2} u_2(\vx_1, \vx_2), \vx_2' - \vx_2}.
    \end{align*}
    Since this holds for any $\vz'$, in particular it holds if we set $\vec\lambda' = (1-  \alpha, \alpha)$, which gives 
    \begin{align*}
        (1 - \alpha) \E_{\vz \sim \mu} \ip{\grad_{\vx_1} u_1(\vx_1, \vx_2), \vx_1' - \vx_1}  \le -3\alpha BR - \alpha \E_{\vz \sim \mu}\ip{\grad_{\vx_2} u_2(\vx_1, \vx_2), \vx_2' - \vx_2} \le -\alpha BR,
    \end{align*}
    where the final inequality follows from Cauchy-Schwarz, since $\norm{\grad_{\vx_2} u_2(\vx_1, \vx_2)} \le B$ and $\norm{\vx_2' - \vx_2} \le 2R$. Dividing by $1 - \alpha$ and noting that $1 - \alpha \le 1$ completes the proof. 
\end{proof}
It may now be tempting to try to run \Cref{algo:optimisticellipsoid} on $\VI(\cP_\alpha, G)$ with $\eps$ large enough to recover $3\alpha BR$-strict EVI solutions via \Cref{theorem:maineither-precise}. Unfortunately, this is impossible, since the Lipschitz constant $L'$ of $G$ scales as  $O_\alpha(1/\alpha^2)$. So, we would be required to take $\eps$ to be at least as large as a (game-dependent) constant, which defeats our purpose. To account for this, we directly modify the semi-separation oracle (\Cref{lemma:semiseparation}). Intuitively, the problematic case is when $\lambda_1$ or $\lambda_2$ is close to $0$. For intuition, let us discuss an extreme case, $\lambda_2 = 0$.\footnote{Since $\alpha > 0$, it is actually impossible for $\lambda_2 =0$ to arise in our algorithm; nonetheless, studying this case will be instructive for intuition.} Let $\va = (1, 0, \vx_1, \vec 0) \in \cP$. Our goal is to either (1) find a Nash equilibrium, or (2) separate $\va$ strictly from the set of Minty solutions to $\VI(\cP_\alpha, G)$. Consider the following process. Let $\vx_2'$ be a best response to $\vx_1$. If $\vx_1$ is also a best response to $\vx_2'$, then $(\vx_1, \vx_2')$ is a Nash equilibrium, and we are done. Otherwise, let $\vx_1' = \Pi_{\cX_1}(\vx_1 + \eta \grad_{\vx_1} u_1(\vx_1, \vx_2'))$. Then, any point of the form $\va' := (\lambda_1, \lambda_2, \lambda_1\vx_1', \lambda_2\vx_2')$ with $\lambda_1, \lambda_2 > 0$ certifies that $\va$ is not Minty: the Minty constraint induced by $\va'$ is 
\[
    \ip{G(\va'), \va' - \va} = \ip{\grad_{\vx_1} u_1(\vx_1', \vx_2'), \vx_1 - \vx_1'} \ge 0,
\]
but this constraint can be refuted by choosing the step size $\eta$ small enough that $\grad_{\vx_1} u_1(\vx_1', \vx_2') \approx \grad_{\vx_1} u_1(\vx_1, \vx_2')$. We now formalize this intuition.
\begin{lemma}
    \label{lemma:semiseparation-twoplayer}
    There is a polynomial-time algorithm that, given $\va = (\lambda_1, \lambda_2, \lambda_1\vx_1, \lambda_2\vx_2) \in \cP_\alpha$ and $\eps > 0$, either
    \begin{itemize}
        \item  outputs an $\eps$-Nash equilibrium (not necessarily $(\vx_1, \vx_2)$), or
        \item returns $\va' \in \cP_\alpha$ such that $\ip{G(\va'), \va - \va'} \ge \tilde\gamma := \eps^6/\poly(B, L, R)$.
    \end{itemize}
\end{lemma}
\begin{proof}
    We split the analysis into two cases.

    {\em Case 1.} $\min\{\lambda_1, \lambda_2\} \le  \gamma/(4BR)$, where $\gamma = \eps^2 L / (B + 4RL)^2$ is the constant in \Cref{lem:progress}. Assume without loss of generality that $\lambda_2 \le  \gamma/(4BR)$ (and $\lambda_2 \le 1/2$). Compute an $\eps$-approximate response $\vx_2'$ to $\vx_1$, and check if $(\vx_1, \vx_2')$ is an $\eps$-Nash equilibrium by computing an approximate best response to $\vx_2'$. (These are both convex optimization problems, solvable using standard techniques.) If so, output $(\vx_1, \vx_2')$. Otherwise,  let $\vx_1' = \Pi_{\cX_1}(\vx_1 + \eta \grad_{\vx_1} u_1(\vx_1, \vx_2'))$ with $\eta = 1/(2L)$ as before. By \Cref{lem:progress} applied to the VI problem $(\cX_1, -\grad_{\vx_1} u_1(\cdot, \vx_2'))$, we have
    \begin{equation*}
        \ip{\grad_{\vx_1} u_1(\vx_1', \vx_2'), \vx_1' - \vx_1} \ge \gamma
    \end{equation*}
    But then, setting $\va' = (\lambda_1', \lambda_2', \lambda_1' \vx_1', \lambda_2' \vx_2')$ (for any $\lambda_1', \lambda_2' > 0$), we have
    \begin{align*}
        \ip{G(\va'), \va' - \va} &= \lambda_1 \ip{\grad_{\vx_1} u_1(\vx_1', \vx_2'), \vx_1 - \vx_1'} + \lambda_2 \ip{\grad_{\vx_2} u_2(\vx_1', \vx_2'), \vx_2 - \vx_2'}
        \le -\lambda_1 \cdot \gamma + \lambda_2 \cdot 2BR
    \end{align*}
    Thus, for $\lambda_2 \le \min\{ 1/2, \gamma/(4BR) \}$ we have $\ip{G(\va'), \va' - \va} \le -\gamma/4$.

    {\em Case 2.} $\min\{\lambda_1, \lambda_2\} > \gamma/(4BR)$. Then let $\eta = 1/(2L')$, where $L' \le \poly(B, L, R)/\gamma^2$ is the Lipschitz constant of $G$ on $\cP_{\alpha}$, where $\alpha := \gamma^2/8BR$. Assume WLOG $L' \ge 4B^2 R/\gamma^2$ so that $\eta \le \gamma^2 / (8B^2R)$. Let $\va' = (\lambda_1', \lambda_2', \lambda_1' \vx_1', \lambda_2' \vx_2') = \Pi_{\cP_\alpha}(\va - \eta G(\va))$. Then, since $\eta \le \gamma/(8B^2R)$, we have $\min\{\lambda_1', \lambda_2'\} \ge \gamma/(8BR)$, so $\va, \va' \in \cP_{\alpha}$. Since $L'$ is the Lipschitz constant of $G$ on $\cP_{\alpha}$, \Cref{lem:progress} implies
    \begin{equation}
        \ip{G(\va'), \va - \va'} \ge \frac{\eps^2 L'}{(B + 4RL')^2} \ge \frac{\eps^2\gamma^2}{\poly(B, L, R)} \ge \tilde\gamma. \tag*{\qedhere}
    \end{equation}
\end{proof}
Following the remainder of the analysis of \Cref{theorem:maineither-precise} with $\gamma$ replaced by $\tilde\gamma$ yields \Cref{th:2p-nash-or-scce}.

It should be noted that~\Cref{th:2p-nash-or-scce} only applies for games with two players. This restriction is fundamental. Indeed, given any $n$-player game $\Gamma$, consider the $(n+1)$-player game $\Gamma'$ created by adding a player to $\Gamma$ whose utility is identically zero. Then the CCE gap for player $n+1$ is always zero, so $\Gamma'$ has no strict CCEs, but a Nash equilibrium of $\Gamma'$ would immediately also yield a Nash equilibrium of $\Gamma$. Thus, solving the $\eps$-Nash-or-strict-CCE problem for $n+1$ players is at least as hard as finding $n$-player Nash equilibria.
\section{Lower bounds}

We now turn to proving certain hardness results concerning MVIs. In \Cref{sec:gameshard}, we will show that deciding whether the Minty condition holds is \coNP-complete, with \coNP-hardness holding even for two-player concave games (\Cref{theorem:convexgame-hard}) or succinct multi-player (normal-form) games (\Cref{theorem:hardmintyconp}) when a constant error $\eps$ is allowed. Our results complement~\Cref{prop:Mintyeasy}, which showed that in \textit{explicitly} represented games there is a polynomial-time algorithm for that problem. In particular, \Cref{theorem:convexgame-hard,theorem:hardmintyconp} imply that determining the EVI that minimizes the equilibrium gap---that is, the strictest ACCE---is intractable.
From the point of view of the duality exposed in~\Cref{prop:duality}, this hardness result is perhaps surprising: an EVI can always be computed in polynomial time through an adaptation of the ellipsoid against hope algorithm~\citep{Zhang25:Expected}---but the dual program turns out to be hard. \Cref{sec:simplehard} provides a simpler proof that deciding the Minty condition is hard---albeit not applicable to games---based on the hardness of deciding membership in the copositive cone. Finally, in~\Cref{sec:hardMVI}, we formalize a straightforward query hardness result for solving MVIs under the promise that the Minty condition holds.

\subsection{Existence of MVI solutions}
\label{sec:gameshard}

We first characterize the complexity of deciding whether a VI problem satisfies the Minty condition (\Cref{ass:Minty}). We will begin by showing membership in \coNP, then show hardness results even in the special case of games.
\begin{definition}
    The {\em $\eps$-approximate Minty decision problem} is the following. Given a VI problem $\VI(\cX, F)$ where $\cX$ is assumed to be bounded and in isotropic position and $F$ satisfies \Cref{ass:polyeval}, and a precision parameter $\eps > 0$, decide whether $(+)$ $\VI(\cX, F)$ satisfies the Minty condition, or $(-)$ an $\eps$-strict EVI solution exists. 
\end{definition}
\begin{proposition}\label{prop:conp minty inclusion}
    The $\eps$-approximate Minty decision problem is in \coNP.
\end{proposition}
\begin{proof}
    By Carath\'eodory's theorem on convex hulls, an $\eps$-strict EVI solution can always have support $\poly(d)$. Further, a strict EVI solution can be checked in polynomial time with a single call to a linear optimization oracle over $\cX$.
\end{proof}
Having established \coNP-membership, we now turn to hardness. 

\subsubsection{Hardness for two-player concave games}
\label{sec:hardtwoplayer}
To begin with, we show that deciding this problem is hard even for two-player games if the strategy sets $\cX_1, \cX_2$ are allowed to be arbitrary polytopes.

\begin{theorem}
    \label{theorem:convexgame-hard}
    Consider a problem $\VI(\cX_1 \times \cX_2, F)$ associated with a two-player concave game per~\eqref{eq:gameoperator} such that $F$ satisfies~\Cref{ass:polyeval}, where $\cX_1 \subseteq \R^{d_1}$ and $\cX_2 \subseteq \R^{d_2}$ are the strategy sets of the two players. Then the $\eps$-approximate Minty decision problem is \coNP-hard, even when $\cX_1, \cX_2$ are polytopes given by explicit linear constraints, the utility functions $u_1, u_2 : \cX_1 \times \cX_2 \to [-1, 1]$ are bilinear, $\eps$ is an absolute constant, and we are promised that if there is a Minty point then $(\vec 0, \vec 0) \in \cX_1 \times \cX_2$ is Minty.
\end{theorem}
    We reduce from the following problem.
    \begin{restatable}[Bilinear optimization is hard]{lemma}{lemBilinearHard}\label{lem:bilinear-hard}
        There exists an absolute constant $\eps > 0$ for which the following problem is \NP-complete: given a bilinear map $f : [0, 1]^{d_1} \times [0, 1]^{d_2} \to [-1, 1]$ and target value $v \in [-1, 1]$, decide whether $(+)$ there exists $(\vx_1, \vx_2) \in [0, 1]^{d_1} \times [0, 1]^{d_2}$ such that $f(\vx_1, \vx_2) \ge  v + \eps$, or $(-)$ for all $(\vx_1, \vx_2) \in [0, 1]^{d_1} \times [0, 1]^{d_2}$, we have $f(\vx_1, \vx_2) \le v$.
    \end{restatable}
    This result is easy to show via reduction from MAX-2-SAT; in the interest of completeness, we include a proof in \Cref{appendix:proofs}.
    
    \begin{proof}[Proof of \Cref{theorem:convexgame-hard}]
        
    We will reduce from bilinear optimization. Given an instance $(f, v)$, construct the following two-player game. $\cX_1 = \{ (\vx_1, s) \in [0, 1]^{d_1} \times [0, 1] : \vec 0 \le \vx_1 \le \vec 1 s\}$, $\cX_2$ is defined similarly, and the utility functions are given by 
    \begin{align*}
        u_2((\vx_1, s), (\vx_2, t)) = 0, \qq{and}
        u_1((\vx_1, s), (\vx_2, t)) = f(\vx_1, \vx_2) + (1-s)v - 2(1-t) s.
    \end{align*}
    \newcommand{\tv}[1]{\tilde{\vec #1}}
    Notationally, we will use $\tilde\vx = (\vec x, s)$ and $\tilde\vy= (\vec y, t)$. 
    This corresponds to taking the normal-form game in which each player's strategy set is $\{0, 1\}^{d_i}$ and P1's utility function is $f$, and adding to it one strategy $\tv0$ for each player such that $u_1(\tv0, \tilde\vx_2) = v$ for all $\tilde\vx_2$, and $u_1(\tilde\vx_1, \tv0) = -2$ for all pure strategies (vertices of $\cX_1$) $\tilde\vx_1 \ne \tv0$. Let $F$ be the operator corresponding to this game. Then deciding if $\VI(\cX_1 \times \cX_2, F)$ satisfies the Minty condition amounts to deciding whether P1 has a dominant strategy.
    

   Suppose first that $f(\vx_1, \vx_2) \le v$ for every $(\vx_1, \vx_2)$. Then we see that $\tv0$ is dominant; indeed, $u_1(\tv0, \tilde\vx_2) = v  \ge u_1(\tilde\vx_1, \tilde\vx_2)$ for every $(\tilde\vx_1, \tilde\vx_2)$. Thus, in this case, $\VI(\cX_1 \times \cX_2, F)$ satisfies the Minty condition.   Conversely, suppose that there is some $(\vx_1^*, \vx_2^*)$ such that $f(\vx_1^*, \vx_2^*) \ge v + \eps$. We claim  here that P1 has no $\eps/4$-approximately dominant (hereafter $\eps/4$-dominant) strategy. Suppose for contradiction that $\tilde\vx_1 := (\vx_1, s)$ were $\eps/4$-dominant. Then in particular we have \[v(1-s)-2s = u_1(\tilde\vx_1, \tv0) \ge u_1(\tv0, \tv0) - \eps/4 = v-\eps/4,\] so $s \le \eps/4$. But then we have
   \begin{equation*}
       u_1(\tilde\vx_1, (\vx_2^*, 1)) \le s + (1-s) v \le v + \eps/2 < u_1((\vx_1^*, 1), (\vx_2^*, 1)) - \eps/4,
   \end{equation*}
   so $\tilde\vx_1$ is not $\eps/4$-dominant. Thus, in this case, $\VI(\cX_1 \times \cX_2, F)$ admits an $\eps/4$-strict EVI solution.
    \end{proof}
    
\subsubsection{Hardness for multi-player normal-form games}
\label{sec:hardsuccinct}

Next, we show that, even if the games are normal form, that is, each player's strategy set $\cX_i$ is a simplex $\Delta(\cA_i)$, deciding the Minty condition becomes hard when the number of players is large.
\begin{theorem}
    \label{theorem:hardmintyconp}
    Consider a problem $\VI(\cX, F)$ associated with a multi-player normal-form game per~\eqref{eq:gameoperator} such that $F$ satisfies~\Cref{ass:polyeval}.  Then the $\eps$-approximate Minty decision problem is \coNP-hard, even when each player has two actions, $\eps$ is an absolute constant, and ``everyone plays their first action'' is a Minty solution if a Minty solution exists. 
\end{theorem}

\begin{proof}
    We again reduce from bilinear optimization (\Cref{lem:bilinear-hard}). Given a bilinear map $f : [0,1]^{d_1} \times [0,1]^{d_2} \to [-1, +1]$, create a game as follows. There are $n := d_1 + d_2 + 2$ players, and two actions per player. We will use $\vx_1[i] \in [0, 1]$ to denote the mixed strategy of Player $i \in \{1, \dots, d_1\}$, $\vx_2[i] \in [0, 1]$ to denote the mixed strategy of $i + d_1 \in \{d_1 + 1, \dots, d_1 + d_2\}$, and $s, t \in [0, 1]$ to denote the mixed strategy of the remaining two players. We will overload notation and also use $s, t$ to refer to these two final players. The utility of every player except player $s$ is $0$. The utility of player $s$ is given by
    \[
    u_s(\vx_1, \vx_2, s, t) =  f(\vx_1, \vx_2) \cdot st + (1-s)v - 2(1-t)s.
    \]
    (This is precisely the utility function given in the previous proof, except reparameterized.) The proof now proceeds similarly to the proof of the previous result. Since only $s$ has nonzero utility, the Minty condition is equivalent to $s$ having a dominant strategy.

    Suppose first that $f(\vx_1, \vx_2) \le v$ for every $(\vx_1, \vx_2)$. Then we also have $u_s(\vx_1, \vx_2, s, t) \le v = u_s(\vx_1, \vx_2, 0, t)$ for all ($\vx_1, \vx_2, s, t)$, so $s=0$ is dominant. Conversely, suppose that there is some $(\vx_1^*, \vx_2^*)$ such that $f(\vx_1^*, \vx_2^*) \ge v + \eps$. Suppose for contradiction that $s \in [0, 1]$ is $\eps/4$-dominant. Then \[
    u_s(\vx_1^*, \vx_2^*, s, 0) = v(1-s) - 2s \ge u_s(\vx_1^*, \vx_2^*, 0, 0) - \eps/4 = v - \eps/4,
    \]
    so $s \le \eps/4$. But then, once again, we have
    \[
    u_s(\vx_1^*, \vx_2^*, s, 1) \le s + (1-s) v \le v + \eps/2 < u_s(\vx_1^*, \vx_2^*, 1, 1) - \eps/4,
    \]
    so $s$ is not $\eps/4$-dominant. Thus, in this case, $\VI(\cX, F)$ admits an $\eps/4$-strict EVI solution.
\end{proof}

To put this into context, we saw earlier in~\Cref{prop:Mintyeasy} that in explicitly given normal-form games, there is a polynomial-time algorithm---one whose running time scales polynomially in $\prod_{i=1}^n |\cA_i|$---that decides whether the Minty property holds. As such, \Cref{theorem:hardmintyconp} separates the complexity of deciding the Minty condition under succinct descriptions from that under explicitly-represented games.

\subsubsection{Hardness for MVI membership beyond games}
\label{sec:simplehard}

Beyond VI problems induced by games, treated in~\Cref{sec:hardtwoplayer,sec:hardsuccinct}, there is an immediate, simpler hardness argument for deciding MVI membership based on the complexity of deciding membership to the \emph{copositive cone}~\citep{Dickinson14:Computational,Murty87:Some}. In particular, a matrix $\mat{A} \in \R^{d \times d}$ is said to be copositive if for all nonnegative vectors $\vx \in \R^d$, we have $\langle \vx, \mat{A} \vx \rangle \geq 0$; this is a $\coNP$-complete problem. We can then construct the mapping $F : [0, 1]^d \ni \vx \mapsto \vx (\langle \vx, \mat{A} \vx \rangle)$. Then, verifying whether  $\vec{0} \in [0, 1]^d$ is an MVI solution is equivalent to checking whether $\mat{A}$ is copositive.

\subsection{Solving Minty VIs}
\label{sec:hardMVI}

Another natural question is whether one can compute Minty VI solutions---as opposed to SVI solutions, treated in~\Cref{theorem:mainMinty-precise}---in polynomial time under the promise that the set of MVI solutions is non-empty. In this subsection, we show lower bounds on this problem in both $\eps$ and $d$. The upcoming lower bounds are straightforward; it is likely that they have appeared elsewhere, but we include them for completeness since we are not aware of an explicit reference.

\subsubsection{Lower bound on \texorpdfstring{$\epsilon$}{epsilon} in one dimension}

Our next hardness result gives an exponential lower bound (in $\log(1/\epsilon)$) for that problem, even in a single dimension.

For a rational $\epsilon \ll 1$, we define 
\begin{equation}
    \label{eq:basefun}
    \phi_\epsilon : [\epsilon, 3 \epsilon] \ni x \mapsto (x - \epsilon)^2 (x - 3 \epsilon)^2 ( x^2(x^2 - 8\epsilon^2) - 1).
\end{equation}

We begin with an elementary calculation that characterizes the derivative of $\phi_\epsilon$.

\begin{claim}
    \label{claim:deriv}
    For any $x \in (\epsilon, 2\epsilon)$ it holds that $\phi'_\epsilon(x) < 0$, whereas for any $x \in (2\epsilon, 3\epsilon)$ it holds that $\phi_\epsilon'(x) > 0$. In particular, $\phi_\epsilon$ obtains its minimum at $x = 2 \epsilon$, with $\phi_\epsilon(2\epsilon) = - \epsilon^4 ( 1 + 16 \epsilon^4)$.
\end{claim}

\begin{proof}
    The derivative $\phi'_\epsilon(x)$ can be expressed as
    \begin{equation*}
         2(x - \epsilon) (x - 3 \epsilon)^2 (x^2(x^2 - 8\epsilon^2) - 1) + 2(x-\epsilon)^2 (x - 3 \epsilon) (x^2 (x^2 - 8 \epsilon^2) - 1) + (x - \epsilon)^2(x - 3 \epsilon)^2 (4x^3 - 16 x \epsilon^2).
    \end{equation*}
    For $x \in (\epsilon, 3 \epsilon)$, we have
    \begin{equation*}
        2(x - \epsilon)(x - 3 \epsilon)^2 (x^2(x^2 - 8\epsilon^2) - 1) < - 2(x-\epsilon)^2 (x - 3 \epsilon) (x^2 (x^2 - 8 \epsilon^2) - 1) \iff (3 \epsilon - x) >  (x - \epsilon )
    \end{equation*}
    and
    \begin{equation*}
        (x - \epsilon)^2(x - 3 \epsilon)^2 (4x^3 - 16 x \epsilon^2) \leq 0 \iff x \leq 2 \epsilon.
    \end{equation*}
\end{proof}
Now, let $\alpha \in [-\epsilon, 1 - 3 \epsilon]$. We extend~\eqref{eq:basefun} as follows.
\begin{equation*}
    f_{\epsilon, \alpha}: [0, 1] \ni x \mapsto 
    \begin{cases}
        \phi_\epsilon(x - \alpha) & \text{if } \epsilon \leq x - \alpha \leq 3 \epsilon, \\
        0 & \text{otherwise}.
    \end{cases}
\end{equation*}

\begin{figure}[ht] 
    \centering 
    \begin{tikzpicture}
        \begin{axis}[
            samples=200,
            domain=0:1,
            ymin=-0.0002, ymax=0.0002, 
            width=12cm, height=4cm, 
            axis x line=middle,
            axis y line=middle,
            xlabel={$x$},
            ylabel={$f_{\epsilon, \alpha}(x)$},
            xtick={0.1,0.2,0.3},
            xticklabels={},
            ytick=\empty,
            enlargelimits=true
        ]
            \addplot[blue, very thick, domain=0.1:0.3] 
            { (x - 0.1)^2 * (x - 0.3)^2 * (x^2*(x^2 - 8*0.1^2) - 1) };
            \addplot[blue, very thick, domain=0:0.1] {0};
            \addplot[blue, very thick, domain=0.3:1] {0};
            
            \node[anchor=south] at (axis cs:0.1,0) {$\epsilon$};
            \node[anchor=south] at (axis cs:0.2,0) {$2\epsilon$};
            \node[anchor=south] at (axis cs:0.3,0) {$3\epsilon$};
        \end{axis}
    \end{tikzpicture}
    \caption{The function $f_{\epsilon, \alpha}(x)$, with $\epsilon = 0.1$ and $\alpha = 0$, over the domain $[0,1]$; the y-axis is at a larger scale for the sake of the illustration.}
    \label{fig:feps_alpha}
\end{figure}
An example of $f_{\epsilon, \alpha}$ is illustrated in~\Cref{fig:feps_alpha}. We then define
\begin{equation}
    \label{eq:Fepsalpha}
    F_{\epsilon, \alpha}: [0, 1] \ni x \mapsto 
    \begin{cases}
        \phi'_\epsilon(x - \alpha) & \text{if } \epsilon < x - \alpha < 3 \epsilon, \\
        0 & \text{otherwise}.
    \end{cases}
\end{equation}

As we show next, the induced VI problem satisfies the Minty condition and is also Lipschitz continuous.

\begin{lemma}
    $\VI([0,1], F_{\epsilon, \alpha})$ satisfies the Minty condition. Furthermore, $F_{\epsilon, \alpha}$ is $O(\epsilon^2)$-Lipschitz continuous.
\end{lemma}

\begin{proof}
    We claim that $x \defeq \alpha + 2\epsilon$ is a solution to the Minty VI problem with respect to $F_{\epsilon, \alpha}$. Indeed, consider any $x' \neq x$. When $x' - \alpha \geq 3 \epsilon$ or $x' - \alpha \leq \epsilon$, it follows that $F_{\epsilon, \alpha}(x')(x' - x) = 0$. Suppose $x' - \alpha \in (\epsilon, 2 \epsilon)$. By~\Cref{claim:deriv}, we have $F_{\epsilon, \alpha}(x') < 0$ while $x' - x < 0$, in turn implying that $F_{\epsilon, \alpha}(x')(x' - x) > 0$. Similarly, when $x' - \alpha \in (2\epsilon, 3 \epsilon)$, we have $F_{\epsilon, \alpha}(x') > 0$ while $x' - x > 0$; this shows that $x = \alpha + 2 \epsilon$ is indeed a solution to the Minty VI problem.

    We continue with the argument that $F_{\epsilon, \alpha}$ is Lipschitz continuous. Let $x, x' \in [0, 1]$ such that $x < x'$. We consider the following cases:

    \begin{itemize}
        \item If $x \leq \alpha + \epsilon$ and $x' \geq \alpha + 3 \epsilon$, it follows that $F_{\epsilon, \alpha}(x) = F_{\epsilon, \alpha}(x') = 0$.
        \item If $x \leq \alpha + \epsilon$ and $x' \in (\alpha + \epsilon, \alpha + 3 \epsilon)$, it suffices to show that $| \phi'_\epsilon(x' - \alpha)| \leq L |x' - \alpha - \epsilon |$ since $|x' - \alpha - \epsilon| \leq |x' - x|$. By the definition of $\phi_\epsilon'$, this holds with $L = O(\epsilon^2)$.
        \item If $x \in (\alpha + \epsilon, \alpha + 3 \epsilon)$ and $x' \geq \alpha + 3 \epsilon$, it suffices to show that $| \phi_\epsilon'(x - \alpha)| \leq L | x - \alpha - 3 \epsilon|$ since $|x - \alpha - 3 \epsilon| \leq |x - x'|$. This again holds with $L = O(\epsilon^2)$.
        \item Finally, we treat the case where $x, x' \in (\alpha + \epsilon, \alpha + 3 \epsilon)$. We can expand $\phi'_\epsilon(x)$ as $24 \epsilon^3 - x(144 \epsilon^6 + 44 \epsilon^2) + x^2(576 \epsilon^5 + 24 \epsilon) - x^3 (668 \epsilon^4 + 4) + 200 \epsilon^3 x^4 + 84 \epsilon^2 x^5 - 56 \epsilon x^6 + 8 x^7$. From this expression it is easy to see that $| \phi_\epsilon'(x - \alpha) - \phi'_\epsilon(x' - \alpha) | \leq L |x - x'|$ for some $L = O(\epsilon^2)$.
    \end{itemize}
\end{proof}

It is worth noting that identifying an exact \emph{Stampacchia} VI solution to $\VI([0, 1], F_{\epsilon, \alpha})$ is trivial: any point outside the region $(\alpha + \epsilon, \alpha + 3 \epsilon)$ suffices since $F_{\epsilon, \alpha}$ is defined as $0$ in such points. On the other hand, any algorithm that outputs a non-trivial approximation to the global minimum of $f$ needs to output a point in the region $(\alpha + \epsilon, \alpha + 3 \epsilon)$. Now, if $\alpha$ is selected initially unbeknownst to the algorithm, it follows that any algorithm that succeeds with constant probability needs to submit $\Omega(1/\epsilon) = \Omega( 2^{\log(1/\epsilon)})$ queries to the evaluation oracle for $F$. We arrive at the following information-theoretic lower bound.

\begin{proposition}
    \label{theorem:Mintyexp}
    For any sufficiently small $\epsilon > 0$, any algorithm that computes with constant probability a point $\vx \in \cX$ such that $f(\vx) \leq \min_{\vx' \in \cX} f(\vx') + \epsilon^4$ of a function whose associated VI problem satisfies the Minty condition requires $\Omega(1/\epsilon)$ gradient evaluations, even when $\cX = [0, 1]$, $L= O(\epsilon^2)$ and $B = O(\epsilon^3)$.
\end{proposition}

Furthermore, suppose that one is instead looking for an approximate MVI solution $x \in [0, 1]$; namely, $F(x') (x' - x) \geq - C \epsilon^4$ for all $x' \in [0, 1]$, where $C$ is a sufficiently small constant. We claim that this forces $x$ to belong in $(\alpha + \epsilon, \alpha + 3 \epsilon)$, at which point the hardness of~\Cref{theorem:Mintyexp} kicks in. For the sake of contradiction, suppose first that $x \geq \alpha + 3 \epsilon$. Taking then $x' \defeq \alpha + 2.5 \epsilon$ leads to a violation since $x' - x < 0, F(x') > 0$ (by \Cref{claim:deriv}), and $|x' - x| = \Theta(\epsilon), |F(x')| = \Theta(\epsilon^3)$. Similar reasoning applies if $x \leq \alpha + \epsilon$ by deviating to $x' \defeq \alpha + 1.5 \epsilon$.

\begin{corollary}
    For any sufficiently small $\epsilon > 0$, any algorithm that computes with constant probability a point $\vx \in \cX$ such that $\langle F(\vx'), \vx' - \vx \rangle \geq - C \epsilon^4$, for a sufficiently small constant $C > 0$, of a problem $\VI(\cX, F)$ that satisfies the Minty condition requires $\Omega(1/\epsilon)$ gradient evaluations, even when $\cX = [0, 1]$, $L= O(\epsilon^2)$ and $B = O(\epsilon^3)$.
\end{corollary}

\subsubsection{Lower bound in high dimensions}

We now address the problem of finding an approximate Minty VI solution in high dimensions, given the promise that such a solution exists. 
\begin{proposition}\label{theorem:Mintyexp-d}
    There exists an absolute constant $\eps > 0$ for which the following holds: any algorithm that computes with constant probability an $\eps$-approximate global optimum of a function $f : \cX \to \R$ whose associated VI problem satisfies the Minty condition requires $\Omega(2^d)$ gradient evaluations, even when $\cX = [-1, 1]^d$ and $\eps, L, B$ are absolute constants.
\end{proposition}
\begin{proof}
Let  $\vc \in \{-1, +1\}^d$, and  
$
f_\vc : [-1, 1]^d \to \R
$ be defined by 
$
f_\vc(\vx) \defeq -\qty(\max \qty{ 1 - \norm{\vx - \vc}^2, 0})^2
$.
Then $f_\vc$ has global minimum $-1$ at $\vx = \vc$, and its gradient is given by
\begin{align*}
    F_\vc(\vx) = \grad f_\vc(\vx) = \begin{cases}
        4(\vx - \vc) (1 - \norm{\vx - \vc}^2) &\qif \norm{\vx - \vc} \le 1,\\
        0 &\qq{otherwise.}
    \end{cases}
\end{align*}
which indeed has absolute constant $L$ and $B$. Finally, if $\sgn(\vx) \ne \vc$, then $f_\vc(\vx) = 0$ and $F_\vc(\vx) = \vec 0$. Therefore, any algorithm that optimizes this family of functions $f_\vc$ learns no information about $\vc$ until it queries a point with $\sgn(\vx) = \vc$, which takes $\Omega(2^d)$ gradient evaluations in expectation if $\vc$ is chosen uniformly at random.
\end{proof}

\subsubsection{On equilibrium collapse}
\label{sec:collapse}

The class of instances behind~\Cref{theorem:Mintyexp} also precludes a natural approach for computing SVI solutions under the Minty condition; namely, the positive result of~\citet{Cai16:Zero} (\emph{cf.}~\citealp{Kalogiannis23:Zero,Park23:Multi}) concerning zero-sum, polymatrix games is based on the observation that there is an ``equilibrium collapse,'' meaning that taking the marginals of any CCE results in a Nash equilibrium; in fact, it is easy to see that this holds more generally under the preconditions of~\Cref{prop:polymatrix}, and more generally under the assumption of monotonicity. To be more precise, by ``equilibrium collapse'' in general VI problems we mean the following.

\begin{definition}[Equilibrium collapse]
    A VI problem $\VI(\cX, F)$ exhibits \emph{equilibrium collapse} if for any $\epsilon$-EVI solution $\mu \in \Delta(\cX)$ (per~\Cref{def:EVIs}), $\E_{\vx \sim \mu} \vx$ is guaranteed to be a $p(\epsilon)$-SVI solution, for some polynomial $p$ in $\epsilon$.
\end{definition}

This phenomenon of equilibrium collapse is the main ingredient in the aforementioned $\polylog(1/\eps)$-time algorithms for solving {\em monotone} VIs (\eg, \citet{Nemirovski10:Accuracy} and citations therein). Those algorithms essentially work as follows: run a cutting-plane algorithm like the ellipsoid algorithm, attempting to compute an $\eps$-EVI solution. If at any point the query point is an $\eps$-SVI solution, terminate immediately; otherwise, use equilibrium collapse to conclude that the $\eps$-EVI solution is in fact an approximate SVI solution.

Here, we show an explicit counterexample that demonstrates that equilibrium collapse is not true in general under the Minty property.

\begin{proposition}\label{prop:equilcollapse}
    There is a VI problem $\VI([-1, 1]^2, F)$ that satisfies the Minty condition, but that has an exact expected VI solution $\mu \in \Delta([-1, 1]^2)$ such that 1) $\E_{\vx\sim\mu} \vx$ is not an SVI solution, and 2) no point in the support of $\mu$ is an SVI solution.
\end{proposition}
\begin{proof}
    Let $F(x, y) = (2+y)(-y, x)$ (see \Cref{fig:vector-field-collapse-counterexample}). Then $(0, 0)$ is the unique Minty point, and $\mu  = \frac{3}{4}(0, -1) + \frac{1}{4}(0, 1)$ is an exact EVI solution, but $(0, -1), (0, 1)$, and $\E_{\vx\sim\mu}\vx = (0, -1/2)$ are not SVI solutions. 
\end{proof}

\begin{figure}
    \centering
    \includegraphics[width=0.5\linewidth]{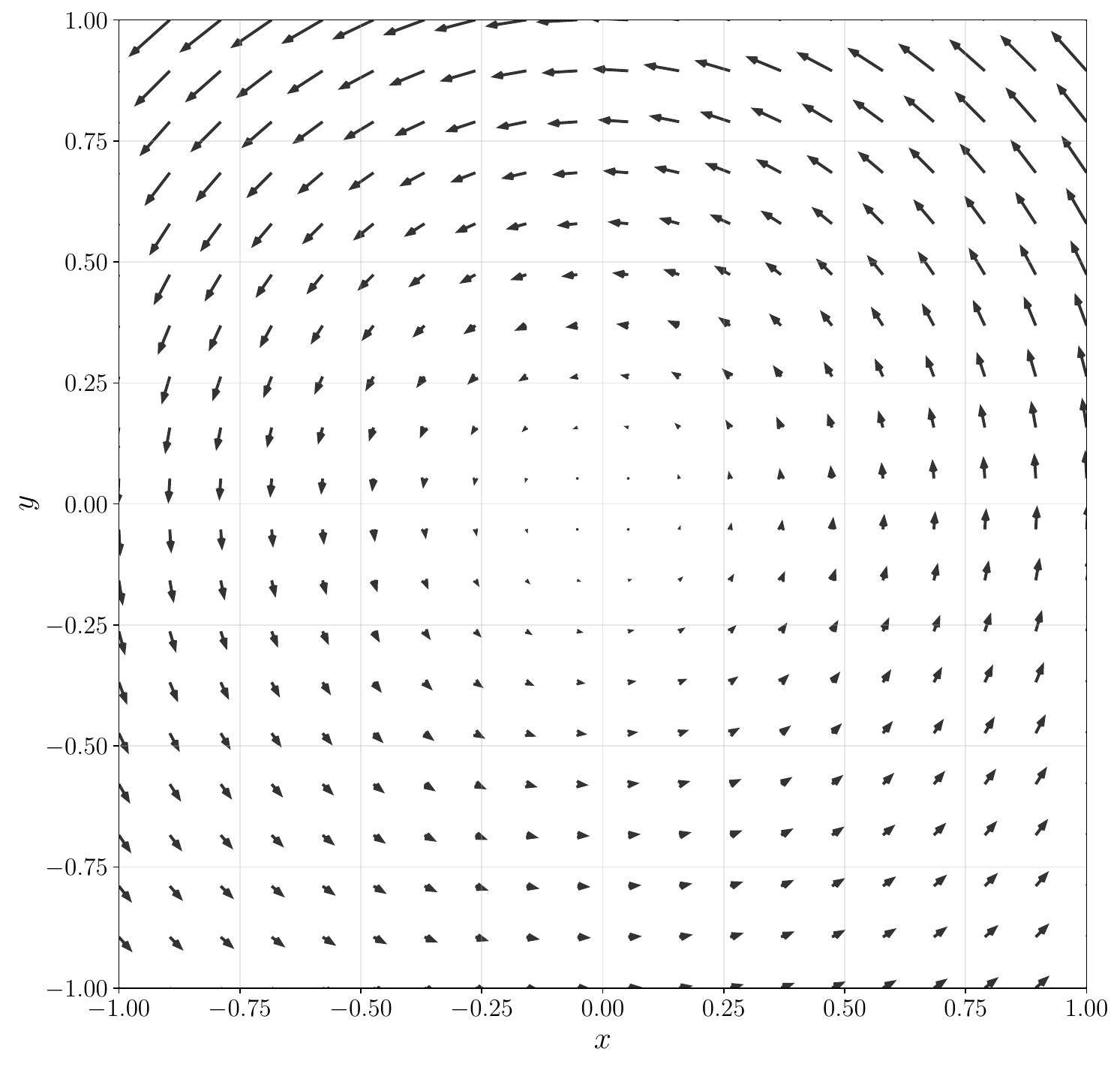}
    \caption{The vector field in the proof of \Cref{prop:equilcollapse}.}
    \label{fig:vector-field-collapse-counterexample}
\end{figure}



\subsection{Ellipsoid without extra-gradient}\label{sec:noopt}

\begin{figure}
    \centering
    \includegraphics[width=0.32\linewidth]{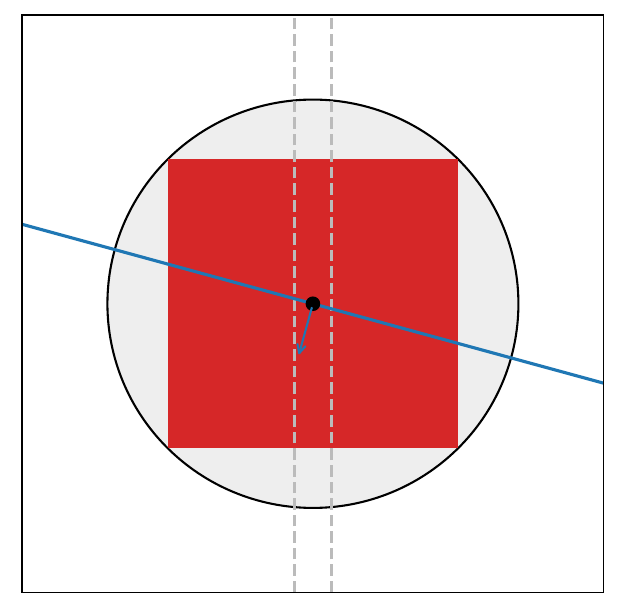}
    \includegraphics[width=0.32\linewidth]{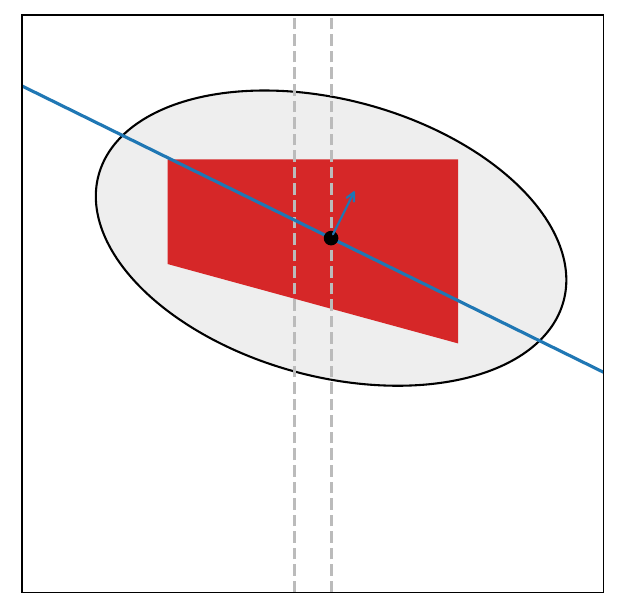}
    \includegraphics[width=0.32\linewidth]{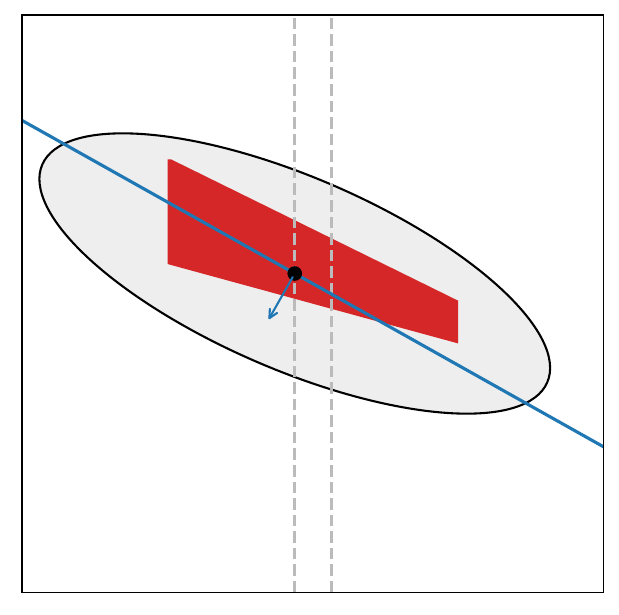}
    \\
    \includegraphics[width=0.32\linewidth]{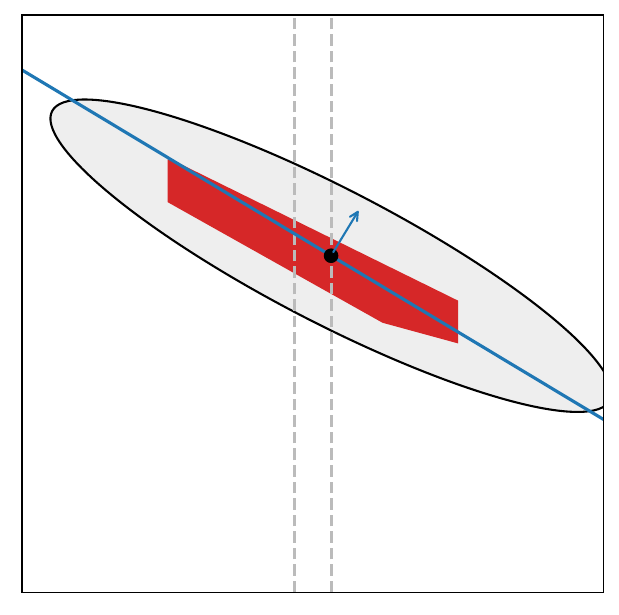}
    \includegraphics[width=0.32\linewidth]{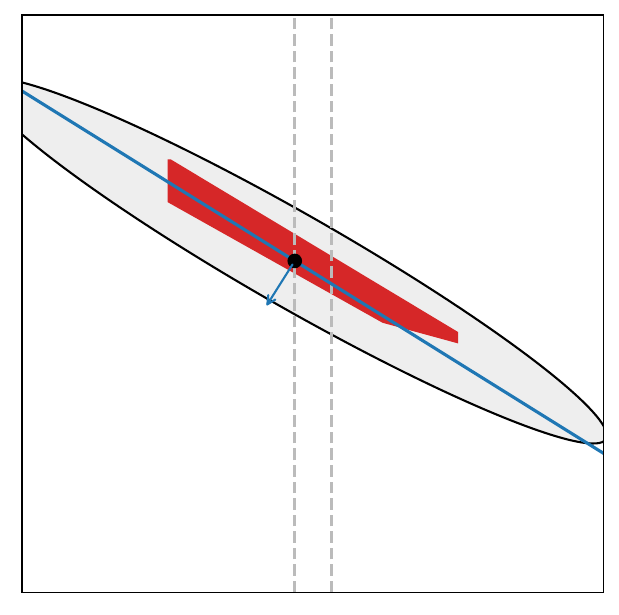}
    \includegraphics[width=0.32\linewidth]{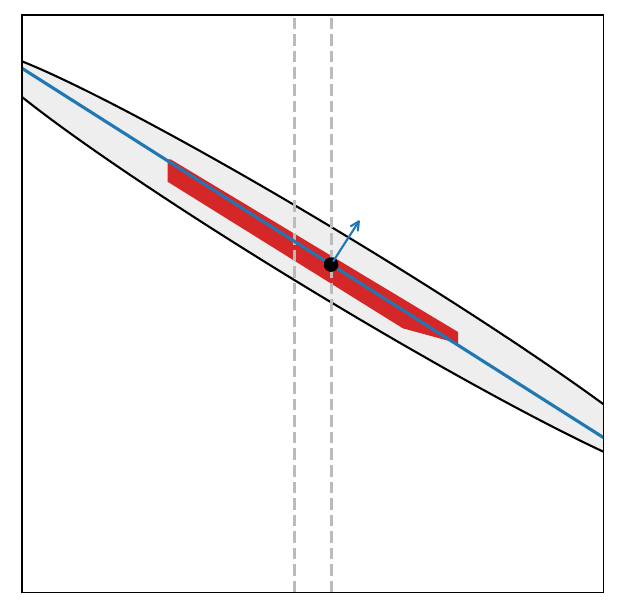}
    \\
    \includegraphics[width=0.32\linewidth]{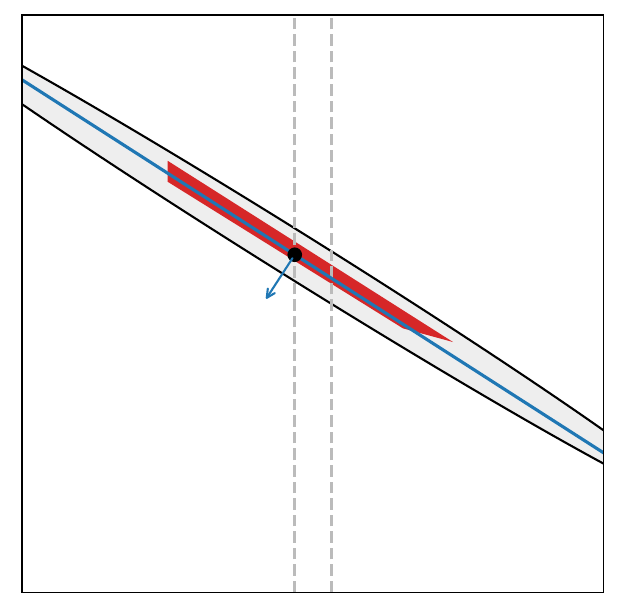}
    \includegraphics[width=0.32\linewidth]{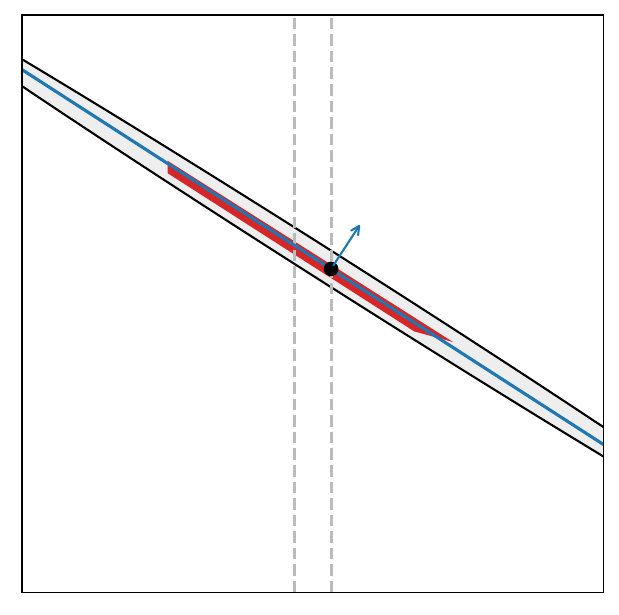}
    \includegraphics[width=0.32\linewidth]{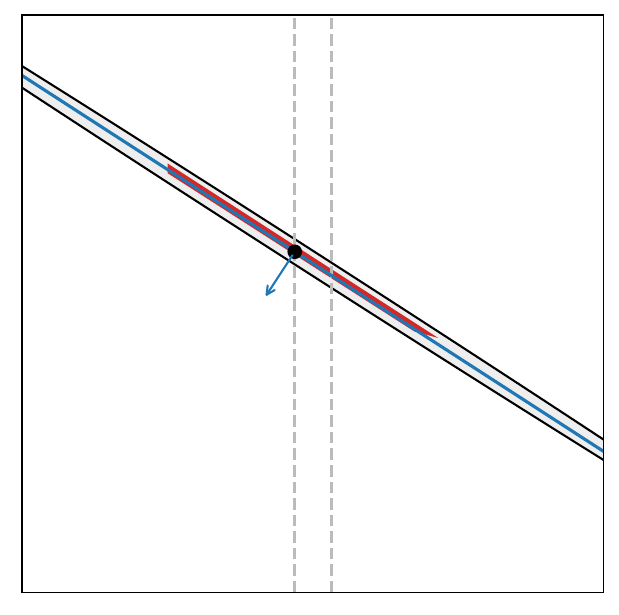}
    \caption{The counterexample for the ellipsoid algorithm without extra-gradient as described in \Cref{sec:noopt}, after $t$ iterations, for $t = 0, 1, \dots, 7, 8$. In each plot, the red polygon is the feasible region where Minty points might still lie, and the blue line/arrow indicates the separating direction $F(\va^{(t)})$. The two dashed vertical lines correspond to~$x = \pm 1/8$. }
    \label{fig:noopt-counterexample}
\end{figure}

It is reasonable to ask whether the extra-gradient step in our main algorithm (\Cref{algo:optimisticellipsoid}) is necessary. In other words, if the extra-gradient step were to be removed, would we still be able to have $\log(1/\eps)$ convergence toward an SVI solution under the Minty condition? In this section, we give strong computational evidence that such a guarantee is impossible. We ran the ellipsoid algorithm on polytope $\cX = [-1, 1]^2$ and starting ellipse $B_{\sqrt{2}}(\vec 0)$, which is the smallest ellipse containing $\cX$.  At each timestep, given ellipsoid center $\va^{(t)} \in \cX$, we generated a separating direction $F(\va^{(t)})$ such that:
\begin{itemize}
    \item when $t$ is even, $F(\va^{(t)})$ has positive $y$-component, and $\va^{(t+1)}$ has $x$-component $1/8$, and
    \item  when $t$ is odd, $F(\va^{(t)})$ has negative $y$-component, and $\va^{(t+1)}$ has $x$-component $-1/8$.
\end{itemize}
As shown in \Cref{fig:noopt-counterexample}, the ellipse quickly grows extremely  thin along one axis. 
After $T = 1293$ iterations,\footnote{We ran this algorithm with $2048$-bit floating-point arithmetic, and numerical precision issues arose at this point.} we observe the following properties.
\begin{itemize}
    \item The algorithm has not proven that no Minty solution exists. That is, the red polytope in \Cref{fig:noopt-counterexample} remains nonempty, and in fact the current center $\va^{(T)}$ is a candidate Minty solution.
    \item The radius of the short axis of the ellipse is less than $1.498 \times 10^{-219}$. 
    \item No approximate SVI point has been found: $\norm{F(\va^{(t)})} = 1$ and $\norm{\va^{(t)}} < 0.481$ for every timestep $t$, so the minimum SVI gap of any queried point is at least $0.519$. 
    \item The algorithm never queries the same point twice; therefore, the operator $F$ is well-defined, and 
    \item No Lipschitz violation has been found:  for every $0 \le s < t \le T$, we have \[\frac{\norm{F(\va^{(s)}) - F(\va^{(t)})}}{\norm{\va^{(s)} - \va^{(t)}}} < 7.997.\]
    
\end{itemize} 

A common idea when using the ellipsoid method for low-dimensional sets is to, once the shortest axis of the ellipse gets small, simply project the problem onto a subspace of smaller dimension and continue. However, this will not work in our setting. Indeed, after such a projection is made, it is possible that $F(\va^{(t)})$ has large magnitude in a direction orthogonal to the subspace (so that $\va^{(t)}$ is not an SVI point) and yet $F(\va^{(t)})$ gives no separation direction, so ellipsoid cannot proceed. 

Thus, even in two dimensions, this counterexample demonstrates that ellipsoid fails to find an approximate SVI solution even when the Minty property holds and the given function is Lipschitz. This demonstrates the need for the additional ideas we introduce in~\Cref{sec:ouralgo} (namely, the extra-gradient step) beyond the naive ellipsoid method.

\section{Conclusion and future research}

In summary, we established the first algorithm for computing $\epsilon$-SVI solutions under the Minty condition that has polynomial dependence on both $d$ and $\log(1/\eps)$. We also provided several new applications of our main results in optimization and game theory, including the first polynomial-time algorithm for quasar-convex optimization. Our results raise a number of interesting questions for future work. First, one of our main contributions was a technique to tackle the lack of full dimensionality when using the ellipsoid algorithm by using strict separation; it is possible that our approach can be applied to other problems as well. Furthermore, our result concerning strict CCEs only applies to two-player games; while it cannot be generalized to games with more than two players or to tighter equilibrium concepts such as correlated equilibrium~\citep{Aumann74:Subjectivity}, it would still be interesting to characterize classes of games and solution concepts for which such extensions are possible. Indeed, our main result provides further motivation for characterizing what problems satisfy the Minty condition and related concepts. Finally, are there other polynomial-time algorithms besides $\optellipsoid$ for solving VIs under the Minty condition? For example, is there some ellipsoid-based variant of the proximal method of~\citet{Nemirovski04:Prox} that yields a similar guarantee to $\optellipsoid$?

\section*{Acknowledgments}

G.F. is supported by the National Science Foundation grant CCF-2443068. T.S. is supported by the Vannevar Bush Faculty Fellowship ONR N00014-23-1-2876, National Science Foundation grants RI-2312342 and RI-1901403, ARO award W911NF2210266, and NIH award A240108S001. This work was performed while B.H.Z. was supported by the CMU Computer Science Department Hans Berliner PhD Student Fellowship. We thank Pablo Parrilo for helpful discussions. We are grateful to the anonymous FOCS reviewers for their valuable feedback.

\bibliography{dairefs}

\appendix

\section{Sufficiency of isotropic position}
\label{sec:isotropic}

This section shows that assuming $\cX$ to be in isotropic position is without any loss; we formalize this in the following lemma (\emph{cf.}~\citet[Appendix A]{Daskalakis24:Efficient}).

\begin{lemma}[Sufficiency of isotropic position]
    Let $\cX \subseteq \cB_{R}(\vec{0})$ be a convex and compact set, with $R \geq 1$, and $F : \cX \to \R^d$. There is an invertible, affine transformation $\psi$ and a mapping $\tilF : \tilcX \defeq \psi(\cX) \to \R^d$ such that the following properties hold:
    \begin{enumerate}
        \item $\tilcX = \psi(\cX)$ is in isotropic position.\label{item:isotropic}
        \item If $\VI(\cX, F)$ satisfies the Minty condition, then $\VI(\tilcX, \tilF)$ also satisfies the Minty condition.\label{item:Minty}
        \item If $\tilde{\vx}$ is an $\epsilon$-SVI solution (per~\Cref{def:approxVI}) to $\VI(\tilcX, \tilF)$, then $\vx \defeq \psi^{-1}(\tilde{\vx})$ is a $(2R \epsilon)$-SVI solution to $\VI(\cX, F)$.\label{item:SVI}
        \item If $F$ is $L$-Lipschitz continuous, then $\tilF$ is $(4R^2 L)$-Lipschitz continuous; if $\| F(\vx) \| \leq B$ for all $\vx \in \cX$, then $\| \tilF(\tilde{\vx}) \| \leq 2R B$ for all $\tilde{\vx} \in \tilcX$; and if $F$ satisfies~\Cref{ass:polyeval}, then $\tilF$ also satisfies~\Cref{ass:polyeval}.\label{item:F}
    \end{enumerate}
\end{lemma}

\begin{proof}
    Let $\psi(\vx) \defeq \mat{A} \vx + \vec{b}$ for an invertible matrix $\mat{A}$. \Cref{item:isotropic} follows by suitably selecting $\mat{A}$ and $\vec{b}$; for example, see the polynomial-time algorithm of~\citet{Lovasz06:Simulated}.

    With regard to~\Cref{item:Minty}, we proceed as follows. Since $\tilcX \supseteq \cB_1(\vec{0})$, there exists $\vx \in \cB_R(\vec{0})$ such that $\vec{0} = \mat{A} \vx + \vec{b}$, implying that $\|\mat{A}^{-1} \vec{b} \| \leq R$. Further, for any $\tilde{\vx}$ with $\|\tilde{\vx} \| = 1$, we have $\tilde{\vx} = \mat{A} \vx + \vec{b}$ for some $\vx \in \cB_R(\vec{0})$. So, $\| \mat{A}^{-1} \tilde{\vx} \| = \| \mat{A}^{-1} (\tilde{\vx} - \vec{b}) + \mat{A}^{-1} \vec{b} \| \leq \|\vx\| + \|\mat{A}^{-1} \vec{b} \| \leq 2R$. This implies that $\| \mat{A}^{-1} \| \leq 2R$.

    Now, we define $\tilF : \tilcX \ni \tilde{\vx} \mapsto (\mat{A}^{-1})^\top F(\psi^{-1}(\tilde{\vx}))$. Suppose that there is $\vx \in \cX$ such that $\langle F(\vx'), \vx' - \vx \rangle \geq 0$ for all $\vx' \in \cX$; that is, $\VI(\cX, F)$ satisfies the Minty condition. Then, we claim that $\tilde{\vx} \defeq \mat{A} \vx + \vec{b}$ satisfies the Minty condition with respect to $\VI(\tilcX, \tilF)$. Indeed, for any $\tilde{\vx}' \in \tilcX$,
\begin{equation*}
    \langle \tilF(\tilde{\vx}'), \tilde{\vx}' - \tilde{\vx} \rangle = \langle (\mat{A}^{-1})^\top F(\psi^{-1}(\tilde{\vx}')), \mat{A} \vx' + \vec{b} - \mat{A} \vx - \vec{b} \rangle = \langle F(\vx'), \vx' - \vx \rangle \geq 0.
\end{equation*}
That is, $\VI(\tilcX, \tilF)$ satisfies the Minty condition. 

Similarly, for~\Cref{item:SVI}, suppose $\tilde{\vx} \in \tilcX^{+\epsilon}$ is an $\epsilon$-SVI solution to $\VI(\tilcX, \tilF)$. Then,
\begin{equation*}
    - \epsilon \leq \langle \tilF(\tilde{\vx}), \tilde{\vx}' - \tilde{\vx} \rangle = \langle (\mat{A}^{-1})^\top F(\psi^{-1}(\tilde{\vx})), \mat{A} \vx' + \vec{b} - \mat{A} \vx - \vec{b} \rangle = \langle F(\vx), \vx' - \vx \rangle
\end{equation*}
for any $\cX \ni \vx' = \psi^{-1}(\tilde{\vx}')$. Further, letting $\vx \defeq \psi^{-1}(\tilde{\vx})$,
\begin{equation*}
    \left\| \vx - \left( \mat{A}^{-1} \left( \proj_{\tilcX}(\tilde{\vx}) - \vec{b} \right) \right) \right\| = \left\| \mat{A}^{-1} \left( \tilde{\vx} - \proj_{\tilcX}(\tilde{\vx}) \right) \right\| \leq \|\mat{A}^{-1} \| \left\| \tilde{\vx} - \proj_{\tilcX}(\tilde{\vx}) \right\| \leq 2R \epsilon,
\end{equation*}
since $\tilde{\vx} \in \tilcX^{+\epsilon}$; this implies that $\vx \in \cX^{+(2R\epsilon)}$ because $\mat{A}^{-1} ( \proj_{\tilcX}(\tilde{\vx}) - \vec{b} ) \in \cX$. We conclude that $\vx = \psi^{-1}(\tilde{\vx})$ is a $(2R \epsilon)$-SVI solution to $\VI(\cX, F)$.

We finally deal with~\Cref{item:F}. For any $\tilde{\vx}, \tilde{\vx}' \in \tilcX$, we have
\begin{align*}
    \| \tilF(\tilde{\vx}) - \tilF(\tilde{\vx}') \| &= \| (\mat{A}^{-1})^\top ( F(\psi^{-1}(\tilde{\vx})) - F(\psi^{-1}(\tilde{\vx}')) ) \| \\
    &\leq \| \mat{A}^{-1} \| \| F(\psi^{-1}(\tilde{\vx})) - F(\psi^{-1}(\tilde{\vx}')) \| \\
    &\leq L \| \mat{A}^{-1} \| \| \psi^{-1}(\tilde{\vx}) - \psi^{-1}(\tilde{\vx}') \| \\
    &\leq L \| \mat{A}^{-1} \| \| \mat{A}^{-1} \tilde{\vx} - \mat{A}^{-1} \tilde{\vx}' \| \\
    &\leq L \| \mat{A}^{-1} \|^2 \|\tilde{\vx} - \tilde{\vx}'\| \leq 4 R^2 L \|\tilde{\vx} - \tilde{\vx}' \|,
\end{align*}
where we used the fact that $F$ is $L$-Lipschitz continuous and $\|\mat{A}^{-1} \| \leq 2R$. 

Next, consider any $\tilde{\vx} \in \tilcX$. Let $\vx = \psi^{-1}(\tilde{\vx}) = \mat{A}^{-1}(\tilde{\vx} - \vec{b})$. We have $\| \tilF(\tilde{\vx}) \| = \|(\mat{A}^{-1})^\top F(\vx) \| \leq 2R B$, as claimed. The final assertion in~\Cref{item:F}---pertaining to~\Cref{ass:polyeval}---is immediate from the definition of $\tilF$.
\end{proof}

\section{Semi-separation under weak oracles}
\label{sec:weak}

For completeness, this section explains how our analysis in~\Cref{sec:ouralgo} can be extended under weak oracles. Throughout this paper, we work with a general convex set $\cX$, which might even be supported solely on irrational points; it will thus be necessary to consider~\eqref{eq:deepset} and~\eqref{eq:approxset}---in place of $\cX$---in the definitions that follow to ensure the existence of points with rational coordinates.

\begin{remark}
    \label{remark:-delta}
    We assume throughout that $\cX$ is well-bounded, in that $\cB_{r}(\vec{a}) \subseteq \cX \subseteq \cB_R(\vec{0})$, which in particular implies that $\cX$ is fully dimensional. Under this assumption, it is known (for example, see~\citet[Lemma 3.2.35]{Grotschel12:Geometric}) that if $\langle \vec{c}, \vx \rangle \leq \gamma$ holds for all $\vx \in \cX^{-\delta}$, it follows that 
    \begin{equation}
        \label{eq:rescaling}
        \langle \vec{c}, \vx \rangle \leq \gamma + \frac{2R}{r} \| \vec{c} \| \delta \quad \forall \vx \in \cX.
    \end{equation}
    As a result, we can safely consider deviations in $\cX$, not merely in $\cX^{-\delta}$, by suitably rescaling the precision parameters---by virtue of~\eqref{eq:rescaling}.
\end{remark}

We are now ready to introduce the weak versions of the oracles assumed in the main body, which are known to be (polynomial-time) equivalent when $\cX$ is well-bounded; in light of~\Cref{remark:-delta}, we can handle deviations $\vx' \in \cX$ instead of $\vx' \in \cX^{-\epsilon}$.

\begin{definition}[Weak membership;~\citealp{Grotschel12:Geometric}]
    \label{def:weakmem}
    Given a point $\vx \in \Q^d$ and a rational number $\epsilon \in \Q_{> 0}$, decide whether $\vx \in \cX^{+\epsilon}$.
\end{definition}

\begin{definition}[Weak separation;~\citealp{Grotschel12:Geometric}]
    \label{def:weaksep}
    Given a point $\vec{x} \in \Q^d$ and a rational number $\epsilon \in \Q_{> 0}$, either
    \begin{itemize}
        \item  assert that $\vx \in \cX^{+ \epsilon}$ or
        \item find a vector $\vec{c} \in \Q^d$ with $\| \vec{c} \|_\infty = 1$ such that $\langle \vec{c}, \vx' \rangle < \langle \vec{c}, \vx \rangle + \epsilon$ for every $\vx' \in \cX^{- \epsilon}$.
    \end{itemize}
\end{definition}

\begin{definition}[Weak optimization;~\citealp{Grotschel12:Geometric}]
    \label{def:weakopt}
    Given a vector $\vec{c} \in \Q^d$ and a rational number $\epsilon \in \Q_{>0}$, either
    \begin{itemize}
        \item  assert that $\cX^{-\epsilon}$ is empty or
        \item find a vector $\vx \in \cX^{+\epsilon} \cap \Q^d$ such that $\langle \vec{c}, \vec{x} \rangle \leq \langle \vec{c}, \vx' \rangle + \epsilon$ for all $\vx' \in \cX^{-\epsilon}$. 
    \end{itemize}
\end{definition}

Similarly, \Cref{ass:polyeval} can be extended as follows.

\begin{assumption}
    \label{ass:weak-polyeval}
    For a fixed $\epsilon_0 \in \Q_{>0}$, the mapping $F : \cX^{+\epsilon_0} \to \R^d$ satisfies the following:
    \begin{enumerate}
        \item for any rational $\vx \in \cX^{+\epsilon_0} \cap \Q^d$, $F(\vx)$ is a rational vector that can be evaluated exactly in $\poly(d)$ time, with $\size(F(\vx)) \leq \poly(\size(\vx))$;
        \item for any $\vx, \vx' \in \cX^{+\epsilon_0}$, $\|F(\vx) - F(\vx') \| \leq L \|\vx - \vx' \|$ for some $L \in \Q_{>0}$; and
        \item for any $\vx \in \cX^{+\epsilon_0}$, $\|F(\vx)\| \leq B$ for some $B \in \Q_{>0}$.
    \end{enumerate}
\end{assumption}
Above, we use the parameter $\epsilon_0 \ll 1$ for convenience in the notation (in~\Cref{lemma:extendedsemi}); for the purpose of the main body, $\epsilon_0 = 0$.\footnote{We also assume that $\cX^{+\epsilon_0} \subseteq \cB_R(\vec{0})$.} Regarding~\Cref{item:polyeval}, a weaker assumption, which suffices for our purposes, is that we have access to an oracle that, for any $\vx \in \cX \cap \Q^d$ and rational $\delta \in \Q_{> 0}$, returns $\vec{g} \in \Q^d$ such that $\| F(\vx) - \vec{g} \| \leq \delta$.

We next state a simple, well-known result regarding optimizing convex functions with respect to a constraint set that admits a separation oracle.

\begin{theorem}[\citealp{Grotschel12:Geometric}]
    Let $\cX \subseteq \R^d$ be in isotropic position, given by a (weak) separation oracle, and $f : \cX \to \R$ a convex function that can be evaluated exactly in $\poly(d)$ time. There is a $\poly(d, \log(1/\epsilon))$-time algorithm that outputs $\vx \in \cX^{+\epsilon}$ such that $f(\vx) \leq \min_{\vx' \in \cX^{-\epsilon}} f(\vx') + \epsilon$.
\end{theorem}

\begin{remark}
    In addition, suppose that $f$ is differentiable and $\mu$-strongly convex: $f(\vx) \geq f(\vx') + \langle \nabla f(\vx'), \vx - \vx' \rangle + \frac{\mu}{2} \| \vx - \vx' \|^2$. If $\vx'$ is the global minimum of $f$ with respect to $\cX^{-\epsilon}$, we have $\| \proj_{\cX^{-\epsilon}}(\vx) - \vx' \|^2 \leq 2\epsilon/\mu + 2( f(\proj_{\cX^{-\epsilon}}(\vx)) - f(\vx') )/\mu $; if $f$ is Lipschitz continuous, it follows that $\vx$ is also geometrically close to $\vx'$.
\end{remark}

We next introduce the relaxation of~\Cref{def:VI} for general convex sets; the only difference is that the solution $\vx$ is allowed to belong to $\cX^{+\epsilon}$, as opposed to $\cX$ in~\Cref{def:VI}. (There is no need to refine~\Cref{def:MVI} concerning MVIs since our computational results are about SVIs; even when positing the Minty condition, we do not assume access to an MVI solution.)

\begin{definition}[Relaxation of~\Cref{def:VI} for general convex sets]
    \label{def:approxVI}
    Let $\cX$ be a convex and compact subset of $\R^d$ and let $F : \cX \to \R^d$ be a mapping. The $\epsilon$-approximate \emph{(Stampacchia) variational inequality (SVI)} problem asks for a point $\vx \in \cX^{+\epsilon}$ such that
    \begin{equation}
        \label{eq:approxVI}
    \langle F(\vx), \vx' - \vx \rangle \geq - \epsilon \quad \forall \vx' \in \cX.
    \end{equation}
\end{definition}
We do not insist on satisfying~\eqref{eq:approxVI} only for $\vx' \in \cX^{-\epsilon}$ in light of~\Cref{remark:-delta}. Now, in the context of~\Cref{sec:ouralgo}, it suffices to provide the approximate version of~\Cref{lemma:strict1}.

\begin{lemma}
    \label{lemma:extendedsemi}
    Suppose that $\va \in \cX^{+\diffeps}$ is not an $\epsilon$-SVI solution, with $\diffeps < \epsilon$. If $\vatil \in \cX^{+\diffeps}$ satisfies
    \begin{equation}
        \label{eq:approx-fo}
        \left\langle F(\va) + \frac{1}{\eta} (\vatil - \va), \vx - \vatil \right\rangle \geq - B \diffeps \quad \forall \vx \in \cX
    \end{equation}
    with $\eta = 1/(2L)$, then
    \begin{equation}
        \label{eq:strictbutnotc}
        \langle F(\vatil), \va - \vx \rangle \geq L \left( B + 4 L R \right)^{-2} (\epsilon - B \diffeps )^2 - \diffeps(3 B + 6 L R).
    \end{equation}
    for any $\vx \in \cX$ that satisfies the Minty VI~\eqref{eq:MVI}.
\end{lemma}

In particular, for a sufficiently small $\diffeps \leq \epsilon^2 \cdot \poly(R^{-1}, L^{-1}, B^{-1})$, one recovers the same guarantee as in~\Cref{lemma:strict1} but with a slightly inferior constant.

\begin{proof}[Proof of~\Cref{lemma:extendedsemi}]
    We first note that $\vatil$ can be obtained as an approximate solution to the $\nicefrac{1}{\eta}$-strongly convex optimization problem
    \begin{equation}
        \label{eq:opt-reform}
        \argmin_{\vx \in \cX} \left\{ \langle \vx, F(\va) \rangle + \frac{1}{2 \eta} \| \vx - \va \|^2 \right\}.
    \end{equation}
    By~\eqref{eq:approx-fo}, we have
    \begin{equation}
    \label{eq:fo-opt}
    \left\langle F(\va) + \frac{1}{\eta} ( \vatil - \va ), \proj_\cX(\va) - \vatil \right\rangle \geq - B \diffeps \implies \langle F(\va), \va - \vatil \rangle \geq \frac{1}{\eta} \| \va - \vatil \|^2 - \diffeps \left( 2 B + \frac{2R}{\eta} \right),
    \end{equation}
    where the inequality follows because
    \begin{equation*}
        \left\langle F(\va) + \frac{1}{\eta} (\vatil - \va), \proj_\cX(\va) - \va \right\rangle \leq \left\| F(\va) + \frac{1}{\eta} (\vatil - \va) \right\| \| \proj_\cX(\va) - \va \| \leq \diffeps \left( B + \frac{2R}{\eta} \right).
    \end{equation*}
    Moreover, for any MVI solution $\vx \in \cX$,
    \begin{align}
        \langle F(\vatil), \va - \vx \rangle &= \langle F(\vatil), \vatil - \vx \rangle + \langle F(\vatil), \va - \vatil \rangle \notag \\
        &\geq \langle F(\proj_\cX(\vatil)), \proj_\cX(\vatil) - \vx \rangle + \langle F(\vatil), \va - \vatil \rangle - \diffeps(B + 2LR) \\
        &\geq \langle F(\va), \va - \vatil \rangle + \langle F(\vatil) - F(\va), \va - \vatil \rangle - \diffeps(B + 2LR) \label{align:Mintya'} \\
        &\geq \frac{1}{\eta} \|\va - \vatil \|^2 - \| F(\vatil) - F(\va) \| \|\va - \vatil \| - \diffeps(3 B + 6 LR) \label{align:fo} \\
        &\geq \frac{1}{\eta} \|\va - \vatil \|^2 - L \|\va - \vatil \|^2 - \diffeps(3 B + 6 LR) \notag \\
        &\geq \frac{1}{2\eta} \| \va - \vatil \|^2 - \diffeps(3 B + 6 LR),\label{align:Lip-eta}
    \end{align}
where~\eqref{align:Mintya'} uses the fact that $\vx$ satisfies~\eqref{eq:MVI}; \eqref{align:fo} follows from~\eqref{eq:fo-opt} and Cauchy-Schwarz; and~\eqref{align:Lip-eta} uses that $F$ is $L$-Lipschitz continuous and $\eta \leq 1/(2L)$. Finally, using again~\eqref{eq:approx-fo}, we have that for any $\vx \in \cX$,
\begin{equation}
    \label{eq:fo-new}
    \left\langle F(\va) + \frac{1}{\eta} (\vatil - \va), \vx - \vatil \right\rangle \geq -B \diffeps \iff \langle F(\va), \vx - \va \rangle + \langle F(\va), \va - \vatil \rangle + \frac{1}{\eta} \langle \vatil - \va, \vx - \vatil \rangle \geq - B \diffeps.
\end{equation}
But $\va \in \cX^{+\diffeps}$ is not an $\epsilon$-SVI solution, which implies that there exists $\vx \in \cX$ such that $\langle F(\va), \vx - \va \rangle < - \epsilon$. So, continuing from~\eqref{eq:fo-new},
\begin{equation*}
    \epsilon < \langle F(\va), \va - \vatil \rangle + \frac{1}{\eta} \langle \vatil - \va, \vx - \vatil \rangle + B \diffeps \leq B \|\va - \vatil \| + \frac{2R}{\eta} \|\va - \vatil \| + B \diffeps,
\end{equation*}
which in turn yields
\begin{equation*}
    \|\va - \vatil \| \geq \left( B + 4 L R \right)^{-1} (\epsilon - B \diffeps).
\end{equation*}
Combining with~\eqref{align:Lip-eta}, we arrive at~\eqref{eq:strictbutnotc}.
\end{proof}

\section{Omitted proofs}
\label{appendix:proofs}
\lemBilinearHard*
\begin{proof}
    \NP-membership is trivial. For hardness, we reduce from MAX-2-SAT. It is known that there exists an absolute constant $\eps$ for which approximating the maximum fraction of satisfiable clauses of a 2-SAT instance to precision better than $\eps$ is \NP-complete~\citep{Haastad01:Some}. Given a formula $\phi$ with $n$ variables $x_1, \dots, x_n$ and $m$ clauses, introduce vectors $\vc = (\vc_0, \vc_1, \dots, \vc_m)$ and $\vx = (\vx_0, \vx_1, \dots, \vx_n)$, and let $f : [0, 1]^{m+1} \times [0, 1]^{n+1} \to [-1, 1]$ be defined by
    \[
        f(\vc, \vx) = \frac{1}{5} \left( 3\vc_0\vx_0 + \frac{1}{m} \sum_{j=1}^m \qty((\vc_0 - \vc_j) \ell_{j1}(\vx) + \vc_j\ell_{j2}(\vx)) \right),
    \]
    where $\ell_{j1}$ and $\ell_{j2}$ encode the two literals in clause $j$: a positive literal $x_i$ is encoded as $\vx_i$, and a negative literal as $\vx_0 - \vx_i$. The function $f$ is bilinear. Each summand has absolute value at most $2$, and $0 \leq 3\vc_0\vx_0 \leq 3$, so $f$ takes values in $[-1, 1]$.

    Bilinear optimization on a product of cubes always admits an integral optimum, so we can restrict our attention to $\vc \in \{0, 1\}^{m+1}$ and $\vx \in \{0, 1\}^{n+1}$. If either $\vc_0 = 0$ or $\vx_0 = 0$, the bonus term vanishes and $f(\vc, \vx) \leq 2/5$. On the other hand, when $\vc_0 = \vx_0 = 1$, every summand is nonnegative, so $f(\vc, \vx) \geq 3/5$. Thus, every integral optimum has $\vc_0 = \vx_0 = 1$.

    Fix these coordinates to $1$, and interpret $(\vx_1, \dots, \vx_n) \in \{0, 1\}^n$ as a truth assignment, with $\vx_i = 1$ if and only if variable $i$ is assigned {\sf True}. For this assignment, choose $\vc_j = 0$ if the first literal of clause $j$ is true, and $\vc_j = 1$ otherwise. This choice maximizes each summand: its value is $1$ if clause $j$ is satisfied and $0$ otherwise. Consequently, if $\operatorname{OPT}(\phi)$ denotes the maximum fraction of clauses that can be satisfied, then
    \[
        \max_{\vc, \vx} f(\vc, \vx) = \frac{3 + \operatorname{OPT}(\phi)}{5}.
    \]
    A MAX-2-SAT target $v$ therefore becomes the target $(3+v)/5$, and the constant approximation gap $\eps$ becomes $\eps/5$. This proves the claimed hardness.
\end{proof}


\end{document}